\documentclass[11pt,reqno]{amsart}
\usepackage[T1]{fontenc}
\usepackage{amsaddr}
\usepackage{lmodern}
\usepackage{mathtools,amssymb,amsthm}
\usepackage[margin=1in]{geometry}
\usepackage[expansion=false]{microtype}
\usepackage{enumitem,booktabs,array}
\usepackage[hidelinks]{hyperref}
\hypersetup{pdftitle={Low-degree filtrations in ordinary and Lagrangian Grassmannian cohomology},pdfauthor={},pdfsubject={Grassmannian cohomology, subalgebra filtrations, and prescribed bases}}
\numberwithin{equation}{section}
\newtheorem{theorem}{Theorem}[section]
\newtheorem{lemma}[theorem]{Lemma}
\newtheorem{proposition}[theorem]{Proposition}
\newtheorem{corollary}[theorem]{Corollary}
\theoremstyle{definition}

\newtheorem{example}[theorem]{Example}
\theoremstyle{remark}
\newtheorem{remark}[theorem]{Remark}
\newcommand{\Q}{\mathbb Q}
\newcommand{\C}{\mathbb C}

\newcommand{\Hilb}{\operatorname{Hilb}}
\newcommand{\Span}{\operatorname{span}}
\newcommand{\im}{\operatorname{im}}
\newcommand{\Ann}{\operatorname{Ann}}

\newcommand{\Ext}{\operatorname{Ext}}
\newcommand{\Fitt}{\operatorname{Fitt}}
\newcommand{\gb}[2]{\genfrac{[}{]}{0pt}{}{#1}{#2}_{q}}

\newcommand{\Gr}{\operatorname{Gr}}
\newcommand{\LG}{\operatorname{LG}}
\newcommand{\length}{\ell}
\newcommand{\Pf}{\operatorname{Pf}}
\newcommand{\qb}{\mathsf q}
\newcommand{\J}{\mathcal J}
\newcommand{\B}{\mathcal B}
\newcommand{\aideal}{\mathfrak a}
\newcommand{\bideal}{\mathfrak b}

\newcommand{\Lf}{\mathcal L}
\title[Grassmannian filtrations, cancellation, and tangency]
{Low-degree filtrations in ordinary and Lagrangian Grassmannian cohomology}

\author[Huan Xiong]{Huan Xiong}

\address{%
\normalfont\small
Institute for Advanced Study in Mathematics\\
Harbin Institute of Technology\\
Heilongjiang 150001, P. R. China\\
Email:
\href{mailto:huan.xiong.math@gmail.com}
{\texttt{huan.xiong.math@gmail.com}}%
}

\begin{document}
\begin{abstract}
We prove the Hilbert-series formulas conjectured by Reiner--Tudose and by the Polymath Jr. Grassmannian group for the subalgebras of ordinary and Lagrangian Grassmannian cohomology generated in bounded degree. We determine the finite coefficient images underlying the proofs, together with the kernels of the coefficient and cancellation maps. We also construct bases adapted to the auxiliary ideal filtrations and identify the cancellation layers, as modules, with shifts of finite complete-intersection algebras. First-order deformations give the derivative inclusions needed to recover the filtration increments.

We disprove the two ordinary basis conjectures of the Polymath group. For every $r\ge2$, the span of each prescribed family has codimension one in weight $3r+2$ of $H^{2*}(\Gr(2,3r+1);\Q)$; replacing one index gives a basis in each of these components. More generally, each prescribed two-row family has codimension at most one in every weight. Weights are half the cohomological degrees, and all coefficients are rational.
\end{abstract}
\maketitle

\section{Introduction}
\label{intro:section}

The ordinary Grassmannian $\Gr(k,N)$ parametrizes the $k$-dimensional
subspaces of $\C^N$. The Lagrangian Grassmannian $\LG(n,\C^{2n})$
parametrizes the $n$-dimensional subspaces on which a fixed
nondegenerate alternating form vanishes. We study their rational
cohomology rings, with each cohomological degree divided by two.
In either ring, let the $m$th filtered subalgebra be generated by
all homogeneous classes of weight at most $m$. The presentation by
characteristic classes gives the Hilbert series of the full ring,
but the smaller subalgebras require further calculation: relations
involving higher generators can impose relations among the lower ones.

Reiner and Tudose conjectured the ordinary filtration formula
in~\cite[Conjecture~1]{RT}. They also showed how it would imply
high-degree generation statements related to Hoffman's
endomorphism theorem~\cite{Hoffman}. The problem appears as
OPAC-008~\cite{OPAC}. The 2020 Polymath Jr. Grassmannian group
interpreted the conjectured formula using $k$-conjugation,
proposed two bases, and formulated a Lagrangian analogue
\cite[Conjectures~1.2--1.4]{Polymath}. We prove both filtration
formulas. We also show that both prescribed ordinary bases fail,
and determine several two-row ranges in which they do form bases.

The longer report~\cite{REU} obtains a conditional Hilbert-series
bound from its proposed complete-function basis
\cite[Theorem~5.13 and Conjecture~5.12(1a)]{REU}. Its
Section~4.4 suggests combining coefficientwise bounds with the
known total Hilbert series. Our proofs follow that strategy, but
obtain the bounds from finite coefficient maps and differentiated
relations. In particular, they do not depend on either basis
conjecture. The counterexamples to the prescribed vectors below
are different from the index-set counterexample in
\cite[Section~5.5]{REU}.

The characteristic-class presentations originate in Borel's description
of homogeneous-space cohomology~\cite{Borel}; we use the Schubert and
symmetric-function conventions of~\cite{Fulton,Macdonald,KT}.
The $k$-Schur family in the basis conjectures is the one in~\cite{LM}.
Positive formulas for related graded $k$-Schur and Catalan functions
are obtained in~\cite{BMPS,BMP}; they do not imply independence of
this varying-level family after passage to the finite cohomology ring.
The Catalan moments used below are scalar moments, rather than these
Catalan symmetric functions.
We use the cancellation identities of~\cite{Stembridge,BR,Pragacz}.
The finite coefficient images, rather than the universal
cancellation identities, are the main algebraic calculation
in the proofs.

\subsection{Conventions and the ordinary filtration}
A partition $\lambda$ is a finite weakly decreasing sequence of positive
integers. Its size is $|\lambda|=\sum_i\lambda_i$, its length is
$\ell(\lambda)$, and its Young diagram has $\lambda_i$ cells in row $i$.
The transpose is denoted by $\lambda'$. We write
$\lambda\subseteq(b^a)$ when the diagram has at most $a$ rows and at
most $b$ columns, and $h_\lambda=\prod_i h_{\lambda_i}$.
A partition is strict if its positive parts are distinct. Missing parts
are read as zero, and we append zero parts when needed.
For the empty partition $\varnothing$, put $|\varnothing|=\ell(\varnothing)=0$
and read its largest part as zero; empty products are $1$.

We use ordinary commutative graded rings; in particular,
a variable of odd \emph{weight} commutes with every other variable.
For a graded vector space $V=\bigoplus_{t\in\mathbb Z}V_t$ with finite-dimensional components,
$\Hilb(V;q)=\sum_t\dim_\Q(V_t)q^t$, where $q$ records the weight. Our shift convention is
$M(a)_t=M_{t+a}$. A filtration quotient below means a quotient of graded
vector spaces, and is not assumed to be a quotient ring.

Fix integers $1\leq k\leq\ell$. Give the polynomial ring
$P=\Q[c_1,\ldots,c_k]$ the grading $\deg c_i=i$, and write
\[
 E(z)=1+c_1z+\cdots+c_kz^k,
 \qquad H(z)=E(z)^{-1}=\sum_{n\geq0}h_nz^n.
\]
Here $e_i$ denotes the $i$th elementary symmetric function, so
$c_i=(-1)^ie_i$. The $h_i$ are the complete symmetric functions.
Put
\[
 I=(h_{\ell+1},\ldots,h_{\ell+k}),\qquad R=P/I,
 \qquad P_m=\Q[c_1,\ldots,c_m],\qquad A_m=\im(P_m\longrightarrow R).
\]
The notation $P_m$ is used for $0\leq m\leq k$.
The standard presentation identifies $R\simeq H^{2*}(\Gr(k,k+\ell);\Q)$,
with the indicated weight grading; see~\cite{Fulton,Macdonald}.
Here $P_0=A_0=\Q$. For $m\geq k$, set $A_m=R$.
The algebra $A_m$ is exactly the subalgebra generated by all
homogeneous elements of weight at most $m$: a monomial of such
weight uses only coefficients $c_i$ with $i\leq m$.
Thus $R$ and $A_m$ are $R^{k,\ell}$ and $R^{k,\ell,m}$,
respectively, in the notation of Reiner and Tudose~\cite{RT}.
The identity $EH=1$ gives, for $n\leq k$,
$h_n=-c_n-\sum_{i=1}^{n-1}c_i h_{n-i}$. Consequently $A_m$ is also
generated by $h_1,\ldots,h_m$; the elementary and complete generators
define the same filtration, although their prescribed monomial bases
pose different questions.
Although $A_{m-1}$ need not be an ideal of $A_m$, the vector-space
quotient $A_m/A_{m-1}$ is naturally an $A_{m-1}$-module.
We use this structure in Corollary~\ref{conseq:module-sequence}.
Define
\[
 \gb nr=\prod_{i=1}^{r}\frac{1-q^{n-r+i}}{1-q^i}
 \quad(0\leq r\leq n),
\]
and set this expression to zero outside the indicated range.
It is the generating polynomial for partitions in an
$r$ by $n-r$ rectangle and satisfies the Gaussian Pascal identity
\[
 \gb nr=\gb{n-1}{r-1}+q^r\gb{n-1}{r}.
\]

In the next theorem, we prove the increment form of the formula in
\cite[Conjecture~1]{RT}. The cumulative endpoint cases $m=1$ and
$m=k$ were established in~\cite[Section~3]{RT}.

\begin{theorem}[The Reiner--Tudose filtration conjecture \cite{RT}]\label{main:theorem}
For $1\leq m\leq k\leq\ell$,
\begin{equation}\label{main:formula}
 \Hilb(A_m/A_{m-1};q)
 =q^m\gb\ell m
   \sum_{j=0}^{k-m}q^{j(\ell-m+1)}\gb{m+j-1}{j}.
\end{equation}
For $k<m\leq\ell$, both sides are zero, with the sum on the right
interpreted as empty.
\end{theorem}

Summing~\eqref{main:formula} gives the equivalent cumulative form
\begin{equation}\label{intro:cumulative}
 \Hilb(A_m;q)=1+\sum_{i=1}^m q^i\gb\ell i
       \sum_{j=0}^{k-i}q^{j(\ell-i+1)}\gb{i+j-1}{j},
 \qquad 0\leq m\leq k.
\end{equation}

\begin{example}\label{intro:ordinary-example}
For $\Gr(3,7)$, the full ring has dimension $35$, whereas
\[
 \dim A_0=1,\qquad \dim A_1=13,\qquad
 \dim A_2=31,\qquad \dim A_3=35.
\]
The first step is spanned by the powers of the hyperplane class
$h_1=-c_1$.
The last step has series
$\Hilb(A_3/A_2;q)=q^3+q^4+q^5+q^6$.
Thus the first two special classes generate every class of weight at
least seven, while four dimensions remain missing in lower weights.
\end{example}

The proof also identifies the quotients by relative Jacobian ideals
with finite coefficient images. A transformed Schur basis is adapted
to these ideals, meaning that each is spanned by a subset of the basis;
their nonterminal layers are shifted complete
intersections with specified $R$-module structures. Differentiation
relates the ideals to the filtration increments.

\subsection{The Lagrangian filtration}
Put
\begin{equation}\label{lg:eq:ring-intro}
 R_n=\Q[e_1,\ldots,e_n]\big/
 \left(e_i^2+2\sum_{a=1}^{n-i}(-1)^a e_{i+a}e_{i-a}
                 :1\le i\le n\right),\qquad \deg e_i=i,
\end{equation}
where $e_0=1$ and $e_j=0$ outside $0\le j\le n$.
This is the cohomology ring of $\LG(n,\C^{2n})$; it is the classical
specialization of the presentation in Kresch--Tamvakis~\cite[Theorem~1]{KT}.
Equivalently, for $E(z)=\sum_{i=0}^n e_i z^i$, the relations say
$E(z)E(-z)=1$. Define
\[
 \Lf_m=\Q[e_1,\ldots,e_m]\subseteq R_n\quad(0\le m\le n),
 \qquad \Lf_0=\Q.
\]
An even generator is a polynomial in the preceding generators, so
$\Lf_{2r}=\Lf_{2r-1}$ whenever $2r\le n$
\cite[discussion following Conjecture~1.4]{Polymath}.
We prove the following Polymath Lagrangian filtration conjecture \cite{Polymath} in the next result.

\begin{theorem}[The Polymath Lagrangian filtration conjecture \cite{Polymath}]\label{lg:thm:main}
For every $n\ge1$ and $0\le m\le n$,
\begin{equation}\label{lg:eq:main}
 \Hilb(\Lf_m;q)
 =1+\sum_{\substack{1\le i\le m\\i\text{ odd}}}
 q^i\sum_{j=0}^{n-i}q^{\binom{j+1}{2}}\gb{i+j}{i}.
\end{equation}
Consequently, for odd $m$,
\begin{equation}\label{lg:eq:increment-intro}
 \Hilb(\Lf_m/\Lf_{m-1};q)
 =q^m\sum_{j=0}^{n-m}q^{\binom{j+1}{2}}\gb{m+j}{m},
\end{equation}
and for positive even $m$ this increment is zero. Here all filtration quotients
are quotients of graded vector spaces, not quotient rings.
\end{theorem}

The endpoint cases and the full-ring summation identity appear
in~\cite[Sections~5.1--5.2 and Proposition~5.2]{Polymath}.
Corollary~\ref{lg:cor:structural} identifies transformed Schur
$Q$-ideals contained in the filtered pieces.

\subsection{The auxiliary quotients and the proofs}
For the ordinary proof, retain the convention $1\le k\le\ell$.
For $1\le s\le k$, the ordinary proof uses the ideals $\mathcal J_s$ generated by the
maximal minors of
\begin{equation}\label{main:hankel}
 (a_{L+i+j})_{0\le i<k,\ 0\le j<s},\qquad
 a_v=[z^v]H(z)^2,\quad L=\ell-k+1.
\end{equation}
Here $[z^v]$ means coefficient extraction, and the maximal minors
are the determinants of its $s$ by $s$ submatrices.
These are the ideals of the last $s$ Jacobian columns
(Proposition~\ref{gp:fitting}). For $d=s-1$, we prove
\begin{equation}\label{intro:polar-summary}
 \Hilb(R/\mathcal J_s;q)
 =\sum_{a=0}^d q^{a(a+\ell-k)}\gb{k}{a}\gb{\ell}{k-a}.
\end{equation}
In the Lagrangian proof, the corresponding ideals
$\mathcal J_{n,d}$ are generated by doubled-power-sum transforms
of Schur $Q$-functions indexed by strict partitions of length
larger than $d$; their exact definition is
\eqref{lg:eq:Q-ideal}. For $0\le d\le n$, we prove
\begin{equation}\label{intro:lg-quotient}
 \Hilb(R_n/\mathcal J_{n,d};q)
 =H_{n,d}(q)=\sum_{a=0}^d q^{\binom{a+1}{2}}\gb{n}{a}.
\end{equation}
Theorems~\ref{gp:theorem} and~\ref{lg:thm:parameter-image}
identify these quotients with finite coefficient images.

The two constructions are summarized in Table~\ref{intro:comparison}.
In its ordinary column, $e=d+\ell-k$ is the number of variables in
the second coefficient group. The layer weights are those in the
cancellation sequences of Theorems~\ref{gp:layer-sequence}
and~\ref{lg:thm:layers}.

\begin{table}[ht]
\centering
\small
\caption{The two finite coefficient constructions}
\label{intro:comparison}
\begin{tabular}{@{}>{\raggedright\arraybackslash}p{.23\linewidth}
>{\raggedright\arraybackslash}p{.34\linewidth}
>{\raggedright\arraybackslash}p{.34\linewidth}@{}}
\toprule
 & Ordinary & Lagrangian\\
\midrule
Auxiliary ideals & Relative Jacobian minors &
Doubled-power-sum transforms of Schur $Q$-functions\\
Cancellation & One common root, $d\to d-1$ &
One opposite pair, $d\to d-2$\\
Layer weight & $de$ & $\binom d2$\\
Tangent generators & All $c_m$, $1\le m\le k$ &
Odd $e_m$, $1\le m\le n$\\
\bottomrule
\end{tabular}
\end{table}

The upper bound comes from the positive operator that doubles
power sums in the symmetric-function algebra. For the lower
bound, we compute transition determinants between two finite
regular sequences and apply the determinant--colon comparison
of~\cite[Proposition~1.2]{SS}. A cancellation map reduces the
number of coefficient variables, and the two bounds agree.
The finite parameter rings can contain nilpotent elements, meaning
elements for which some positive power vanishes. We therefore
perform cancellation in a free extension obtained by adjoining a
root of a monic polynomial; no division by the cancellation product
or assumption of reducedness is used.

First-order deformations then give a derivative inclusion for the
relations in the retained generators. Theorem~\ref{common:rigidity}
expresses the discrepancy from each proposed increment as a sum
of Hilbert series with nonnegative coefficients. The known total
Hilbert series forces every discrepancy to vanish. The equality
case gives the differentiation sequences and the containment of
the auxiliary ideals in the filtered subalgebras.

\subsection{Contributions and earlier work}
\label{intro:attribution}
The Hilbert-series formulas were conjectured in~\cite{RT,Polymath},
and the strategy of matching bounds with the total Hilbert series
appears in~\cite[Section~4.4]{REU}. Hall and Schur $Q$ orthogonality,
Schur cancellation, and the determinant--colon lemma are standard
\cite{Macdonald,Stembridge,BR,Pragacz,SS}. Our contribution is the finite
coefficient calculation: we determine the images and their defining kernels,
identify the cancellation kernels, and establish the derivative inclusions
needed for the filtration formulas. These descriptions specify the
cohomology-ring actions on the complete-intersection layers and give
bases adapted to the auxiliary ideal filtrations. We also prove an
infinite family of prescribed-basis failures and the two-row basis
ranges stated below.

In the ordinary case, the filtration formula also gives the sharp
high-degree generation bound and the further generation assertions
of Reiner--Tudose
(Corollaries~\ref{conseq:stabilization}
and~\ref{conseq:RT-generation}).

The ordinary universal common-factor equations are given by
Pragacz~\cite[Theorem~1.5]{Pragacz96}. Their universal injection
does not by itself identify the contraction after the coefficient
window is imposed. The proof below establishes this finite
contraction by matching independent Hilbert-series bounds. This identifies
the actual kernel of the finite coefficient map and hence the quotient
algebra, including its nilpotents. Nonreduced quotients occur in the deformed symmetric-function
setting~\cite[Section~7, Theorem~10]{SV}.
Generalized power-sum and supersymmetric coefficient algebras are
studied in~\cite{ER,Sergeev}. The coefficient windows and Schur
truncations imposed here require a separate finite contraction calculation.

The auxiliary ideal bases are different from the prescribed bases
of the filtered subalgebras. Question~5.3 of~\cite{Polymath} asks for
a shifted analogue of their combinatorial Theorem~1.1 and a possible
$k$-Schur theory for Schur $Q$-functions. Our auxiliary ideal bases
do not provide such a construction.
Also, a layer's Gaussian Hilbert series does not identify its
coefficient algebra with a Grassmannian cohomology ring
(Examples~\ref{gp:not-grassmannian-product} and
\ref{lg:ex:non-Grassmannian-coefficients}). The finite conductor
statements include their endpoint exceptions. The Lagrangian
ideal calculation holds for every $d$, while the deformation
argument uses odd $m$.

\subsection{The prescribed ordinary bases}
We recall the $k$-conjugation convention of~\cite{LM,Polymath}.
The hook length of a cell is one plus the number of cells directly
to its right and directly below it. A partition is a $t$-core if
none of its hook lengths is divisible by $t$, and it is $k$-bounded
if its largest part is at most $k$. The standard bijection from
$(k+1)$-cores to $k$-bounded partitions counts, in each row,
the cells whose hook lengths are at most $k$. If $c_k(\lambda)$
is the core corresponding to $\lambda$, define
$\lambda^{\omega(k)}$ by applying these row counts to
$c_k(\lambda)'$. This is an involution on $k$-bounded partitions and preserves size,
since core transposition preserves the multiset of hook lengths.

For a positive integer $K$, we use the $K$-Schur functions
$s_\lambda^{(K)}$ of~\cite{LM}. They form a basis of
$\Q[h_1,\ldots,h_K]$, indexed by $K$-bounded partitions, with
$s_\varnothing^{(K)}=1$ and the $K$-Pieri rule
\[
 h_a s_\lambda^{(K)}
 =\sum_{\substack{\mu\text{ $K$-bounded}\\
                  \mu/\lambda\text{ horizontal }a\text{-strip}\\
                  \mu^{\omega(K)}/\lambda^{\omega(K)}
                        \text{ vertical }a\text{-strip}}}
                    s_\mu^{(K)}\qquad(1\le a\le K).
\]
Here a horizontal strip has at most one cell in each column, and
a vertical strip has at most one cell in each row; both strips
in the sum have $a$ cells. This is the convention in~\cite[Theorem~29]{LM}.

Let $\Lambda$ denote the ring of rational symmetric functions.
For the basis questions, let $1\le\ell\le k$ and $0\le m\le\ell$,
and use the parameter order of~\cite{Polymath}:
\[
 R^{\ell,k}=\Lambda/(e_i:i>\ell,\ h_j:j>k),\qquad
 R^{\ell,k,m}=\Q[h_1,\ldots,h_m]\subseteq R^{\ell,k}.
\]
Thus $R^{\ell,k}$ is the cohomology ring of $\Gr(\ell,k+\ell)$,
with $\ell$ rows and width $k$; this reverses the parameter names
used above. The indexing set proposed in~\cite{Polymath} is
\begin{equation}\label{lg:eq:candidate-index}
 \mathcal P^{\ell,k,m}
 =\{\lambda:\lambda_1\le m,\
               \lambda^{\omega(k)}\subseteq(k^\ell)\}.
\end{equation}
Its two proposed families are $h_\lambda$ and
$s_\lambda^{(\lambda_1)}$. The second family uses the level
$\lambda_1$ of each index. The empty index contributes $1$ in either
family; it does not require a level-zero Schur function. We obtain the following result.

\begin{theorem}[An infinite family of prescribed-basis failures]
\label{lg:thm:counter-intro}
For every integer $r\ge2$, put $k=3r-1$, $N=3r+2$, and
$t=\lfloor N/2\rfloor$. In weight $N$ of
$R^{2,k,2}=R^{2,k}=H^{2*}(\Gr(2,3r+1);\Q)$, each family prescribed
by Conjectures~1.2 and~1.3 of~\cite{Polymath} has $t-2$ candidates,
but their spans have dimension $t-3$. In particular, each misses
precisely one dimension of this component.

For $r=2$, the respective candidates are, in both cases,
$h_1^8$ and $h_2^4$, and in $R^{2,5}$ one has
\begin{equation}\label{lg:eq:counter-intro}
 3h_1^8-14h_2^4=0.
\end{equation}
Their span is one-dimensional in the two-dimensional weight-eight
component of $H^{2*}(\Gr(2,7);\Q)$.
\end{theorem}

The relation in Theorem~\ref{lg:thm:counter-intro} disproves
\cite[Conjectures~1.2 and~1.3]{Polymath} for the specified vectors.
Section~\ref{lg:sec:counter} gives a polynomial certificate for the
whole family and computes its exact rank defect. Replacing the
index $(1^N)$ by $(2,1^{N-2})$ gives bases in these components
(Corollary~\ref{counter:uniform-repair}). These replacements are
componentwise and are not asserted to give a basis in general.

For a two-row rectangle of width $k$, we call $d=N-k$ the excess
of weight $N$. Both original families are bases at every nonnegative
even excess $d\le k$ (see Proposition~\ref{counter:even-offset}). A coefficient
matrix gives the defect at arbitrary positive excess
(see Lemma~\ref{counter:projection-criterion}); it yields the answer at
excesses one and three (see Proposition~\ref{counter:odd-small}) and at
excess five (see Proposition~\ref{counter:excess-five}). For every excess,
any failure has codimension one (see Proposition~\ref{counter:defect-bound}).
In the six highest weights, both families are bases except in
$(k,N)=(5,8)$ and $(8,11)$ (see Proposition~\ref{counter:near-top}).
The proof for even excess
uses a Catalan moment determinant. Its positivity follows from
classical determinant integration and nonnegativity at real arguments of Schur polynomials indexed by
partitions with even parts
\cite{Forrester,HM}.

\subsection{Organization}
Section~\ref{common:section} gives the algebraic lemmas shared by the
two proofs. Sections~\ref{sf:section}--\ref{conseq:section} establish
the ordinary filtration and its consequences.
Sections~\ref{lg:sec:q}--\ref{lg:sec:consequences} treat the Lagrangian
filtration. Section~\ref{lg:sec:counter} then concerns the prescribed
two-row bases. Appendix~\ref{ca:section} records the standard
complete-intersection facts; Appendix~\ref{counter:appendix} contains
the longer two-row determinant calculations, and
Appendix~\ref{comp:section} describes the computational supplement.
Appendix~\ref{lg:appendix:coefficients} gives the coefficient-algebra
comparison used in Example~\ref{lg:ex:non-Grassmannian-coefficients}.

\section{Algebraic tools and a common filtration argument}
\label{common:section}

The polynomial rings and algebras below are nonnegatively graded,
with positive integer weights on polynomial variables. Graded modules
and their shifts may have integer degrees. For ideals $I,J$ of a ring $S$, write
$I:J=\{u\in S:uJ\subseteq I\}$ and $I:f=I:(f)$.
An Artinian ring has no infinite strictly descending chain of ideals;
the parameter rings used here are finite dimensional over $\Q$.
A ring is reduced when its only nilpotent element is zero.
An ideal has \emph{finite colength} if its quotient is finite dimensional
 over $\Q$. A sequence $f_1,\ldots,f_h$ is \emph{regular} if its ideal
is proper and multiplication by $f_i$ is injective modulo the preceding
terms for every $i$. A quotient of a polynomial ring by a regular sequence is a
\emph{complete intersection}. The empty sequence is regular.
For a formal series $F(z)$, write $F_j=[z^j]F(z)$ for its coefficient of $z^j$.
Rational powers of a series with constant term $1$ are understood through
the formal binomial expansion; polynomial degree in $z$ is written $\deg_z$.

\subsection{Positive compression}
\begin{lemma}[Positive compression]\label{common:positive}
Let $V$ be a finite-dimensional real inner-product space with an
orthogonal basis $(v_\lambda)_{\lambda\in\mathcal I}$. Let $U$ be spanned
by the basis vectors with indices in $\mathcal I_0\subseteq\mathcal I$,
and let $p:V\to U$ be the orthogonal projection. If $T:V\to V$ is
self-adjoint and positive definite, then the vectors
\[
 pT(v_\lambda),\qquad \lambda\in\mathcal I_0,
\]
form a basis of $U$. Every subfamily is therefore linearly independent.
The assertion applies degree by degree to a graded vector space with
finite-dimensional homogeneous components. Our later matrices are rational,
so their ranks are the same over $\Q$ and over $\mathbb R$.
\end{lemma}
\begin{proof}
For $0\ne u\in U$,
$\langle pT(u),u\rangle=\langle T(u),u\rangle>0$.
Consequently $pT|_U$ has zero kernel, and is an automorphism of the
finite-dimensional vector space $U$. Its values on the indicated basis
are precisely the asserted vectors.
\end{proof}

The Schur and Schur $Q$ quotients below identify with the spans of the
retained basis vectors. Their quotient maps are coefficient projections;
the positive operator need not preserve the quotient ideal.

\subsection{A finite complete-intersection comparison}
We use the finite graded determinant--colon comparison
of Simon--Strooker~\cite[Proposition~1.2]{SS}.

\begin{lemma}[Determinant--colon comparison]\label{common:ci-comparison}
Let $S=\Q[y_1,\ldots,y_h]$ be positively graded. Suppose
$\mathfrak a=(f_1,\ldots,f_h)$ and
$\mathfrak b=(j_1,\ldots,j_h)$ are homogeneous ideals of finite
colength. For $h>0$, assume that the $f_i$ and $j_i$ are homogeneous
of positive degrees. Put $\boldsymbol f=(f_1,\ldots,f_h)^{\mathsf T}$ and
$\boldsymbol j=(j_1,\ldots,j_h)^{\mathsf T}$. Suppose
$\boldsymbol f=M\boldsymbol j$, where
$M_{ab}$ is homogeneous of weight $\deg f_a-\deg j_b$ whenever it is
nonzero. Then
\begin{equation}\label{common:colon}
 \mathfrak a:\det M=\mathfrak b.
\end{equation}
Thus multiplication by $\det M$ embeds $S/\mathfrak b$ into
$S/\mathfrak a$, with the corresponding grading shift when the
determinant is homogeneous.
\end{lemma}
\begin{proof}
For $h=0$, the assertion is the identity $(0):1=(0)$ in $\Q$.
For $h>0$, both sequences are regular by Lemma~\ref{ca:ci}.
The matrix $M$ lifts the surjection $S/\mathfrak a\to S/\mathfrak b$
to a map of Koszul resolutions whose top component is $\det M$.
Applying internally graded $\operatorname{Hom}_S(-,S)$ therefore gives
multiplication by $\det M$ on the correspondingly shifted quotient rings,
\[
 \Ext^h_S(S/\mathfrak b,S)\longrightarrow
 \Ext^h_S(S/\mathfrak a,S).
\]
This map is injective: the exact sequence
$0\to\mathfrak b/\mathfrak a\to S/\mathfrak a\to S/\mathfrak b\to0$
and finite-length duality~\eqref{ca:finite-duality} give
$\Ext^{h-1}_S(\mathfrak b/\mathfrak a,S)=0$.
The adjugate identity makes the multiplication map well defined, and
its kernel is $(\mathfrak a:\det M)/\mathfrak b$. This proves the claim.
\end{proof}

\subsection{Tangency and the defect identity}
We use the nonnegative-error strategy of~\cite[Section~4.4]{REU}.
The following criterion expresses each defect, after a grading shift,
as the sum of two Hilbert series.

Let $R$ be a finite-dimensional graded commutative $\Q$-algebra with
$R_0=\Q$, generated by homogeneous elements $a_1,\ldots,a_r$
of positive weights $w_1,\ldots,w_r$.
Set
\[
 S_j=\Q[x_1,\ldots,x_j],\quad \deg x_i=w_i,\quad
 A_j=\Q[a_1,\ldots,a_j]\subseteq R,\quad
 K_j=\ker(S_j\longrightarrow A_j),
\]
with $S_0=A_0=\Q$. In particular $A_r=R$.
For each $1\le j\le r$, let $J_j$ be a homogeneous ideal of $R$,
and put
\[
 Q_j=\{f\in S_j:f(a_1,\ldots,a_j)\in J_j\},\quad
 W_j=\partial_{x_j}K_j,\quad H_j(q)=\Hilb(R/J_j;q).
\]
The notation $W_j$ denotes a vector-space image. It is not assumed to
be an ideal. In target weight $u$, its source is $(K_j)_{u+w_j}$,
including weights above the top weight of $R$; no truncation of $K_j$
is understood.

\begin{theorem}[Tangency--rigidity principle]\label{common:rigidity}
Suppose that $W_j\subseteq Q_j$ for every $j$. Let polynomials
$T_1(q),\ldots,T_r(q)$ satisfy
\begin{equation}\label{common:targets}
 T_j=q^{w_j}\left(H_j-\sum_{i>j}T_i\right),\qquad
 \sum_{j=1}^rT_j=\Hilb(R;q)-1.
\end{equation}
Then, for every $j$,
\begin{equation}\label{common:conclusions}
 \Hilb(A_j/A_{j-1};q)=T_j(q),\qquad
 W_j=Q_j,\qquad J_j\subseteq A_j.
\end{equation}
In particular $W_j$ is an ideal. Differentiation gives an exact sequence
of graded $A_{j-1}$-modules
\begin{equation}\label{common:module-sequence}
 0\longrightarrow A_{j-1}\longrightarrow A_j
 \xrightarrow{\ D_j\ }(A_j/J_j)(-w_j)\longrightarrow0,
 \qquad D_j([f])=[\partial_{x_j}f].
\end{equation}
No flatness assumption or ring structure on $A_j/A_{j-1}$ is needed.
\end{theorem}
\begin{proof}
Write $\Delta_j=\Hilb(A_j/A_{j-1};q)$.
Differentiation induces a well-defined linear map of degree $-w_j$
from $A_j=S_j/K_j$ to $S_j/W_j$. It is surjective because every
polynomial over $\Q$ has a polynomial antiderivative with respect
to $x_j$. Its kernel is
$A_{j-1}$: if $\partial_{x_j}f\in W_j$, there is $g\in K_j$ with
$\partial_{x_j}f=\partial_{x_j}g$, so $f-g\in S_{j-1}$.
The converse follows by choosing a representative independent of $x_j$.
Consequently
\begin{equation}\label{common:derivative}
 q^{-w_j}\Delta_j=\Hilb(S_j/W_j;q).
\end{equation}
Thus $S_j/W_j$ is finite-dimensional.

Define
\[
 \epsilon_j=\Delta_j-q^{w_j}
                 \left(H_j-\sum_{i>j}\Delta_i\right).
\]
Since $W_j\subseteq Q_j$, the natural vector-space sequence gives
\[
 \Hilb(S_j/W_j)=\Hilb(Q_j/W_j)+\Hilb(S_j/Q_j).
\]
The last quotient identifies with $A_j/(A_j\cap J_j)$.
Using $\Hilb(A_j)=\Hilb(R)-\sum_{i>j}\Delta_i$ and
$\Hilb(J_j)=\Hilb(R)-H_j$ therefore proves the exact identity
\begin{equation}\label{common:defect}
 q^{-w_j}\epsilon_j
  =\Hilb(Q_j/W_j;q)
   +\Hilb\bigl(J_j/(J_j\cap A_j);q\bigr).
\end{equation}
Both terms have nonnegative integer coefficients, so each
$\epsilon_j$ is a polynomial with nonnegative coefficients.

Let $\delta_j=\Delta_j-T_j$ and
$S^{\mathrm{err}}_j=\sum_{i\ge j}\delta_i$, with
$S^{\mathrm{err}}_{r+1}=0$. Subtracting~\eqref{common:targets}
from the definition of $\epsilon_j$ gives
\[
 S^{\mathrm{err}}_j=\epsilon_j+(1-q^{w_j})S^{\mathrm{err}}_{j+1}.
\]
The total-sum identity in~\eqref{common:targets}, together with
$A_r=R$ and $A_0=\Q$, gives $S^{\mathrm{err}}_1=0$.
Iteration yields
\begin{equation}\label{common:positive-identity}
 0=\sum_{j=1}^r\epsilon_j(q)\prod_{i<j}(1-q^{w_i}).
\end{equation}
Suppose that some $\epsilon_j$ is nonzero, and let $D$ be the
smallest exponent occurring in any nonzero $\epsilon_j$. Each factor
$\prod_{i<j}(1-q^{w_i})$ has constant term $1$, and all its other
terms have positive degree. Since no $\epsilon_j$ has a term of
smaller degree than $D$, the coefficient of $q^D$ on the right
of~\eqref{common:positive-identity} is
$\sum_j[q^D]\epsilon_j(q)>0$, a contradiction. Thus every $\epsilon_j$ is zero. The triangular recurrence gives
$\delta_j=0$ for all $j$, and~\eqref{common:defect} forces both
quotient spaces to vanish. Hence $W_j=Q_j$ and $J_j\subseteq A_j$;
in particular, $W_j$ is an ideal.

The same differentiation map now takes values in
$S_j/Q_j=A_j/J_j$. Changing a representative by an element of $K_j$
changes its derivative by an element of $Q_j$, so it is well defined.
Every element of $A_{j-1}$ has a representative independent of $x_j$;
the product rule thus gives $A_{j-1}$-linearity. The integration and
kernel arguments already given prove exactness. Under the convention
$M(a)_n=M_{n+a}$, the shift $(-w_j)$ makes the degree-$-w_j$
derivative homogeneous of degree zero in~\eqref{common:module-sequence}.
\end{proof}

\begin{example}[A derivative image that is not an ideal]\label{common:nonideal-derivative}
Let $S=\Q[x,y]$, with $\deg x=1$ and $\deg y=2$, and put
\[
 f=y^2+x^2y+x^4,\qquad K=(f,x^6),\qquad W=\partial_yK.
\]
Since $f$ is monic in $y$, the quotient $S/K$ is free on $1,y$
over $\Q[x]/(x^6)$; in particular it is an Artinian complete
intersection. In this example, $W_t$ denotes the weight-$t$
component of $W$. The equalities
\[
 K_4=\Q f,\qquad K_6=\Span\{x^2f,yf,x^6\}
\]
give
\[
 W_2=\Span\{x^2+2y\},\qquad
 W_4=\Span\{x^4+2x^2y,y^2\}.
\]
Since $y(x^2+2y)\notin W_4$, the space $W$ is not an ideal.
Thus the vector-space formulation in Theorem~\ref{common:rigidity}
is necessary even for Artinian complete intersections.
\end{example}

The derivative inclusion can be verified over the commuting dual numbers
$\mathcal T[\varepsilon]/(\varepsilon^2)$. Every element has a unique
expression $a+\varepsilon b$, with $a,b\in\mathcal T$, so its
$\varepsilon$-coefficient is defined even when $\mathcal T$ is nonreduced.

\begin{lemma}[A first-order witness for tangency]\label{common:witness}
In the preceding notation, suppose $\phi_j:R/J_j\hookrightarrow \mathcal T_j$
is an injective graded algebra map. Suppose further that there is a
homomorphism
$\widetilde\phi_j:R\to \mathcal T_j[\varepsilon]/(\varepsilon^2)$,
where $\deg\varepsilon=w_j$, reducing modulo $\varepsilon$ to the
quotient map followed by $\phi_j$, such that
\[
 \widetilde\phi_j(a_i)=\phi_j(\overline a_i)\ (i<j),\qquad
 \widetilde\phi_j(a_j)=\phi_j(\overline a_j)+\varepsilon.
\]
Then $\partial_{x_j}K_j\subseteq Q_j$.
\end{lemma}
\begin{proof}
For $f\in K_j$, the coefficient of $\varepsilon$ in
$\widetilde\phi_j(f(a_1,\ldots,a_j))=0$ is
$\phi_j(\overline{\partial_{x_j}f(a_1,\ldots,a_j)})$.
Injectivity gives the assertion. This argument is valid over an
arbitrary, possibly nonreduced, coefficient ring $\mathcal T_j$.
\end{proof}
\section{Symmetric functions and a positive pairing}
\label{sf:section}
We recall the conventions of Macdonald~\cite[Chapter~I]{Macdonald}.
Let $\Lambda=\Q[p_1,p_2,\ldots]$, with $\deg p_r=r$. For a partition
$\lambda$, write $p_\lambda=\prod_i p_{\lambda_i}$. The complete and
elementary symmetric functions are defined by
\[
 H(z)=\sum_{n\ge0}h_nz^n
 =\exp\left(\sum_{r\ge1}\frac{p_rz^r}{r}\right),\qquad
 \sum_{n\ge0}e_nz^n=H(-z)^{-1}.
\]
Here $h_0=e_0=1$, and $h_n=e_n=0$ for $n<0$.
For $\mu\subseteq\lambda$, the skew Schur function is
\[
 s_{\lambda/\mu}=\det(h_{\lambda_i-\mu_j-i+j})_{1\le i,j\le N},
 \qquad N\ge\length(\lambda).
\]
We append zero parts as needed and put $s_\lambda=s_{\lambda/\varnothing}$.
Thus the definition is independent of $N$.
An \emph{alphabet} is a finite list $X=(x_1,\ldots,x_N)$; specialization
at $X$ means $p_r\mapsto\sum_i x_i^r$. For $\length(\lambda)\le N$, the alternant formula is
\begin{equation}\label{sf:alternant}
 s_\lambda(X)=
 \frac{\det(x_j^{\lambda_i+N-i})_{1\le i,j\le N}}
      {\det(x_j^{N-i})_{1\le i,j\le N}}.
\end{equation}
If $\length(\lambda)>N$, then $s_\lambda(X)=0$. Otherwise the
quotient is a symmetric polynomial and extends to repeated values
of the $x_i$. The Schur functions form a homogeneous
basis of $\Lambda$; in $N$ variables, those with at most $N$ parts
form a basis of the symmetric polynomial ring.

\begin{lemma}[Hall pairing]\label{sf:hall}
The Hall pairing on $\Lambda$ is given by
\[
 \langle p_\lambda,p_\mu\rangle=\delta_{\lambda\mu}z_\lambda,
 \qquad z_\lambda=\prod_{r\ge1}r^{m_r(\lambda)}m_r(\lambda)!,
\]
where $m_r(\lambda)$ is the multiplicity of $r$ and
$\delta_{\lambda\mu}$ is the Kronecker delta. The Schur basis is
orthonormal. The graded algebra endomorphism $\theta_2:\Lambda\to\Lambda$
defined by $\theta_2(p_r)=2p_r$ is self-adjoint and
positive definite on every homogeneous component.
\end{lemma}
\begin{proof}
Schur orthonormality is the Hall pairing identity
\cite[Chapter~I, Section~4]{Macdonald}.
In the orthogonal power-sum basis, $\theta_2$ has eigenvalues
$2^{\length(\lambda)}>0$.
\end{proof}

For two alphabets $X,Y$, the notation $X-Y$ denotes the specialization
$p_r\mapsto p_r(X)-p_r(Y)$, rather than a set difference. Thus
\[
 h_n(X-Y)=[z^n]\frac{\prod_j(1-y_jz)}{\prod_i(1-x_i z)}.
\]
The skew Cauchy and dual Jacobi--Trudi identities give
\begin{align}
 s_\lambda(X-Y)
 &=\sum_{\nu\subseteq\lambda}(-1)^{|\lambda|-|\nu|}
       s_\nu(X)s_{\lambda'/\nu'}(Y),
       \label{sf:skew-expansion}\\
 s_\lambda(X-Y)&=(-1)^{|\lambda|}s_{\lambda'}(Y-X).
       \label{sf:transpose}
\end{align}
These are polynomial identities, including when an alphabet is empty;
see~\cite[Chapter~I, Sections~3 and~5]{Macdonald}.

\begin{lemma}[Rectangular Schur projection]\label{sf:rectangle}
The images of $s_\lambda$ for $\lambda\subseteq(\ell^k)$ form a graded
basis of $R$; every other Schur function maps to zero. In particular,
\begin{equation}\label{sf:total}
 \Hilb(R;q)=\gb{k+\ell}{k}.
\end{equation}
Under the identification of $R$ with the rectangular Schur span,
the quotient map from $\Lambda$ is orthogonal projection for the Hall pairing.
\end{lemma}
\begin{proof}
Specialization to $k$ variables kills partitions with more than $k$
parts. The recurrence $EH=1$ puts $h_j$ in $I$ for all $j>\ell$;
Jacobi--Trudi then kills every partition with first part greater than
$\ell$. The retained classes form the Schubert basis of $R$
\cite{Fulton,Macdonald}. Counting their weights gives~\eqref{sf:total},
and Hall orthonormality identifies the quotient map with orthogonal
projection.
\end{proof}

\section{Supersymmetric cancellation and a resultant colon ideal}
\label{res:section}
We recall the cancellation sequence from supersymmetric polynomial
and hook Schur theory~\cite{Stembridge,BR}, then compute the colon
ideal for the finite coefficient ring. All coefficient identities
are formal-series identities over $\Q$-algebras.

\subsection{Supersymmetric cancellation and its conductor}

For this subsection take $d,e\geq1$, with alphabets
$X=(x_1,\ldots,x_d)$ and $Y=(y_1,\ldots,y_e)$, and write
\[
 G(z)=\prod_{i=1}^d(1-x_i z),\qquad
 B(z)=\prod_{j=1}^e(1-y_j z).
\]
Let
\begin{gather*}
 \mathsf P=\Q[X]^{\mathfrak S_d}
              \otimes\Q[Y]^{\mathfrak S_e},
 \qquad h_n(X-Y)=[z^n]B(z)/G(z),\\
 \mathsf C_{d,e}=\Q[h_n(X-Y):n\geq1],
 \qquad
 \rho=\prod_{i=1}^d\prod_{j=1}^e(x_i-y_j).
\end{gather*}
For a subalgebra $C\subseteq T$, its \emph{conductor} in $T$ is
$\operatorname{Cond}(C,T)=\{f\in T:fT\subseteq C\}$, the largest ideal
of $T$ contained in $C$.
The notation $\mathfrak S_d$ denotes the symmetric group acting by
permutation of the $d$ variables. We call $\mathsf C_{d,e}$ the
\emph{supersymmetric coefficient algebra}.
The weighted degree of $\rho$ is $de$.
The cancellation sequence and conductor formula follow from
\cite{Stembridge,BR}.

\begin{lemma}[Supersymmetric cancellation and conductor]\label{res:conductor}
The cancellation map obtained by setting $x_d=y_e=t$ gives an exact
sequence
\begin{equation}\label{gp:universal-conductor}
 0\longrightarrow\rho\mathsf P
   \longrightarrow\mathsf C_{d,e}
   \longrightarrow\mathsf C_{d-1,e-1}\longrightarrow0.
\end{equation}
Its image is independent of $t$. In addition,
\[
 \{f\in\mathsf P:f\mathsf P\subseteq\mathsf C_{d,e}\}
       =\rho\mathsf P.
\]
If a variable group is empty, its symmetric polynomial ring is
$\Q$, and the same conventions apply to the smaller algebra.
\end{lemma}
\begin{proof}
We use the classical hook factorization~\cite{BR,Pragacz}. If
$\length(\alpha)\le d$ and $\beta_1\le e$, pad $\alpha$ with zeros
and put $\lambda=(e+\alpha_1,\ldots,e+\alpha_d,\beta_1,\beta_2,\ldots)$.
Then
\begin{equation}\label{gp:rectangle-factorization}
 s_\lambda(X-Y)=(-1)^{|\beta|}\rho\,s_\alpha(X)s_{\beta'}(Y).
\end{equation}
The same hook Schur identities give $s_\nu(X-Y)=0$ if
$\nu_{d+1}>e$; missing parts are read as zero. Since the products
$s_\alpha(X)s_{\beta'}(Y)$ form a basis of $\mathsf P$,
\eqref{gp:rectangle-factorization} gives $\rho\mathsf P\subseteq\mathsf C_{d,e}$.

A separately symmetric polynomial vanishing at $x_d=y_e$ is divisible
by $x_d-y_e$, and symmetry gives divisibility by every cross factor.
These distinct linear factors are relatively prime; their product is
$\rho$, and the quotient remains separately symmetric. Thus cancellation
on $\mathsf P$ has kernel $\rho\mathsf P$. On $\mathsf C_{d,e}$ its image
is $\mathsf C_{d-1,e-1}$, since the common factor cancels in $B/G$.
This proves~\eqref{gp:universal-conductor}.

If $f\mathsf P\subseteq\mathsf C_{d,e}$, its cancellation image is a
polynomial $c$ independent of $t$. Put $g_1=[z]G(z)$ and
$g_1^\circ=[z]\prod_{i<d}(1-x_i z)$. The image of $fg_1$ is
$c(g_1^\circ-t)$ and is also independent of $t$, so $c=0$.
Hence $f\in\rho\mathsf P$. The reverse inclusion follows because
$\rho\mathsf P$ is a $\mathsf P$-ideal contained in $\mathsf C_{d,e}$.
\end{proof}

\subsection{An explicit resultant colon ideal}

Let $d,e,m\geq1$, put $h=d+e$, and work in the polynomial ring
\[
 \mathsf P=\Q[g_1,\ldots,g_d,b_1,\ldots,b_e],
 \qquad \deg g_i=i,\quad\deg b_j=j.
\]
Set
\[
 G(z)=1+\sum_{i=1}^d g_i z^i,\qquad
 B(z)=1+\sum_{j=1}^e b_j z^j,\qquad
 D(z)=(G(z)B(z))^{-1/2}=\sum_{n\geq0}D_nz^n,
\]
\[
 U(z)=\frac{zG'(z)D(z)}{G(z)},\qquad
 V(z)=\frac{zB'(z)D(z)}{B(z)}.
\]
Primes denote differentiation with respect to $z$. Define
\[
 \mathfrak I=(D_{m+1},\ldots,D_{m+h}),\qquad
 \mathfrak J=(U_{m+1},\ldots,U_{m+d},
                     V_{m+1},\ldots,V_{m+e}).
\]
The coefficient index is the weight of every displayed generator.
Let $\rho=\prod_{i,j}(x_i-y_j)$, with $G=\prod_i(1-x_i z)$
and $B=\prod_j(1-y_j z)$ as above.

\begin{lemma}[The resultant colon ideal]\label{res:colon}
Both displayed sequences are regular, and
\begin{equation}\label{gp:resultant-colon}
 \mathfrak I:\rho=\mathfrak J,
 \qquad
 \Hilb\bigl(\mathsf P/(\mathfrak I:\rho);q\bigr)
       =\gb{m+d}{d}\gb{m+e}{e}.
\end{equation}

\end{lemma}
\begin{proof}
Write $G(z)B(z)=\sum_{i=0}^h q_i z^i$, with $q_0=1$.
The identity $(GB)D'+\tfrac12(GB)'D=0$ gives
\begin{equation}\label{gp:positive-recurrence}
 nD_n=-\sum_{i=1}^h\left(n-\frac i2\right)q_iD_{n-i}
 \qquad(n\geq1),
\end{equation}
where negative-index coefficients are zero. This recurrence
propagates the window defining $\mathfrak I$ to every $D_n$
with $n>m$. Explicitly, for $n>m+h$ all indices $n-i$ on the
right satisfy $n-i>m$; induction on $n$, starting with the $h$
vanishing coefficients, gives $D_n\in\mathfrak I$. At a common complex zero, therefore, $D$ becomes
a polynomial of degree at most $m$. The polynomial identity
$GBD^2=1$ forces $G=B=D=1$. Thus the only common zero is
the origin, and Lemma~\ref{ca:ci} proves regularity of
the $h$ generators of $\mathfrak I$.

The exact series identities
\begin{equation}\label{gp:UV-identities}
 GU=zG'D,\qquad BV=zB'D,\qquad U+V=-2zD'
\end{equation}
give a matrix expressing the generators of $\mathfrak I$ in
terms of those of $\mathfrak J$. The resulting inclusion proves
that $\mathfrak J$ has finite colength; its displayed generators
will then be regular by Lemma~\ref{ca:ci}.

\medskip\noindent\emph{The transition matrix and its determinant.}

Use the following column vectors of coefficient generators
\[
 \mathbf f=(D_{m+1},\ldots,D_{m+h})^{\mathsf T},\qquad
 \mathbf j=(U_{m+1},\ldots,U_{m+d},
                 V_{m+1},\ldots,V_{m+e})^{\mathsf T},
 \qquad h=d+e.
\]
We construct $M$ such that $\mathbf f=M\mathbf j$.

Put $\nu_t=m+t$, and let $\mathbf e_1,\ldots,\mathbf e_h$ be
the standard row vectors. For $1\leq t\leq h$, construct row
vectors $\mathbf U_t,\mathbf V_t,\mathbf D_t$ in this order:
\[
 \mathbf U_t=
 \begin{cases}
  \mathbf e_t,&t\leq d,\\
  -\displaystyle\sum_{i=1}^d g_i\mathbf U_{t-i}
    +\displaystyle\sum_{i=1}^d i g_i\mathbf D_{t-i},&t>d,
 \end{cases}
\]
\[
 \mathbf V_t=
 \begin{cases}
  \mathbf e_{d+t},&t\leq e,\\
  -\displaystyle\sum_{i=1}^e b_i\mathbf V_{t-i}
    +\displaystyle\sum_{i=1}^e i b_i\mathbf D_{t-i},&t>e,
 \end{cases}
\]
\[
 \mathbf D_t=-\frac{\mathbf U_t+\mathbf V_t}{2\nu_t}.
\]
All row vectors on the right of a recurrence have index strictly
less than $t$. Moreover $t-i\geq1$ whenever a recurrence is used,
and $\nu_t>0$. Thus these formulas define polynomial row vectors
over the coefficient ring. Let $M$ have rows
$\mathbf D_1,\ldots,\mathbf D_h$.
If a seed has weight $w$, its entry in row $t$ of $M$ has
weight $m+t-w$, by induction in the displayed recurrence.
Thus $M$ respects the generator degree shifts.

The exact identities
$GU=zG'D$, $BV=zB'D$, and
$U+V=-2zD'$ show by induction that evaluating
these row vectors on $\mathbf j$ gives
$U_{m+t},V_{m+t},D_{m+t}$, respectively.
Consequently
\begin{equation}\label{gp:transition-matrix}
 \mathbf f=M\mathbf j.
\end{equation}
This proves $\mathfrak I\subseteq\mathfrak J$. Since
$\mathfrak J$ has $h$ positive-degree generators and contains
$\mathfrak I$, its common zero set is the origin.
Lemma~\ref{ca:ci} therefore proves regularity of its
displayed generators as well.

We claim
\begin{equation}\label{gp:transition-determinant}
 \det M=\pm
       \frac{1}{2^h\prod_{t=1}^{h}(m+t)}\,\rho,
 \qquad
 \rho=\prod_{i=1}^d\prod_{j=1}^e(x_i-y_j).
\end{equation}
To calculate it, introduce $2h$ independent module coordinates
$u_1,\ldots,u_h,v_1,\ldots,v_h$, and put
\[
 \delta_t=-\frac{u_t+v_t}{2\nu_t}.
\]
For $d<t\leq h$ and $e<t\leq h$, respectively, define the
recurrence residuals
\begin{gather*}
 r_t^G=u_t+\sum_{i=1}^d g_i u_{t-i}
                   -\sum_{i=1}^d i g_i\delta_{t-i},\\
 r_t^B=v_t+\sum_{i=1}^e b_i v_{t-i}
                   -\sum_{i=1}^e i b_i\delta_{t-i}.
\end{gather*}
Let $\mathcal A$ be the linear map whose outputs, in order, are
\[
 u_1,\ldots,u_d,\ v_1,\ldots,v_e,\
 r_{d+1}^G,\ldots,r_h^G,\ r_{e+1}^B,\ldots,r_h^B.
\]
Its determinant is $\pm1$. Indeed, order both inputs and their
corresponding seed or residual outputs by the time index $t$.
Every residual is its own $u_t$ or $v_t$ plus terms of strictly
earlier index, so that matrix is unit lower triangular.
Returning to the displayed ordering changes only a sign.

Let $\mathcal B$ instead have outputs
\[
 \delta_1,\ldots,\delta_h,\
 r_{d+1}^G,\ldots,r_h^G,\ r_{e+1}^B,\ldots,r_h^B.
\]
The upper left block of $\mathcal B\mathcal A^{-1}$ is $M$,
because setting all residuals to zero is precisely the seed
recursion defining $M$. The other lower blocks are zero and
the identity. Hence $\det M=\pm\det\mathcal B$.

Change coordinates from $(u,v)$ to $(\delta,u)$. Its determinant
and the inverse coordinate formula are, respectively,
\[
 \pm\frac{1}{2^h\prod_t\nu_t},\qquad
 v_t=-2\nu_t\delta_t-u_t.
\]
In the residuals, the coefficients of the $u$ coordinates are
now the matrix
\[
 \begin{pmatrix} I_e&0\\0&-I_d\end{pmatrix}C,
\]
where the $h$ by $h$ matrix $C$ is explicitly
\[
 C_{rj}=g_{d+r-j}\quad(1\leq r\leq e),\qquad
 C_{e+r,j}=b_{e+r-j}\quad(1\leq r\leq d),
 \qquad 1\leq j\leq h.
\]
Here $g_0=b_0=1$, and coefficients outside their polynomial
degrees are zero. Since the first outputs of $\mathcal B$
are the $\delta$ coordinates themselves, block elimination gives
\[
 \det\mathcal B
    =\pm\frac{1}{2^h\prod_t\nu_t}
              \det C.
\]
The matrix $C$ is the Sylvester coefficient matrix, but its
determinant can be verified without invoking a resultant formula.
Put $\chi_G(w)=\sum_{i=0}^d g_iw^{d-i}$ and
$\chi_B(w)=\sum_{i=0}^e b_iw^{e-i}$. At a common root $t$,
the vector $(1,t,\ldots,t^{h-1})^{\mathsf T}$ is in the
kernel of $C$: its $G$-row $r$ evaluates to
$t^{r-1}\chi_G(t)$, and its $B$-row $r$ to
$t^{r-1}\chi_B(t)$. Thus each cross difference $x_i-y_j$
divides $\det C$ in the root polynomial ring, and their
product $\rho$ divides it.

Assign row weight $d+r$ to $G$-row $r$, row weight $e+r$
to $B$-row $r$, and column weight $j$ to column $j$.
Every determinant term then has total weight $de$, so
$\det C$ is a scalar multiple of $\rho$. That scalar is
nonzero and determined explicitly: set all $g_i=0$ for $i>0$.
The $G$-rows select the last $e$ columns. In the remaining
first $d$ columns the $B$-block is upper triangular with
diagonal $b_e$. Hence $\det C=(-1)^{de}b_e^d$ at this
specialization. All $x_i$ are now zero, so $\rho=b_e^d$.
Consequently $\det C=(-1)^{de}\rho$ identically. This
proves~\eqref{gp:transition-determinant}.
\medskip\noindent\emph{The colon ideal.}
Both sequences are regular, so Lemma~\ref{common:ci-comparison}
applied to~\eqref{gp:transition-matrix} gives
\begin{equation}\label{gp:transition-linkage}
 \mathfrak I:\det M=\mathfrak J.
\end{equation}
The scalar in~\eqref{gp:transition-determinant} is a nonzero
rational number, so $\mathfrak I:\rho=\mathfrak J$.
The degrees of the regular sequence defining $\mathfrak J$ give
\[
 \Hilb\bigl(\mathsf P/(\mathfrak I:\rho);q\bigr)
  =\frac{\prod_{i=1}^d(1-q^{m+i})\prod_{i=1}^e(1-q^{m+i})}
          {\prod_{i=1}^d(1-q^i)\prod_{i=1}^e(1-q^i)}
  =\gb{m+d}{d}\gb{m+e}{e}.
\]
\end{proof}

\section{Jacobian-tail quotients and coefficient parameterizations}
\label{gp:section}
Return to the Grassmannian ring $R$ of Section~\ref{intro:section}.
Use $a_n=[z^n]H(z)^2$, $L=\ell-k+1$, and the Jacobian-tail
ideals $\mathcal J_s$ of~\eqref{main:hankel}. We take $h_0=a_0=1$
and set coefficients with negative indices to zero.

\subsection{Relative differentials and coefficient parameters}
A basis is \emph{adapted} to a nested sequence of ideals if each ideal
is spanned by a subset of that basis. We construct such bases for the
auxiliary ideals in this section. For $0\leq d\leq k-1$, set
\[
 m=k-d,\qquad e=\ell-m=d+L-1,\qquad s=d+1.
\]
Let $\mathcal J_s$ be the ideal in $R$ generated by the maximal
minors of
\[
 M_s=(a_{L+i+j})_{\substack{0\leq i\leq k-1\\0\leq j\leq d}}.
\]
It is exactly the ideal of maximal minors of the Jacobian columns
for $c_m,\ldots,c_k$ of the generators of $I$. Indeed,
$\partial_{c_j}h_n=-a_{n-j}$; reversing the column order gives
the displayed matrix up to column signs, which do not change the ideal.

For a ring inclusion $A\subseteq R$, the module of relative differentials
$\Omega_{R/A}$ is generated over $R$ by symbols $dr$, subject to
$d(r+s)=dr+ds$, $d(rs)=r\,ds+s\,dr$, and $da=0$ for $a\in A$.
For a presentation $R^a\xrightarrow{M}R^b\to V\to0$, the zeroth
Fitting ideal $\Fitt_0^R(V)$ is generated by the $b$ by $b$ minors of $M$;
it is independent of the presentation~\cite[Tag~07Z6]{Stacks}.
The standard presentation of relative differentials gives the following
description~\cite[Tag~00RM]{Stacks}.

\begin{proposition}[Relative differentials]\label{gp:fitting}
For $1\leq m\leq k$, there is an exact graded presentation
\begin{equation}\label{gp:differential-presentation}
 \bigoplus_{i=1}^k R(-\ell-i)
 \xrightarrow{\ J^{\mathsf T}\ }
 \bigoplus_{j=m}^k R(-j)
 \longrightarrow\Omega_{R/A_{m-1}}\longrightarrow0,
 \qquad J_{ij}=\frac{\partial h_{\ell+i}}{\partial c_j}.
\end{equation}
Consequently
$\Fitt_0^R(\Omega_{R/A_{m-1}})=\mathcal J_{k-m+1}$.
\end{proposition}
\begin{proof}
The conormal presentation of $R=P/I$ has generators $dc_1,\ldots,dc_k$
and relations $dh_{\ell+1},\ldots,dh_{\ell+k}$. Passing to differentials
relative to $A_{m-1}=\Q[c_1,\ldots,c_{m-1}]\subseteq R$ kills the first
$m-1$ generators, giving~\eqref{gp:differential-presentation}. The matrix
entry has weight $\ell+i-j$, as required by the displayed shifts.
Its maximal minors generate the zeroth Fitting ideal. Since
$\partial_{c_j}h_n=-a_{n-j}$, reversing the tail columns identifies
that ideal with $\mathcal J_{k-m+1}$.
\end{proof}

\begin{theorem}[Jacobian-tail quotients]\label{gp:theorem}
With the preceding notation,
\begin{equation}\label{gp:polar-hilbert}
 \Hilb(R/\mathcal J_s;q)=
 \sum_{r=0}^{d}q^{r(r+L-1)}
       \gb{k}{r}
       \gb{\ell}{k-r}.
\end{equation}
This is the generating function of partitions $\lambda$ contained in
the rectangle $\ell^k$ and satisfying $\lambda_{d+1}\leq e$.
Moreover, put
\[
 G(z)=1+\sum_{i=1}^{d}g_i z^i,\qquad
 B(z)=1+\sum_{j=1}^{e}b_j z^j,\qquad
 \deg g_i=i,\quad\deg b_j=j.
\]
With independent coefficients $g_i,b_j$, set
\[
 D(z)=(G(z)B(z))^{-1/2},\qquad
 T=\Q[g_1,\ldots,g_d,b_1,\ldots,b_e]/
           (D_{m+1},\ldots,D_{m+d+e}).
\]
For an empty coefficient list the corresponding polynomial is $1$.
If $d=e=0$, the parameter ring is $\Q$ and its relation list is empty.
The coefficient assignments
\[
 E=GD,\qquad H=BD
\]
give an injection $R/\mathcal J_s\longrightarrow T$. Its image is
the image of the algebra generated by all coefficients of $B/G$.
\end{theorem}

\begin{remark}[Relation with universal elimination]
Pragacz's common-factor theorem~\cite[Theorem~1.5]{Pragacz96},
after reversing the monic polynomials to the constant-one convention,
with degrees $k,\ell$ and common-factor degree $m=k-d$,
has defining rectangle $(\ell-m+1)^{d+1}$.
The universal factorization $E=GD$, $H=BD$ gives $H/E=B/G$.
The additional relation $EH=1$ used here becomes $GBD^2=1$, so that
$D=(GB)^{-1/2}$ and $H/E=H^2$.
The rectangular Schur equations and common-factor parameterization
thus come from the universal construction. For the finite quotient
in Theorem~\ref{gp:theorem}, we prove the contraction equality for
the coefficient window by comparing graded dimensions. Injectivity
of the universal parameter map alone does not imply this equality
after specialization.
\end{remark}

We prove Theorem~\ref{gp:theorem} by matching upper and lower
bounds for the two Hilbert series.

\subsection{Shifted recurrence wedges and the Hall upper bound}

Work in $R$, with $0\leq d\leq k-1$.
The identity $EH^2=H$ gives the fixed coefficient recurrence
\[
 a_n+c_1a_{n-1}+\cdots+c_ka_{n-k}=0\qquad(n>\ell).
\]
Let $j=(a_L,\ldots,a_\ell)^{\mathsf T}\in R^k$ and let $S$ be
the companion matrix of this recurrence, explicitly
\[
 S(v_1,\ldots,v_k)^{\mathsf T}
 =\left(v_2,\ldots,v_k,-\sum_{i=1}^k c_{k+1-i}v_i\right)^{\mathsf T}.
\]
Then every shifted window is
\[
 j^{[r]}=(a_{L+r},\ldots,a_{\ell+r})^{\mathsf T}=S^rj
       \qquad(r\geq0).
\]
The new last index in $Sj$ is $\ell+1$, so all these uses of
the recurrence are within its valid range. The coordinates of
$j\wedge Sj\wedge\cdots\wedge S^dj$ generate $\mathcal J_s$.

On $(R^k)^{\otimes s}$ let $S_i$ act as $S$ in factor
$i$. These operators commute. For strictly increasing
nonnegative integers $r_0<\cdots<r_d$, alternant division is
the polynomial identity
\[
 \det(X_i^{r_j})_{\substack{1\leq i\leq s\\0\leq j\leq d}}
    =s_\mu(X_1,\ldots,X_s)
       \det(X_i^j)_{\substack{1\leq i\leq s\\0\leq j\leq d}},
\]
where
$\mu=(r_d-d,r_{d-1}-(d-1),\ldots,r_0)$.
Substitute $S_i$ for $X_i$ and apply to $j^{\otimes s}$.
A symmetric polynomial in the $S_i$ commutes with the alternating
idempotent $s!^{-1}\sum_{\pi\in\mathfrak S_s}\operatorname{sgn}(\pi)\pi$,
where $\mathfrak S_s$ permutes the tensor factors and
$\operatorname{sgn}(\pi)$ is the sign of the permutation.
Since $R$ is a $\Q$-algebra, the image of this idempotent identifies
with $\bigwedge^sR^k$. We obtain
\[
 S^{r_0}j\wedge\cdots\wedge S^{r_d}j
   =s_\mu(S_1,\ldots,S_s)
       (j\wedge Sj\wedge\cdots\wedge S^dj).
\]
Repeated exponents give zero, and changing the order gives
only a sign. Every coordinate of every shifted wedge therefore
belongs to $\mathcal J_s$.

Suppose $\lambda\subseteq\ell^k$ but $\lambda_{d+1}>e$.
Apply Jacobi--Trudi, padded to size $k$, to the sequence $a_n$.
In its first $s=d+1$ rows, every index satisfies
\[
 \lambda_i-i+j\ \geq\
       \lambda_s-s+1\ \geq\ (e+1)-(d+1)+1=L.
\]
Laplace-expand along these rows. Every selected $s$ by $s$
minor, after transposition, is a coordinate of a wedge of
shifted windows: row $i$ starts at shift
$\lambda_i-i+1-L\geq0$, and its selected columns specify
components of that window. The row shifts are strictly decreasing,
so reversing their order changes the determinant only by a sign. The preceding
polynomial wedge argument proves
\begin{equation}\label{gp:outside-schur-in-ideal}
 \theta_2(s_\lambda)\ \in\mathcal J_s,\qquad
 \theta_2(p_r)=2p_r.
\end{equation}
Here the class is taken in $R$; indeed
$\theta_2(h_n)=[z^n]H^2=a_n$.

In each weight, Lemmas~\ref{sf:hall} and~\ref{common:positive},
applied to the rectangular Schur span, show that the images of
$\theta_2s_\lambda$ for $\lambda\subseteq\ell^k$ form a basis of $R$.
Here the quotient map is the coefficient projection of
Lemma~\ref{sf:rectangle}. The outside-hook subfamily is therefore
independent and lies in $\mathcal J_s$ by
\eqref{gp:outside-schur-in-ideal}. Subtracting its size gives
\begin{equation}\label{gp:hook-upper}
 \Hilb(R/\mathcal J_s;q)
    \ \leq\ \sum_{\substack{\lambda\subseteq\ell^k\\
                              \lambda_{d+1}\leq e}}q^{|\lambda|}
\end{equation}
coefficientwise in $q$.

\subsection{The parameter image and the matching lower bound}

Let $T$ be the parameter ring in Theorem~\ref{gp:theorem}, with
$D=(GB)^{-1/2}$. The coefficient recurrence proves
that $T$ is a finite complete intersection and that $D_n=0$
in $T$ for $n>m$. Therefore $GD,BD$ have degrees at most
$m+d=k,m+e=\ell$ and their product is $1$.
This gives a map $R\longrightarrow T$.

In $T$ the ratio $(BD)^2=B/G$ obeys the $d$-step recurrence
of $G$ at every coefficient index greater than $e$. At row
$i$ of $M_s$, the last index is $L+i+d=e+1+i>e$.
Consequently the last column is a fixed linear combination
of the first $d$ columns, with its own coefficient equal to
$1$. All maximal minors vanish. For $d=0$ the single column
is zero because $B$ has degree $e$.
Thus the map factors to a surjection onto its image
\[
 R/\mathcal J_s\longrightarrow\mathcal B_d\subseteq T.
\]
The image is exactly that of $\mathsf C_{d,e}$: coefficients
of $B/G=(BD)^2$ lie in the $R$-image, while coefficients of
$(B/G)^{-1/2}=GD$ lie in the coefficient algebra of $B/G$.

When $d=0$, this coefficient algebra is the entire polynomial
ring in $B$, and
\[
 \Hilb(\mathcal B_0;q)
    =\gb{\ell}{k}.
\]
For $e>0$ this follows from the regular coefficient window of
$B^{-1/2}$ in degrees $k+1,\ldots,k+e$; the same recurrence
and the polynomial identity $BD^2=1$ prove regularity with one factor
absent. If $e=0$, the ring is simply $\Q$.

For $d\geq1$, the universal conductor inclusion puts the
entire ideal $\rho T$ inside $\mathcal B_d$. Multiplication
by $\rho$, together with~\eqref{gp:resultant-colon}, gives
\begin{equation}\label{gp:resultant-piece}
 \Hilb(\rho T;q)
   =q^{de}\gb{m+d}{d}
             \gb{m+e}{e}
   =q^{de}\gb{k}{d}
             \gb{\ell}{e}.
\end{equation}
Cancel one common root by writing
\[
 G=(1-tz)G^\circ ,\qquad B=(1-tz)B^\circ ,\qquad
 D=\frac{D^\circ }{1-tz},\qquad D^\circ =(G^\circ B^\circ )^{-1/2}.
\]
Give $t$ weight one. The next parameter ring $T'$ has factor degrees $d-1,e-1$
and common-factor degree $m+1$; it kills every $D^\circ _n$ with
$n>m+1$. There is a well-defined homomorphism
\begin{equation}\label{gp:cancellation-specialization}
 T\longrightarrow T'[t]/(f_{m+1}),\qquad
 f_{m+1}=\sum_{j=0}^{m+1}t^jD^\circ _{m+1-j}.
\end{equation}
Indeed, all pulled-back coefficients are
$f_n=\sum_{j=0}^n t^jD^\circ _{n-j}$, and
$f_{n+1}=tf_n+D^\circ _{n+1}$ kills every required coefficient
once $f_{m+1}=0$. The defining polynomial is monic of degree
$m+1$ in $t$. Polynomial division by this monic polynomial gives
a unique remainder of degree at most $m$, over the ring $T'$ itself.
Thus the target is a free $T'$-module with basis
$1,t,\ldots,t^m$, and the map from $T'$ to the target is
injective. This argument does not require $T'$ to be a domain. It
also covers the case in which both remaining factor groups
are empty.

On the supersymmetric coefficient algebra, the ratio $B/G$
cancels its common factor. Its image in
\eqref{gp:cancellation-specialization} is exactly
$\mathcal B_{d-1}$ inside the embedded $T'$, and the resultant
maps to zero. Therefore there is a graded surjection
\[
 \mathcal B_d/(\rho T)\longrightarrow\mathcal B_{d-1}.
\]
Let $\sigma_d$ denote~\eqref{gp:cancellation-specialization}.
Its restriction gives the commuting square
\begin{equation}\label{gp:finite-cancellation-square}
\begin{array}{ccc}
 \mathcal B_d&\lhook\joinrel\longrightarrow&T\\
 {\scriptstyle\delta_d}\downarrow&&\downarrow{\scriptstyle\sigma_d}\\
 \mathcal B_{d-1}&\lhook\joinrel\longrightarrow&T'[t]/(f_{m+1}).
\end{array}
\end{equation}
The lower inclusion is through the constant copy of $T'$.
Here $\delta_d$ is surjective and $\rho T\subseteq\ker\delta_d$.
The matching Hilbert series below will give
$\ker\delta_d=\rho T$, as used in Theorem~\ref{gp:layer-sequence}.

Equation~\eqref{gp:resultant-piece} and the base
case give the coefficientwise lower bound
\begin{equation}\label{gp:hook-lower}
 \Hilb(\mathcal B_d;q)\ \geq\
 \sum_{r=0}^d q^{r(r+L-1)}
       \gb{k}{r}
       \gb{\ell}{k-r}.
\end{equation}
At each induction step $k=m+d$ and $\ell=m+e$ stay fixed.

The sum in~\eqref{gp:hook-lower} counts the partitions in
\eqref{gp:hook-upper}.
For a partition in $\ell^k$, let $r$ be the largest index
such that $\lambda_r\geq r+L-1$, and put $r=0$ if no such
index exists. The condition $\lambda_{d+1}\leq d+L-1=e$
is equivalent to $r\leq d$. For fixed $r$, remove the
$r$ by $(r+L-1)$ rectangle. The part to its right is an
arbitrary partition in $r$ by $(k-r)$, and the part below is
an arbitrary partition in $(k-r)$ by $(r+L-1)$. They are
independent, and their generating functions are
$\gb{k}{r}$ and
$\gb{\ell}{k-r}$.
The removed rectangle has weight $r(r+L-1)$, proving the count.

The graded surjection $R/\mathcal J_s\longrightarrow\mathcal B_d$,
the upper bound~\eqref{gp:hook-upper}, and the equal lower
bound~\eqref{gp:hook-lower} force equality in every weight.
They prove~\eqref{gp:polar-hilbert} and show that this surjection
is an isomorphism. Composing with $\mathcal B_d\subseteq T$
proves the claimed parameter-image injection. In addition,
the classes $\theta_2s_\lambda$ for rectangular $\lambda$
outside the $(d,e)$ hook form a vector-space basis of
$\mathcal J_s$, by their independence and the equality of dimensions.

This proves Theorem~\ref{gp:theorem}.

\subsection{A simultaneous basis and exact cancellation layers}

Write $\pi_\Lambda:\Lambda\to R$ for the rectangular quotient map.
For a partition $\lambda\subseteq\ell^k$, define
\[
 b_\lambda=\pi_\Lambda(\theta_2s_\lambda),\qquad
 r(\lambda)=\max\bigl(\{0\}\cup
       \{r\in\{1,\ldots,k\}:\lambda_r\geq r+L-1\}\bigr).
\]
Thus $r(\lambda)$ is the largest $r$ for which an $r$ by
$(r+L-1)$ rectangle fits inside $\lambda$; we call this its
shifted Durfee rectangle.

\begin{theorem}[A basis adapted to all Jacobian tails]
\label{gp:adapted-basis}
The elements $b_\lambda$, for $\lambda\subseteq\ell^k$, form a
graded basis of $R$. For every $1\leq s\leq k$,
\begin{equation}\label{gp:flag-basis}
 \mathcal J_s=\Span_\Q\{b_\lambda:r(\lambda)\geq s\}.
\end{equation}
Thus the classes with $r(\lambda)<s$ form a basis of
$R/\mathcal J_s$, and
\[
 R\supseteq\mathcal J_1\supseteq\mathcal J_2
 \supseteq\cdots\supseteq\mathcal J_k=R_{k\ell}.
\]
For $1\leq s<k$, the images of the $b_\lambda$ with
$r(\lambda)=s$ form a basis of $\mathcal J_s/\mathcal J_{s+1}$.
\end{theorem}
\begin{proof}
Lemmas~\ref{sf:hall}, \ref{sf:rectangle}, and~\ref{common:positive}
give the full basis $(b_\lambda)$ by positive compression. For $s=d+1$, the condition
$\lambda_s>e=s+L-2$ is equivalent to $r(\lambda)\ge s$.
These vectors belong to $\mathcal J_s$ by
\eqref{gp:outside-schur-in-ideal}; Theorem~\ref{gp:theorem} gives equality
of their number and the dimension of the ideal in every weight.
They are therefore a basis of $\mathcal J_s$. Complements and
differences of the nested subsets give the quotient bases.
Finally, $r(\lambda)\ge k$ holds only for the full rectangle,
so $\mathcal J_k=R_{k\ell}$.
\end{proof}

\begin{theorem}[Complete-intersection layers]\label{gp:layer-sequence}
Fix $1\leq d\leq k-1$, and put $m=k-d$ and $e=\ell-m$.
Use the coefficient polynomial ring $\mathsf P$, the sequences
$\mathfrak I,\mathfrak J$, and the resultant $\rho$ of
Lemma~\ref{res:colon}. There is an exact sequence of graded
$R$-modules
\begin{equation}\label{gp:exact-layer}
 0\longrightarrow(\mathsf P/\mathfrak J)(-de)
 \xrightarrow{\ \iota_d\ }R/\mathcal J_{d+1}
 \longrightarrow R/\mathcal J_d\longrightarrow0.
\end{equation}
The last map is the natural quotient map. Under the embedding
$R/\mathcal J_{d+1}\hookrightarrow T=\mathsf P/\mathfrak I$,
the first map is given by $\iota_d(\overline f)=\rho f$.
The $R$-action on $\mathsf P/\mathfrak J$ is induced by
$R\to T\to\mathsf P/\mathfrak J$. In particular,
\begin{equation}\label{gp:layer-module}
 \mathcal J_d/\mathcal J_{d+1}
 \simeq(\mathsf P/\mathfrak J)(-de),\qquad
 \Hilb(\mathcal J_d/\mathcal J_{d+1};q)
     =q^{de}\gb{k}{d}\gb{\ell}{e}.
\end{equation}
\end{theorem}
\begin{proof}
Let $\mathcal B_d\subseteq T$ be the parameter image from
Theorem~\ref{gp:theorem}. The universal conductor gives
$\rho T\subseteq\mathcal B_d$.
Multiplication by $\rho$ identifies
$(\mathsf P/\mathfrak J)(-de)$ with $\rho T$ because
$\mathfrak I:\rho=\mathfrak J$. This is an isomorphism of graded
$R$-modules for the stated action.

The cancellation construction
\eqref{gp:cancellation-specialization} gives a surjection
$\mathcal B_d/\rho T\to\mathcal B_{d-1}$.
By~\eqref{gp:polar-hilbert} and~\eqref{gp:resultant-piece},
its source has Hilbert series
\[
 \Hilb(\mathcal B_d;q)-\Hilb(\rho T;q)
 =\sum_{r=0}^{d-1}q^{r(r+L-1)}
             \gb{k}{r}\gb{\ell}{k-r}
 =\Hilb(\mathcal B_{d-1};q).
\]
It is therefore an isomorphism in every weight.
To identify this map with the natural quotient, observe that the
cancellation substitution gives
\[
 GD=(1-tz)G^\circ \frac{D^\circ }{1-tz}=G^\circ D^\circ ,\qquad
 BD=(1-tz)B^\circ \frac{D^\circ }{1-tz}=B^\circ D^\circ .
\]
These are exactly the coefficient assignments defining the smaller
parameter image. Hence the composite from $R$ to that image is
its original quotient map with kernel $\mathcal J_d$.
Transporting the exact sequence
$0\to\rho T\to\mathcal B_d\to\mathcal B_{d-1}\to0$
through the two parameter-image isomorphisms proves
\eqref{gp:exact-layer} and \eqref{gp:layer-module}.
\end{proof}

\begin{proposition}[The full finite conductor]\label{gp:finite-conductor}
For $0\le d\le k-1$, the parameter image of
Theorem~\ref{gp:theorem} has conductor
\begin{equation}\label{gp:finite-conductor-eq}
 \{u\in T:uT\subseteq\mathcal B_d\}=\rho T,
\end{equation}
where $\rho=1$ when $d=0$.
\end{proposition}
\begin{proof}
The inclusion from right to left follows because $\rho T$ is a
$T$-ideal contained in $\mathcal B_d$. For $d=0$, the image is
all of $T$. Suppose $d\ge1$, and let
$\psi:T\to T'[t]/(f_{m+1})$ be the cancellation homomorphism
of~\eqref{gp:cancellation-specialization}. Its target is free over
$T'$ with basis $1,t,\ldots,t^m$, and $m=k-d\ge1$.
By Theorem~\ref{gp:layer-sequence}, the restriction of $\psi$ to
$\mathcal B_d$ has image $\mathcal B_{d-1}\subseteq T'$ and kernel
$\rho T$.

If $u$ belongs to the conductor, then $u,ug_1\in\mathcal B_d$.
Write $\psi(u)=u'\in T'$. The coefficient substitution gives
$\psi(g_1)=g_1^\circ-t$, with $g_1^\circ=0$ if $d=1$. Thus
$\psi(ug_1)=u'g_1^\circ-u't$ must also lie in the constant summand $T'$.
The independence of $1,t$ over $T'$ forces $u'=0$, even if $T'$
is nonreduced. Hence $u\in\rho T$, proving the reverse inclusion.
\end{proof}

\begin{example}[A noncyclic coefficient module]\label{gp:noncyclic-layer}
For $R=H^{2*}(\Gr(2,4);\Q)$ and $m=d=e=1$, the coefficient
ring in~\eqref{gp:exact-layer} is
\[
 C=\Q[g,b]/\bigl(g(3g+b),b(g+3b)\bigr),
 \qquad \deg g=\deg b=1.
\]
It has basis $1,g,b,gb$. The specified $R$-action sends $c_1$
to $a=(g-b)/2$ and $c_2$ to $a^2/2=-gb/3$.
Its image is $\Q[a]/(a^3)$, whereas
$C/R_{>0}C\cong\Q[g]/(g^2)$ has dimension two.
Thus $C$ is not cyclic as an $R$-module, and the layer
$\mathcal J_1/\mathcal J_2\simeq C(-1)$ requires generators
of weights one and two. The coefficient ring in the layer theorem
need not be an algebra quotient of $R$.
\end{example}

\begin{example}[A coefficient algebra over any characteristic-zero field]
\label{gp:not-grassmannian-product}
For $m=2$ and $d=e=1$, the coefficient ring of
Theorem~\ref{gp:layer-sequence} is
\[
 C=\Q[g,b]/(5g^3+2g^2b+gb^2,\ g^2b+2gb^2+5b^3),
 \qquad\deg g=\deg b=1.
\]
Indeed the two displayed relations are $8U_3/3$ and $8V_3/3$.
Its Hilbert series is $(1+q+q^2)^2$, but over any field extension
of $\Q$ it is not isomorphic to
$H^{2*}(\mathbb P^2\times\mathbb P^2;\Q)$ after the same extension.

To prove this, $g^2b,gb^2$ form a basis of $C_3$. The cube of
$\alpha g+\beta b$ is zero precisely when
\[
 F=-2\alpha^3+15\alpha^2\beta-\beta^3=0,\qquad
 G=-\alpha^3+15\alpha\beta^2-2\beta^3=0.
\]
The identities
\[
 \begin{aligned}
 F-G&=-(\alpha-\beta)(\alpha^2-14\alpha\beta+\beta^2),\\
 F+G&=-3(\alpha+\beta)(\alpha^2-6\alpha\beta+\beta^2)
 \end{aligned}
\]
show that these equations have only the solution $\alpha=\beta=0$:
substitution of $\alpha=\beta$ or $\alpha=-\beta$ forces both to
vanish, and otherwise subtracting the quadratic factors gives
$8\alpha\beta=0$.
In contrast, either standard degree-one generator of
$\Q[s,t]/(s^3,t^3)$ is nonzero and has cube zero.
Thus the two graded algebras are not isomorphic over any such field.
An ungraded isomorphism would preserve the maximal ideal and induce
an isomorphism of its associated graded algebras. Both rings are standard graded, meaning that their degree-zero part
is the ground field and they are generated in degree one. For such a
ring, its grading agrees with $\operatorname{gr}_{\mathfrak m}C
=\bigoplus_{j\ge0}\mathfrak m^j/\mathfrak m^{j+1}$. Thus an
ungraded isomorphism is also impossible.
\end{example}

\begin{remark}
The complete intersection in \eqref{gp:layer-module} is the explicit
quotient by the $U$- and $V$-coefficients in
Lemma~\ref{res:colon}. Its product of Gaussian Hilbert polynomials
does not by itself identify it with a tensor product of Grassmannian
rings. The conductor calculation uses the finite cancellation kernel and
the independence of $1,t$ in the monic extension, as in
Proposition~\ref{gp:finite-conductor}. The basis in
Theorem~\ref{gp:adapted-basis} describes the Jacobian ideals; it
does not give the prescribed bases of the subalgebras $A_m$.
\end{remark}

\section{Ordinary tangency and the filtration formula}
\label{tc:section}
Put $N=k\ell$. Retain $P_m$ and $A_m$ from
Section~\ref{intro:section}, and let $K_m=I\cap P_m$, so $A_m=P_m/K_m$.
Write
\begin{equation}\label{tc:actual-increments}
 \Delta_m(q)=\Hilb(A_m/A_{m-1};q)\qquad(1\le m\le k).
\end{equation}
The increments have nonnegative coefficients and satisfy
\begin{equation}\label{tc:actual-total}
 \sum_{m=1}^k\Delta_m=\Hilb(R;q)-1.
\end{equation}
For $s=k-m+1$, Theorem~\ref{gp:theorem} gives
\begin{equation}\label{tc:polar-series}
 \mathcal H_s(q):=\Hilb(R/\mathcal J_s;q)
 =\sum_{r=0}^{s-1}q^{r(r+L-1)}\gb{k}{r}\gb{\ell}{k-r},
 \qquad L=\ell-k+1.
\end{equation}
For $1\le m\le k$, let $\mathcal T_m$ be its parameter ring with
$d=k-m$ and $e=\ell-m$, and denote the corresponding series $D$ by
$\mathcal D$. We have an injection
\begin{equation}\label{tc:polar-injection}
 R/\mathcal J_s\hookrightarrow\mathcal T_m
\end{equation}
defined by
\begin{equation}\label{tc:coefficient-map}
 E=G\mathcal D,\qquad H=B\mathcal D.
\end{equation}
In $\mathcal T_m$, the series $\mathcal D$ has degree at most $m$.
The empty coefficient and relation conventions are those of
Theorem~\ref{gp:theorem}.

\subsection{The tangent inclusion}
\label{tc:tangent-subsection}

\begin{lemma}
\label{tc:tangent-inclusion}
Fix $1\le m\le k$. Let $\pi:P\to R$ be the quotient map and set
\[
 Q_m=P_m\cap\pi^{-1}(\mathcal J_{k-m+1}).
\]
Then, as graded vector spaces,
\begin{equation}
 \partial_{c_m}K_m\subseteq Q_m.
 \label{tc:derivative-inclusion}
\end{equation}
\end{lemma}

\begin{proof}
Give $\varepsilon$ weight $m$ and work in
$\mathcal T_m[\varepsilon]/(\varepsilon^2)$. Set
\begin{equation}\label{tc:dual-number-family}
 E_\varepsilon=G(\mathcal D+\varepsilon z^m),\qquad
 H_\varepsilon=B(\mathcal D-\varepsilon z^m).
\end{equation}
They have degrees at most $k$ and $\ell$, constant term one, and
\[
 E_\varepsilon H_\varepsilon
   =GB(\mathcal D^2-\varepsilon^2z^{2m})=1.
\]
Their coefficients thus define a homomorphism from $R$ whose reduction
is~\eqref{tc:coefficient-map}. Since $G(0)=1$, the coefficients below
$m$ are unchanged and the coefficient at $m$ increases by $\varepsilon$.
Lemma~\ref{common:witness}, with the injection
\eqref{tc:polar-injection}, now gives~\eqref{tc:derivative-inclusion}.
\end{proof}

\subsection{The ordinary defect identity}
\label{tc:defects-subsection}
In Theorem~\ref{common:rigidity}, take $w_m=m$ and
$J_m=\mathcal J_{k-m+1}$. Put $W_m=\partial_{c_m}K_m$,
using the full kernel $K_m$, and define
\begin{equation}\label{tc:defect-definition}
 \epsilon_m=\Delta_m-q^m\left(\mathcal H_{k-m+1}
                                     -\sum_{i>m}\Delta_i\right).
\end{equation}
Lemma~\ref{tc:tangent-inclusion} and~\eqref{common:defect} give
\begin{equation}\label{tc:defect-decomposition}
 q^{-m}\epsilon_m
 =\Hilb(Q_m/W_m;q)+\Hilb\bigl(
 \mathcal J_{k-m+1}/(\mathcal J_{k-m+1}\cap A_m);q\bigr).
\end{equation}
Both terms are Hilbert series of graded quotient spaces, so
$\epsilon_m$ has nonnegative coefficients. It remains to verify
\eqref{common:targets} for the proposed increments.

\subsection{The target recursion}
\label{tc:target-subsection}

Use the increment polynomials proposed by Reiner--Tudose
\cite[Conjecture~1 and Section~2]{RT}:
\begin{equation}
 T_m(q)=q^m\begin{bmatrix}\ell\\m\end{bmatrix}_q
 \sum_{j=0}^{k-m}q^{j(\ell-m+1)}
 \begin{bmatrix}m+j-1\\j\end{bmatrix}_q.
 \label{tc:target-polynomial}
\end{equation}

The total-sum identity~\eqref{tc:target-total} is
\cite[Proposition~8]{RT}. The following calculation derives it together with the triangular
recursion for the quotient series $\mathcal H_{k-m+1}$.

\begin{lemma}
\label{tc:target-recursion}
These polynomials satisfy
\begin{equation}
 T_m=q^m\left(\mathcal H_{k-m+1}-\sum_{i>m}T_i\right)
 \qquad(1\leq m\leq k),
 \label{tc:target-triangular}
\end{equation}
and
\begin{equation}
 \sum_{m=1}^kT_m=\operatorname{Hilb}R-1.
 \label{tc:target-total}
\end{equation}
\end{lemma}

\begin{proof}
Write $F_m=q^{-m}T_m$.  For $m<k$, put $b=\ell-m$,
$d=k-m$, and
\[
 C_j=\begin{bmatrix}m+j\\j\end{bmatrix}_q,
 \qquad C_{-1}=0.
\]
Gaussian Pascal gives
\[
 q^j\begin{bmatrix}m+j-1\\j\end{bmatrix}_q=C_j-C_{j-1}.
\]
Finite summation by parts consequently yields
\begin{equation}
 \sum_{j=0}^d q^{j(b+1)}
 \begin{bmatrix}m+j-1\\j\end{bmatrix}_q
 =(1-q^b)\sum_{j=0}^{d-1}q^{jb}C_j+q^{db}C_d.
 \label{tc:finite-gaussian-telescope}
\end{equation}
Using
\[
 (1-q^{m+1})\begin{bmatrix}\ell\\m+1\end{bmatrix}_q
 =(1-q^{\ell-m})\begin{bmatrix}\ell\\m\end{bmatrix}_q,
\]
equation \eqref{tc:finite-gaussian-telescope} becomes
\begin{equation}
 F_m-(1-q^{m+1})F_{m+1}
 =q^{(k-m)(\ell-m)}
   \begin{bmatrix}k\\m\end{bmatrix}_q
   \begin{bmatrix}\ell\\m\end{bmatrix}_q.
 \label{tc:F-difference}
\end{equation}
The right side is exactly
$\mathcal H_{k-m+1}-\mathcal H_{k-m}$ by
\eqref{tc:polar-series}.  The boundary is
$F_k=\begin{bmatrix}\ell\\k\end{bmatrix}_q=\mathcal H_1$.
Descending induction therefore gives
\[
 \mathcal H_{k-m+1}=F_m+\sum_{i>m}T_i,
\]
which is \eqref{tc:target-triangular}.

For the total, the hook counted by $\mathcal H_k$ excludes only
the full rectangle from the Schur basis of $R$.  Thus
\[
 \mathcal H_k=\operatorname{Hilb}R-q^N.
\]
Directly from \eqref{tc:target-polynomial},
\[
 T_1=q\frac{1-q^N}{1-q}.
\]
The case $m=1$ of the triangular identity now gives
\[
 \sum_{m=1}^kT_m
 =\mathcal H_k+(1-q^{-1})T_1
 =\operatorname{Hilb}R-1.
\]
Since $q$ divides $T_1$, the last identity is an identity of
polynomials.
\end{proof}

\subsection{Completion of the ordinary proof}
\label{tc:rigidity-subsection}
\begin{proof}[Proof of Theorem~\ref{main:theorem}]
Apply Theorem~\ref{common:rigidity} with weights $w_m=m$ and
$J_m=\mathcal J_{k-m+1}$. Theorem~\ref{gp:theorem} gives the
quotient series $H_m$, Lemma~\ref{tc:tangent-inclusion} gives
$W_m\subseteq Q_m$, and Lemma~\ref{tc:target-recursion} verifies
both identities in~\eqref{common:targets}. Thus $\Delta_m=T_m$,
$W_m=Q_m$, and $\mathcal J_{k-m+1}\subseteq A_m$, proving
\eqref{main:formula}. For $m>k$, both successive subalgebras are
$R$ and the displayed sum is empty.
\end{proof}

\section{Consequences for ordinary Grassmannians}
\label{conseq:section}
The equality case of~\eqref{tc:defect-decomposition} gives ideal
containments and differentiation sequences. The dimension and
generation statements then follow from Theorem~\ref{main:theorem};
the implications among Conjectures~1--4 were given
in~\cite[Section~2]{RT}.

\begin{corollary}[Derivative ideals and Jacobian-tail containment]
\label{tc:descent-as-consequence}
For every $1\le m\le k$, the derivative image is an ideal and satisfies
\[
 \partial_{c_m}(I\cap P_m)
   =P_m\cap\pi^{-1}(\mathcal J_{k-m+1}).
\]
Moreover, as graded subspaces of $R$,
\[
 \mathcal J_{k-m+1}\subseteq A_m.
\]
\end{corollary}

\begin{proof}
These are the equality-case conclusions of
Theorem~\ref{common:rigidity} applied in the proof of
Theorem~\ref{main:theorem}; equivalently, both quotients in
\eqref{tc:defect-decomposition} vanish.
\end{proof}

\begin{corollary}[Differentiation exact sequence]
\label{conseq:module-sequence}
For $1\leq m\leq k$, put $s=k-m+1$.
Differentiation induces an exact sequence of graded
$A_{m-1}$-modules
\begin{equation}\label{conseq:derivative-sequence}
 0\longrightarrow A_{m-1}\longrightarrow A_m
 \xrightarrow{\ D_m\ }(A_m/\mathcal J_s)(-m)
 \longrightarrow0,
 \qquad D_m([f])=[\partial_{c_m}f].
\end{equation}
In particular,
$A_m/A_{m-1}\simeq(A_m/\mathcal J_s)(-m)$ as graded
$A_{m-1}$-modules.
\end{corollary}
\begin{proof}
Apply~\eqref{common:module-sequence} with generator weight $m$ and
$J_m=\mathcal J_s$. Corollary~\ref{tc:descent-as-consequence} gives
$\mathcal J_s\subseteq A_m$ and the required derivative-image equality.
The resulting map is $D_m$, with the displayed shift $(-m)$.
\end{proof}

\begin{corollary}[Total dimensions]\label{conseq:dimensions}
For $1\leq m\leq k$,
\begin{equation}\label{conseq:increment-dimension}
 \dim_{\Q}(A_m/A_{m-1})=\binom{k}{m}\binom{\ell}{m}.
\end{equation}
Consequently, for $0\leq m\leq k$,
\begin{equation}\label{conseq:subalgebra-dimension}
 \dim_{\Q}A_m=\sum_{i=0}^m\binom{k}{i}\binom{\ell}{i}.
\end{equation}
\end{corollary}
\begin{proof}
A Gaussian polynomial has value $\binom{a}{b}$ at $q=1$.
Evaluating~\eqref{main:formula} therefore gives
\[
 \dim_{\Q}(A_m/A_{m-1})
 =\binom{\ell}{m}\sum_{j=0}^{k-m}\binom{m+j-1}{j}.
\]
The sum is $\binom{k}{m}$: with $r=m+j-1$, the identity
$\binom{r}{m-1}=\binom{r+1}{m}-\binom{r}{m}$ telescopes
from $r=m-1$ to $k-1$. Summing the increments and using
$\dim A_0=1$ proves~\eqref{conseq:subalgebra-dimension}.
\end{proof}

\begin{corollary}[Sharp high-degree stabilization]
\label{conseq:stabilization}
Put $N=k\ell$. For every $1\leq m\leq k$,
\[
 \deg\Hilb(A_m/A_{m-1};q)=N-m(m-1),
\]
and the coefficients of both the lowest term $q^m$ and the
highest term $q^{N-m(m-1)}$ are one. Moreover,
\begin{equation}\label{conseq:high-degree}
 (A_{m-1})_n=R_n\quad\text{for }n>N-m(m-1),
\end{equation}
whereas
\begin{equation}\label{conseq:sharpness}
 \dim_{\Q}(R/A_{m-1})_{N-m(m-1)}=1.
\end{equation}
\end{corollary}
\begin{proof}
The Gaussian polynomial $\gb{a}{b}$ is monic of degree
$b(a-b)$ and has constant coefficient one, by its rectangle
interpretation. The $j$th summand in~\eqref{main:formula},
including the common factor, has largest exponent
\begin{align*}
 m+m(\ell-m)+j(\ell-m+1)+j(m-1)
   &=m(\ell-m+1)+j\ell.
\end{align*}
These exponents strictly increase with $j$. At $j=k-m$ the
exponent is $N-m(m-1)$, with coefficient one. The term $j=0$
contributes the unique lowest term $q^m$, since
$\ell-m+1\geq1$.

As graded vector spaces,
$\Hilb(R/A_{m-1};q)=\sum_{i=m}^k T_i(q)$.
The largest exponents $N-i(i-1)$ strictly decrease for $i\geq1$.
Thus this sum has degree $N-m(m-1)$ and leading coefficient one,
which proves both assertions.
\end{proof}

In particular,~\eqref{conseq:high-degree} gives the high-degree
generation bound formulated in~\cite[Conjecture~2]{RT}, and
\eqref{conseq:sharpness} shows that the bound cannot be lowered.

\begin{corollary}[The further generation assertions of Reiner--Tudose]
\label{conseq:RT-generation}
Set $N=k\ell$ and write $e_i=(-1)^ic_i$ for the elementary classes.
Then
\begin{equation}\label{conseq:RT3}
 e_m e_1^{\,N-m^2+1}\in A_{m-1}\qquad(1\leq m\leq k).
\end{equation}
For $3\leq m\leq k$ one also has
\begin{equation}\label{conseq:RT4}
 e_m e_1^{\,N-2m}\in A_{m-1}.
\end{equation}
Thus Conjectures~2, 3, and~4 of~\cite{RT} follow from
Theorem~\ref{main:theorem}.
\end{corollary}
\begin{proof}
Both exponents are nonnegative in their stated ranges because
$k\ell\geq k^2$.
The weight in \eqref{conseq:RT3} is
$N-m^2+m+1=N-m(m-1)+1$, which lies strictly above the threshold
of Corollary~\ref{conseq:stabilization}.
The weight in \eqref{conseq:RT4} is $N-m$.
For $m\geq3$, the inequality $m(m-1)>m$ again puts this weight
strictly above that threshold. Corollary~\ref{conseq:stabilization}
also gives Conjecture~2 itself, with its optimal endpoint.
\end{proof}

The application of these generation statements to Hoffman's
endomorphism theorem~\cite{Hoffman} appears
in~\cite[Sections~4--5]{RT}.

\begin{remark}[The boundary increments]\label{conseq:boundary}
At the two ends of the filtration, the formula becomes
\[
 T_1(q)=q+q^2+\cdots+q^{k\ell},
 \qquad T_k(q)=q^k\gb{\ell}{k}.
\]
For $k=1$, these are the same polynomial and
$R\cong\Q[c_1]/(c_1^{\ell+1})$.
When $k=\ell=m$, the parameter construction has $d=e=0$,
$T=\Q$, and $\mathcal H_1=1$, with empty coefficient and
relation lists.
\end{remark}

\begin{remark}[A small example]\label{conseq:example}
For $k=\ell=3$, the three increments are
\begin{align*}
 T_1(q)&=q+q^2+\cdots+q^9,\\
 T_2(q)&=q^2(1+q+q^2)(1+q^2+q^3)\\
       &=q^2+q^3+2q^4+2q^5+2q^6+q^7,\\
 T_3(q)&=q^3.
\end{align*}
Thus $(\dim A_0,\dim A_1,\dim A_2,\dim A_3)=(1,10,19,20)$,
and the whole weight-$n$ component belongs to $A_1$ for $n>7$.
At weight $7$, the quotient $R/A_1$ is one-dimensional, as
predicted by Corollary~\ref{conseq:stabilization} with $m=2$.
\end{remark}

\begin{remark}[A proper Artinian parameter image]
\label{conseq:artinian-example}
Take $k=\ell=2$ and $d=e=m=1$. Write $G=1+gz$, $B=1+bz$.
Then $\rho=b-g$ and
\begin{align*}
 D_2&=\frac{3g^2+2gb+3b^2}{8},&
 D_3&=-\frac{5g^3+3g^2b+3gb^2+5b^3}{16},\\
 U_2&=-\frac{g(3g+b)}2,&
 V_2&=-\frac{b(g+3b)}2.
\end{align*}
The transition matrix of Lemma~\ref{res:colon} is
\[
 \begin{pmatrix}D_2\\D_3\end{pmatrix}
 =\begin{pmatrix}
 -1/4&-1/4\\(5g+b)/24&(g+5b)/24
 \end{pmatrix}
 \begin{pmatrix}U_2\\V_2\end{pmatrix},
 \qquad \det M=-\frac{\rho}{24}.
\]
Consequently
\[
 (D_2,D_3):\rho=(g(3g+b),\,b(g+3b)),\qquad
 \Hilb(\mathsf P/\mathfrak J;q)=(1+q)^2.
\]
The parameter ring $T=\Q[g,b]/(D_2,D_3)$ has Hilbert polynomial
$1+2q+2q^2+q^3$. Its coefficient image satisfies
\[
 \mathcal B_1=\Q\oplus\rho T\quad\text{as graded vector spaces},
 \qquad \Hilb(\mathcal B_1;q)=1+q+2q^2+q^3.
\]
Indeed, Theorem~\ref{gp:layer-sequence} identifies its quotient by
$\rho T$ with $\mathcal B_0=\Q$.
The images of the Grassmannian generators are explicitly
$c_1=(g-b)/2$ and $c_2=(b-g)(g+3b)/8$.
Thus $R/\mathcal J_2\simeq\mathcal B_1$ has dimension five,
whereas $T$ has dimension six.
\end{remark}

\begin{remark}[Characteristic zero and distinct basis questions]
The divisions in the coefficient recurrences, Hall positivity, and
integration in the newest generator are used over $\Q$. No integral or
positive-characteristic analogue is asserted. The ordinary transformed
Schur basis is adapted to the Jacobian ideals; it does not provide the
prescribed complete or $k$-Schur bases. Section~\ref{lg:sec:counter}
shows that those two stronger assertions are false.
\end{remark}

\section{Schur \texorpdfstring{$Q$}{Q}-functions and positive projections}
\label{lg:sec:q}

Let
\[
 \Gamma=\Q[p_1,p_3,p_5,\ldots],\qquad
 \mathcal Q(z)=\sum_{r\ge0}\qb_r z^r
 =\exp\left(2\sum_{r\text{ odd}}\frac{p_rz^r}{r}\right).
\]
Thus $\qb_0=1$ and $\mathcal Q(z)\mathcal Q(-z)=1$.
The Schur $Q$-functions, indexed by strict partitions, form a homogeneous
basis of $\Gamma$. We use $Q_\varnothing=1$ and the normalization
\begin{equation}\label{lg:eq:Q-pair}
 Q_{a,b}=\qb_a\qb_b+2\sum_{j=1}^b(-1)^j\qb_{a+j}\qb_{b-j}
 \qquad(a>b\ge0),
\end{equation}
extended by $Q_{b,a}=-Q_{a,b}$ and $Q_{a,a}=0$.
For a skew-symmetric $2s$ by $2s$ matrix $A$, its \emph{Pfaffian} is
\[
 \Pf(A)=\frac{1}{2^s s!}\sum_{\sigma\in\mathfrak S_{2s}}
 \operatorname{sgn}(\sigma)\prod_{j=1}^s A_{\sigma(2j-1),\sigma(2j)},
 \qquad\Pf(\varnothing)=1.
\]
For $\lambda$ of even length, set $Q_\lambda=\Pf(Q_{\lambda_i,\lambda_j})$;
for odd length, append a zero part before taking the Pfaffian.
For an alphabet $X=(x_1,\ldots,x_d)$, specialization means
\begin{equation}\label{lg:eq:Q-specialization}
 \sum_{r\ge0}\qb_r(X)z^r
   =\prod_{i=1}^d\frac{1+x_i z}{1-x_i z}.
\end{equation}
In particular $Q_\lambda(X)=0$ when $\length(\lambda)>d$.
These identities and normalizations are given in
\cite[Chapter~III, Section~8]{Macdonald}; see also
\cite[Theorems~4.1 and~5.1]{Okada} for the Pfaffian formulas and
the Cauchy identity.

If $\alpha$ is a partition with odd parts and $m_r(\alpha)$ is the
multiplicity of $r$, put
$z_\alpha=\prod_r r^{m_r(\alpha)}m_r(\alpha)!$.
The scalar product on $\Gamma$ is
\begin{equation}\label{lg:eq:scalar}
 \langle p_\alpha,p_\beta\rangle
  =\delta_{\alpha\beta}2^{-\length(\alpha)}z_\alpha,
 \qquad
 \langle Q_\lambda,Q_\mu\rangle
  =\delta_{\lambda\mu}2^{\length(\lambda)}.
\end{equation}
The second equality follows, equivalently, by comparing the power-sum
and Schur $Q$ expansions of the Cauchy kernel
$\prod_{i,j}(1+x_i y_j)/(1-x_i y_j)$.

Define the graded endomorphism $\Theta:\Gamma\to\Gamma$ by
$\Theta(p_r)=2p_r$ for odd $r$. Then
\begin{equation}\label{lg:eq:theta}
 \Theta\mathcal Q(z)=\mathcal Q(z)^2,
 \qquad
 \langle f,\Theta f\rangle>0\quad\text{for }0\ne f\in\Gamma.
\end{equation}
In fact, if $f=\sum_\alpha a_\alpha p_\alpha$ with rational coefficients,
then $\langle f,\Theta f\rangle=\sum_\alpha z_\alpha a_\alpha^2$.
In particular, $\Theta$ is self-adjoint and positive definite on each
homogeneous component.

\begin{lemma}\label{lg:lem:ambient-basis}
There is a surjective graded homomorphism
\[
 \pi_n:\Gamma\longrightarrow R_n,
 \qquad \pi_n(\qb_i)=e_i\ (i\le n),\quad
        \pi_n(\qb_i)=0\ (i>n).
\]
The classes $\pi_n(Q_\lambda)$ with $\lambda$ strict and $\lambda_1\le n$
are a basis of $R_n$. Moreover,
\begin{equation}\label{lg:eq:ambient-Hilb}
 \Hilb(R_n;q)=\prod_{i=1}^n(1+q^i).
\end{equation}
\end{lemma}
\begin{proof}
The relation $E(z)E(-z)=1$ gives the indicated map from the algebra of
coefficients of $\mathcal Q$. Equivalently, the odd coefficients of
$\log E(z)$ determine the images of the independent odd power sums.
The map is surjective because its image contains every generator $e_i$.

For a direct dimension calculation, choose an integer $C>n$ and give
the polynomial variable $e_i$ the auxiliary positive weight $Ci-i^2$.
Use any monomial order refining this weight. In the $i$th relation
in~\eqref{lg:eq:ring-intro}, the leading monomial is $e_i^2$:
the weight of $e_{i+a}e_{i-a}$ is lower by $2a^2$, also when $i-a=0$
and $e_0=1$. The leading monomials are pairwise relatively prime,
so the relations form a Gr\"obner basis: their leading monomials
generate the leading-monomial ideal of all relations. The squarefree
monomials, whose exponents are all $0$ or $1$, are therefore a basis, proving~\eqref{lg:eq:ambient-Hilb}
and $\dim R_n=2^n$.

If $\lambda_1>n$, every entry in the first row of its Pfaffian is zero
under $\pi_n$, by~\eqref{lg:eq:Q-pair}. Thus all such $Q_\lambda$ vanish.
The remaining $Q$-classes span, and there are exactly $2^n$ of them,
one for each subset of $\{1,\ldots,n\}$. They are consequently a basis.
\end{proof}

For $0\le d\le n$ define
\begin{equation}\label{lg:eq:H}
 H_{n,d}(q)=\sum_{a=0}^d q^{\binom{a+1}{2}}\gb{n}{a}.
\end{equation}
This counts strict partitions with largest part at most $n$ and length
at most $d$: subtracting $(a,a-1,\ldots,1)$ from a strict partition of
length $a$ leaves a partition inside $((n-a)^a)$.
Write
\begin{equation}\label{lg:eq:Q-ideal}
 Z_\lambda=\pi_n(\Theta Q_\lambda),\qquad
 \J_{n,d}=(Z_\lambda:\lambda\text{ strict},\ \length(\lambda)>d)
       \subseteq R_n.
\end{equation}
The displayed definition uses all strict partitions, without a bound on
the largest part. This causes no finiteness difficulty because $R_n$ is
finite-dimensional.

\begin{lemma}[Positive-projection bound]\label{lg:lem:positive}
Coefficientwise in $q$,
\begin{equation}\label{lg:eq:upper}
 \Hilb(R_n/\J_{n,d};q)\le H_{n,d}(q).
\end{equation}
More precisely, the elements $Z_\lambda$ with
$\lambda_1\le n$ and $\length(\lambda)>d$ are linearly independent.
\end{lemma}
\begin{proof}
In each weight, apply Lemma~\ref{common:positive} to $\Theta$ and the
span of $Q_\lambda$ with $\lambda_1\le n$. The scalar product
\eqref{lg:eq:scalar} makes this basis orthogonal, and
\eqref{lg:eq:theta} gives positivity. Lemma~\ref{lg:lem:ambient-basis}
identifies the quotient map with the corresponding coefficient
projection. Thus all retained $Z_\lambda$ form a basis of $R_n$.
The subfamily of length greater than $d$ lies in $\J_{n,d}$ and is
independent. Counting it gives~\eqref{lg:eq:upper}.
\end{proof}

\begin{remark}\label{lg:rem:no-descent}
The definition $Z_\lambda=\pi_n(\Theta Q_\lambda)$ applies
$\Theta$ in $\Gamma$ before the quotient map. No endomorphism of
$R_n$ induced by $\Theta$ is used.
\end{remark}

\section{An opposite-root product and a colon identity}\label{lg:sec:colon}

We use opposite-pair cancellation to compute the finite colon ideal.

Let
\[
 S_d=\Q[g_1,\ldots,g_d],\quad \deg g_i=i,\quad
 G(z)=1+g_1z+\cdots+g_dz^d=\prod_{i=1}^d(1+x_i z).
\]
Here the last expression identifies $S_d$ with
$\Q[x_1,\ldots,x_d]^{\mathfrak S_d}$; it does not require the existence
of individual roots in any later quotient ring. Put
\begin{equation}\label{lg:eq:rho}
 \rho_d=\prod_{1\le i<j\le d}(x_i+x_j),\qquad
 \rho_0=\rho_1=1.
\end{equation}
It is homogeneous of degree $\binom d2$.
The next determinant is the staircase Schur polynomial
$s_{(d-1,d-2,\ldots,1)}$, by the dual Jacobi--Trudi identity;
its alternant gives the product $\rho_d$.
We use the normalization in
\cite[Chapter~I, Section~3]{Macdonald}.

\begin{lemma}[Hurwitz determinant; staircase Schur identity]\label{lg:lem:hurwitz}
Let $d\ge1$. With $g_0=1$ and $g_i=0$ outside $0\le i\le d$, one has
\begin{equation}\label{lg:eq:hurwitz}
 \det\big(g_{d+a-2j}\big)_{1\le a,j\le d-1}=\rho_d.
\end{equation}
For $d=1$ the determinant is the empty determinant $1$.
\end{lemma}
\begin{proof}
Each determinant term has degree $\binom d2$.
Write $\chi(w)=\sum_{i=0}^d g_iw^{d-i}=\prod_i(w+x_i)$.
If $x_1=-x_2=t\ne0$, the even and odd parts of $\chi$ both vanish at
$t$. Multiplying the matrix in~\eqref{lg:eq:hurwitz} by the nonzero vector
$(t^2,t^4,\ldots,t^{2d-2})^{\mathsf T}$ gives, in row $a$, $t^a$ times
the appropriate parity part of $\chi(t)$. Every row is therefore zero.
Polynomial continuity includes $t=0$. It follows that $x_1+x_2$
divides the determinant in the root polynomial ring. Symmetry gives the
same divisibility for every $x_i+x_j$, so their product divides it.
The degrees are equal, and the determinant is a constant multiple of
$\rho_d$.

To identify the constant and prove it is nonzero, set $x_d=0$.
Then $g_d=0$, and the first column has only its first entry
$g_{d-1}$ nonzero. Expansion in that column reduces the determinant to
$g_{d-1}$ times the determinant for $d-1$ variables. The product
$\rho_d$ has the same specialization
$(x_1\cdots x_{d-1})\rho_{d-1}=g_{d-1}\rho_{d-1}$.
Induction from $d=1$ proves~\eqref{lg:eq:hurwitz}.
\end{proof}

Let $C_d\subseteq S_d$ be the subalgebra generated by all coefficients
of $G(z)/G(-z)$. To compare this notation with the literature, put
$p_r(X)=\sum_{i=1}^d x_i^r$. The formal identity
\[
 \log\frac{G(z)}{G(-z)}
   =2\sum_{\substack{r\ge1\\r\text{ odd}}}\frac{p_r(X)}{r}z^r
\]
and the inverse exponential expansion imply, coefficient by
coefficient over $\Q$, that
\[
 C_d=\Q[p_1(X),p_3(X),p_5(X),\ldots]\subseteq S_d.
\]
Here the right-hand side denotes the subalgebra generated by the
specialized odd power sums, without any assertion of their algebraic
independence. This algebra is characterized by the
$Q$-cancellation property~\cite[Theorem~2.11(Q), p.~145]{Pragacz}:
for $d\ge2$, setting $(x_{d-1},x_d)=(t,-t)$ gives a polynomial
independent of $t$.
An odd-power-sum description is also given in
\cite[Theorem~1]{CT}. Our Schur $Q$ notation is that of
\cite[Chapter~III, Section~8]{Macdonald}.

The characterization gives the inclusion below, since
$\rho_ds$ vanishes under opposite-pair cancellation for $s\in S_d$.
The Schur $Q$ proof also determines the normalization.

\begin{lemma}[A conductor inclusion]\label{lg:lem:conductor}
For every $d\ge0$,
\begin{equation}\label{lg:eq:conductor}
 \rho_d S_d\subseteq C_d.
\end{equation}
For $d\ge2$, write $G^\circ(z)=\prod_{i=1}^{d-2}(1+x_i z)$.
Replacing $G(z)$ by $(1-t^2z^2)G^\circ(z)$ removes the pair
$t,-t$ from the coefficient ratio and sends $\rho_d$ to zero.
\end{lemma}
\begin{proof}
For $d=0,1$, one has $\rho_d=1$ and $C_d=S_d$.
For $d\ge2$ let $\mu$ be any partition of length at most $d$, padded
with zeros, and put
$\lambda=\mu+(d-1,d-2,\ldots,1,0)$, dropping a final zero.
This is strict and has length $d$ or $d-1$. The symmetrization formula
for Schur $Q$-functions gives
\begin{equation}\label{lg:eq:conductor-Schur}
 Q_\lambda(X)=2^{\length(\lambda)}\rho_d s_\mu(X).
\end{equation}
Here is the cancellation in that formula. For either of these two
lengths, the product of the numerator factors $x_i+x_j$ is exactly
$\rho_d$, and the symmetrization denominator is the Vandermonde
$\prod_{i<j}(x_i-x_j)$. The remaining alternating sum has exponents
$\mu+(d-1,\ldots,0)$ and is the ordinary alternant for $s_\mu$.
The factorial $(d-\length(\lambda))!$ equals $1$, leaving exactly the
factor $2^{\length(\lambda)}$ in~\eqref{lg:eq:conductor-Schur}.
Since $Q_\lambda(X)$ is in $C_d$ by~\eqref{lg:eq:Q-specialization}, and
the $s_\mu(X)$ form a basis of $S_d$, this proves~\eqref{lg:eq:conductor}.
Finally, the factor $1-t^2z^2$ cancels in $G(z)/G(-z)$, and the pair
$t,-t$ gives a zero factor in~\eqref{lg:eq:rho}. All these are polynomial
or formal-series identities before specialization.
\end{proof}

The cancellation sequence below is an elementary consequence of the classical
$Q$-cancellation characterization~\cite[Theorem~2.11(Q)]{Pragacz}.
The divisibility and inclusion used in the proof also appear in
\cite[Lemmas~3 and~4]{CT}. We include the argument and deduce the full
conductor equality from the cancellation kernel.

\begin{proposition}[Universal cancellation and its conductor]\label{lg:prop:universal-conductor}
For $d\ge2$, cancellation of an opposite pair induces an exact sequence
of graded vector spaces
\begin{equation}\label{lg:eq:universal-cancellation}
 0\longrightarrow\rho_dS_d\longrightarrow C_d
 \xrightarrow{\ \chi\ }C_{d-2}\longrightarrow0.
\end{equation}
In fact $\rho_dS_d$ is the full conductor of $C_d\subseteq S_d$:
\[
 \{f\in S_d:fS_d\subseteq C_d\}=\rho_dS_d.
\]
For $d=0,1$, one has $C_d=S_d$.
\end{proposition}
\begin{proof}
Realize $S_d$ as the symmetric polynomials in $x_1,\ldots,x_d$.
The substitution $(x_{d-1},x_d)=(t,-t)$ sends
$G(z)$ to $(1-t^2z^2)G^\circ(z)$. Its restriction to $C_d$ has image
exactly $C_{d-2}$, embedded as the polynomials independent of $t$,
since the factors from the opposite pair cancel in $G(z)/G(-z)$.

If $f\in\ker\chi$, then $f$ vanishes when $x_{d-1}=-x_d$, so
$x_{d-1}+x_d$ divides $f$ in the root polynomial ring. Symmetry gives
divisibility by every $x_i+x_j$, $i<j$. These distinct linear
polynomials are pairwise relatively prime, so their product $\rho_d$
divides $f$. The quotient is symmetric because both $f$ and $\rho_d$
are symmetric; it therefore belongs to $S_d$. The reverse inclusion
uses Lemma~\ref{lg:lem:conductor} and the vanishing of $\rho_d$ under
cancellation. This proves~\eqref{lg:eq:universal-cancellation}.

If $fS_d\subseteq C_d$, then $f\in C_d$, and let $c=\chi(f)$.
Write $g_2^\circ=[z^2]G^\circ(z)$. The image of $g_2$ under
cancellation is $g_2^\circ-t^2$, where $g_2^\circ=0$
when $d-2<2$. Since $fg_2\in C_d$, its image
$c(g_2^\circ-t^2)$ must be independent of $t$. The coefficient of $t^2$
forces $c=0$, so $f\in\rho_dS_d$ by the kernel calculation.
Conversely $\rho_dS_d$ is an ideal of $S_d$ contained in $C_d$, and
hence is contained in the conductor. The cases $d=0,1$ follow from
the first coefficient of $G(z)/G(-z)$ and characteristic zero.
\end{proof}

For the finite quotient, we construct the coefficient sequences and
the transition matrix required by Lemma~\ref{common:ci-comparison}.
Fix integers $M\ge1$ and $d\ge1$. Define
\begin{equation}\label{lg:eq:D-U}
 D(z)=(G(z)G(-z))^{-1/2}=\sum_{v\ge0}D_vz^v,
 \qquad U(z)=\frac{zG'(z)D(z)}{G(z)}=\sum_{v\ge0}U_vz^v.
\end{equation}
Here $G'(z)=dG(z)/dz$. All inverses and square roots are formal series
with the indicated constant term; their coefficients lie in $S_d$. In particular $D$ is
even. Put $\widehat M=2\lfloor M/2\rfloor$ and define
\begin{equation}\label{lg:eq:ab-def}
 \begin{aligned}
 \aideal_{M,d}
 &=\begin{cases}
 (D_{M+2},D_{M+4},\ldots,D_{M+2d-2},\ g_dD_M),&M\text{ even},\\
 (D_{M+1},D_{M+3},\ldots,D_{M+2d-1}),&M\text{ odd},
 \end{cases}\\
 \bideal_{M,d}&=(U_{M+1},U_{M+2},\ldots,U_{M+d}).
 \end{aligned}
\end{equation}
Each displayed list defining $\aideal_{M,d}$ has exactly $d$ entries.
For $M$ even and $d=1$, the coefficient list before $g_1D_M$ is
empty. The two parities use different coefficient windows.

\begin{lemma}[Finite coefficient windows]\label{lg:lem:finite}
For every $M\ge1$ and $d\ge1$, the ideal $\aideal_{M,d}$ has finite
colength and its displayed generators form a regular sequence. In
$S_d/\aideal_{M,d}$ one has
\begin{equation}\label{lg:eq:D-truncate}
 D_v=0\quad(v>\widehat M),\qquad
 \deg_z(GD)\le M+d-1.
\end{equation}
\end{lemma}
\begin{proof}
Write $G(z)G(-z)=\sum_{i=0}^d w_i z^{2i}$, where $w_0=1$ and
$w_d=(-1)^dg_d^2$. Logarithmic differentiation gives, for $a\ge1$,
\begin{equation}\label{lg:eq:D-recurrence}
 aD_{2a}=-\sum_{i=1}^{\min(d,a)}
                 \left(a-\frac i2\right)w_iD_{2(a-i)}.
\end{equation}
Set $r=\lfloor M/2\rfloor$. If $M$ is odd, the window kills
$D_{2a}$ for $r<a\le r+d$; the recurrence then kills all subsequent
even coefficients. If $M$ is even, it kills $D_{2a}$ for
$r<a<r+d$. At $a=r+d$ the only possibly surviving summand contains
$w_dD_M=(-1)^dg_d(g_dD_M)=0$; all later even coefficients again
vanish by induction. Odd coefficients were zero from the outset.

For odd $M$, this gives $\deg_z D\le M-1$, hence the asserted degree
bound for $GD$. For even $M$, it gives $\deg_z(GD)\le M+d$, and the
last coefficient is $g_dD_M=0$, proving the same bound.

At a common complex zero, $D$ is a polynomial satisfying
\[
 G(z)G(-z)D(z)^2=1\qquad\text{in }\C[z].
\]
Every factor is constant. Thus $G=D=1$, and the common zero is the origin.
Lemma~\ref{ca:finite-zero} proves finite colength. There are $d$
homogeneous positive-degree generators in $d$ variables, so
Lemma~\ref{ca:ci} proves regularity.
\end{proof}

\begin{theorem}[Opposite-root colon identity]\label{lg:thm:colon}
For every $M\ge1$ and $d\ge1$,
\begin{equation}\label{lg:eq:colon}
 \aideal_{M,d}:\rho_d=\bideal_{M,d},\qquad
 \Hilb(S_d/\bideal_{M,d};q)=\gb{M+d}{d}.
\end{equation}
In particular $\rho_d(S_d/\aideal_{M,d})$ has Hilbert series
$q^{\binom d2}\gb{M+d}{d}$.
\end{theorem}
\begin{proof}
The two coefficient identities used in both cases are
\begin{equation}\label{lg:eq:U-identities}
 GU=zG'D,\qquad U(z)+U(-z)=-2zD'(z).
\end{equation}
Thus $U_v=-vD_v$ for even $v$. We construct a matrix $A$ expressing
the ordered generators of $\aideal_{M,d}$ in terms of those of
$\bideal_{M,d}$ and compute its determinant in each parity.

\smallskip\noindent
\emph{Case 1: $M$ is even.}
Start with independent $u_1,\ldots,u_d$. Define $v_t$ for these
indices and then define the subsequent $u_t,v_t$ recursively by
\begin{equation}\label{lg:eq:u-recurrence}
 \begin{aligned}
 v_t&=\begin{cases}-u_t/(M+t),&t\text{ even},\\0,&t\text{ odd},\end{cases}\\
 u_t&=-\sum_{i=1}^d g_i u_{t-i}
          +\sum_{i=1}^d i g_i v_{t-i}
             \qquad(d<t\le2d-2).
 \end{aligned}
\end{equation}
Every index $t-i$ in the recursive line is positive and less than
$t$. Substitution of $u_t=U_{M+t}$ for $t\le d$ realizes the later
$u_t$ as $U_{M+t}$ and the $v_t$ as $D_{M+t}$, by
\eqref{lg:eq:U-identities}. In degree $M+d$, the same identity and
$U_M=-MD_M$ give
\begin{equation}\label{lg:eq:b-row}
 b:=\frac{u_d+\sum_{i=1}^{d-1}g_i u_{d-i}
                   -\sum_{i=1}^{d-1}i g_i v_{d-i}}{M+d}
       \ \longmapsto\ g_dD_M.
\end{equation}
Define $A$ by
\begin{equation}\label{lg:eq:transition-rows}
 (v_2,v_4,\ldots,v_{2d-2},b)^{\mathsf T}
       =A(u_1,\ldots,u_d)^{\mathsf T}.
\end{equation}
We claim
\begin{equation}\label{lg:eq:detA}
 \det A=\frac{(-1)^{\binom{d-1}{2}}\rho_d}
 {(M+d)\prod_{j=1}^{d-1}(M+2j)}.
\end{equation}
For $d=1$, $A=(1/(M+1))$, as asserted. For $d\ge2$, temporarily
treat $u_1,\ldots,u_{2d-2}$ as independent, and put
\[
 r_t=u_t+\sum_{i=1}^d g_i u_{t-i}
                 -\sum_{i=1}^d i g_i v_{t-i}
                   \quad(d<t\le2d-2).
\]
The map to $(u_1,\ldots,u_d,r_{d+1},\ldots,r_{2d-2})$ has determinant
$1$: each residual is $u_t$ plus earlier variables. In these
coordinates, the map to $(v_2,v_4,\ldots,v_{2d-2},b,r_{d+1},\ldots,
 r_{2d-2})$ has lower blocks $0,I$ and upper-left block $A$.
Its determinant is therefore $\det A$.

Alternatively reorder the original variables as the even ones
followed by the odd ones, and rescale the even ones to $v_{2j}$.
With $r=d-1$, the reordering has $r(r+1)/2$ inversions, and the
rescaling contributes $(-1)^r/\prod_{j=1}^{d-1}(M+2j)$.
Their combined sign is $(-1)^{\binom{d-1}{2}}$. With all $v_{2j}=0$,
remove $1/(M+d)$ from the row $b$. The remaining rows
$b,r_{d+1},\ldots,r_{2d-2}$, on the odd variables
$u_1,u_3,\ldots,u_{2d-3}$, have matrix
\[
 (g_{d+a-2j})_{1\le a,j\le d-1}.
\]
Lemma~\ref{lg:lem:hurwitz} identifies its determinant with $\rho_d$.
This proves \eqref{lg:eq:detA}, with its sign.

\smallskip\noindent
\emph{Case 2: $M$ is odd.}
Start with the same $d$ seeds, but define
\begin{equation}\label{lg:eq:odd-transition}
 \begin{aligned}
 v_t&=\begin{cases}-u_t/(M+t),&t\text{ odd},\\0,&t\text{ even},\end{cases}\\
 u_t&=-\sum_{i=1}^d g_i u_{t-i}
          +\sum_{i=1}^d i g_i v_{t-i}
                 \qquad(d<t\le2d-1).
 \end{aligned}
\end{equation}
Now $M+t$ is even precisely when $t$ is odd. Substitution of the
$U$-coefficients therefore identifies $v_t$ with $D_{M+t}$ as before.
Define $A$ by
\[
 (v_1,v_3,\ldots,v_{2d-1})^{\mathsf T}
       =A(u_1,\ldots,u_d)^{\mathsf T}.
\]
Then
\begin{equation}\label{lg:eq:detA-odd}
 \det A=\frac{(-1)^{\binom{d+1}{2}}\rho_d}
 {\prod_{j=1}^d(M+2j-1)}.
\end{equation}
Indeed, use $2d-1$ independent $u_t$ and the residuals
$r_{d+1},\ldots,r_{2d-1}$ given by the same residual formula.
The seed--residual coordinate change again has determinant $1$.
Reorder to the odd $u_t$ followed by the even $u_t$, and rescale
the odd ones to $v_t$. The permutation sign is
$(-1)^{d(d-1)/2}$ and the rescaling factor is
$(-1)^d/\prod_{j=1}^d(M+2j-1)$. When all odd $v_t$ are zero,
the residual rows have coefficients
$(g_{d+a-2j})_{1\le a,j\le d-1}$ on $u_2,u_4,\ldots,u_{2d-2}$.
The same Hurwitz determinant gives \eqref{lg:eq:detA-odd}.
For $d=1$, the residual block is empty and $A=(-1/(M+1))$.

\smallskip
In either case we have the polynomial identity
$\boldsymbol f=A\boldsymbol j$, with the displayed generator
orders from \eqref{lg:eq:ab-def}. Assigning $u_t$ weight $M+t$
shows that this comparison respects generator-degree shifts.
The sum of the degrees of the $f_i$ is
$dM+d^2$ in both cases, whereas that of the $j_i$ is
$dM+d(d+1)/2$. Their difference is $\deg\rho_d=\binom d2$.

Thus $\aideal_{M,d}\subseteq\bideal_{M,d}$. The latter ideal has
finite colength and $d$ positive-degree generators, so it too is a
complete intersection. Lemma~\ref{common:ci-comparison} and the two
nonzero determinant formulas give the colon identity. Finally,
\[
 \Hilb(S_d/\bideal_{M,d};q)
 =\frac{\prod_{t=1}^d(1-q^{M+t})}{\prod_{t=1}^d(1-q^t)}
 =\gb{M+d}{d}.
\]
Multiplication by $\rho_d$ has degree $\binom d2$, proving the
last assertion.
\end{proof}

\begin{example}[The first nontrivial elimination step]\label{lg:ex:d3}
For even $M$ and $d=3$, the seed variables are $u_1,u_2,u_3$, and the only additional
recurrence needed is
\[
 u_4=-g_3u_1-\frac{M+4}{M+2}g_2u_2-g_1u_3.
\]
The three rows in~\eqref{lg:eq:transition-rows} therefore give
\[
 A=\begin{pmatrix}
 0&-\dfrac1{M+2}&0\\[5pt]
 \dfrac{g_3}{M+4}&\dfrac{g_2}{M+2}&\dfrac{g_1}{M+4}\\[5pt]
 \dfrac{g_2}{M+3}&\dfrac{g_1}{M+2}&\dfrac1{M+3}
 \end{pmatrix},
 \qquad
 \det A=-\frac{g_1g_2-g_3}{(M+2)(M+3)(M+4)}.
\]
Since $g_1g_2-g_3=(x_1+x_2)(x_1+x_3)(x_2+x_3)$, this is the asserted
factor $\rho_3$. The relation $g_dD_M$ supplies the final row
of the transition matrix.
For $M=d=2$, replacing the window by $(D_4,D_6)$ does not impose
the required degree bound. Indeed,
\[
 D_4=\tfrac38g_1^4-\tfrac32g_1^2g_2+g_2^2,
 \qquad g_2D_2=\tfrac12g_1^2g_2-g_2^2.
\]
The second polynomial is not a multiple of the first in weight four,
while $D_6$ has weight six. Thus $g_2D_2\ne0$ modulo $(D_4,D_6)$:
the coefficient of $z^4$ in $GD$ would not vanish. The parameter
map for $n=3$ would then not have the required degree bound.
\end{example}

\section{Finite parameter images and their exact Hilbert series}
\label{lg:sec:parameters}

Fix $n\ge1$ and any $d$ with $0\le d\le n$. In this section
\begin{equation}\label{lg:eq:parameters}
 M=n-d+1\ge1.
\end{equation}
Both parities of $n-d$ are allowed. The tangency argument in
Section~\ref{lg:sec:rigidity} uses $d=n-m$ with $m$ odd, hence
$M$ even.
For $d\ge1$, put
\[
 T_{n,d}=S_d/\aideal_{M,d},
\]
and use the images of $G,D$ in this quotient without changing notation.
For $d=0$ put $T_{n,0}=\Q$ and $G=D=1$.
Lemma~\ref{lg:lem:finite} says that
\begin{equation}\label{lg:eq:E-parameter}
 E_*(z):=G(z)D(z)\in T_{n,d}[z],\qquad
 \deg_z E_*\le n,\qquad E_*(z)E_*(-z)=1.
\end{equation}
Thus it defines a graded algebra homomorphism
\begin{equation}\label{lg:eq:parameter-map}
 \phi_{n,d}:R_n\longrightarrow T_{n,d},\qquad
 e_i\longmapsto[z^i]E_*(z).
\end{equation}
Let $\B_{n,d}=\im\phi_{n,d}$.

\begin{lemma}\label{lg:lem:B-description}
The algebra $\B_{n,d}$ is also the image of $C_d$ in $T_{n,d}$.
One has
\begin{equation}\label{lg:eq:parameter-kills}
 \J_{n,d}\subseteq\ker\phi_{n,d},\qquad
 \rho_d T_{n,d}\subseteq\B_{n,d}.
\end{equation}
The second subspace is an ideal of $\B_{n,d}$.
\end{lemma}
\begin{proof}
Equation~\eqref{lg:eq:E-parameter} gives
\[
 E_*(z)^2=\frac{G(z)}{G(-z)}.
\]
The coefficient algebra of a constant-one series is the same as that
of its square over $\Q$: squaring gives one inclusion, and the formal
square-root expansion gives the other, coefficient by coefficient.
This proves the description of $\B_{n,d}$.
By~\eqref{lg:eq:theta}, evaluating $\Theta Q_\lambda$ under $\phi_{n,d}\pi_n$
is the same as evaluating $Q_\lambda$ on the alphabet of $d$ roots in
\eqref{lg:eq:Q-specialization}. It is zero when $\length(\lambda)>d$.
This proves the first inclusion in~\eqref{lg:eq:parameter-kills}.
The second follows by specializing Lemma~\ref{lg:lem:conductor}; it is an
ideal because $\B_{n,d}$ is a subalgebra of $T_{n,d}$.
\end{proof}

\begin{lemma}[Opposite-pair cancellation over the finite ring]
\label{lg:lem:cancellation}
If $d\ge2$, there is a surjective graded homomorphism
\begin{equation}\label{lg:eq:cancellation-surjection}
 \B_{n,d}/(\rho_dT_{n,d})\longrightarrow\B_{n,d-2}.
\end{equation}
No reducedness or flat specialization assumption is needed.
\end{lemma}
\begin{proof}
Put $d'=d-2$ and $M'=M+2$, so that $n=M'+d'-1$ as required.
Write $G^\circ,D^\circ$ for the series defining $T'=T_{n,d'}$.
When $d'=0$, interpret $G^\circ=D^\circ=1$ and $g^\circ_0=1$.
Let $r=\lfloor M/2\rfloor$ and write
\[
 D^\circ(z)=\sum_{a\ge0}d^\circ_a z^{2a},\qquad
 f_a(t)=\sum_{j=0}^a t^{2j}d^\circ_{a-j}.
\]
The polynomial $f_{r+1}(t)$ is monic of degree $2r+2$, since $d^\circ_0=1$.
Set
\[
 \widetilde T=T'[t]/(f_{r+1}(t)),\qquad \deg t=1.
\]
Division by this monic polynomial makes $\widetilde T$ a free
$T'$-module with basis $1,t,\ldots,t^{2r+1}$. In particular, the natural
map $T'\to\widetilde T$ is injective even when $T'$ has nilpotents.

Specialize the coefficients of $G$ by
\begin{equation}\label{lg:eq:pair-specialization}
 G(z)=(1-t^2z^2)G^\circ(z).
\end{equation}
Then $D(z)=D^\circ(z)/(1-t^2z^2)$, so $D_{2a}$ maps to $f_a(t)$.
The identity
\[
 f_{a+1}=t^2f_a+d^\circ_{a+1}
\]
and the truncation $d^\circ_a=0$ for $a>r+1$ show that $f_a=0$ for all
$a\ge r+1$ in $\widetilde T$. Thus every coefficient generator in either branch of
$\aideal_{M,d}$ vanishes. When $M$ is even, the additional last
generator also vanishes: since $g_d=-t^2g^\circ_{d'}$, one has
\[
 g_d f_r=-g^\circ_{d'}t^2f_r=g^\circ_{d'}d^\circ_{r+1}=0.
\]
In this even-$M$ case, the final equality is the top-coefficient relation in $T'$; for
$d'=0$ it instead follows from $D^\circ=1$ and $r+1>0$.
Consequently~\eqref{lg:eq:pair-specialization} induces a homomorphism
$T_{n,d}\to\widetilde T$.

Under this map the series $GD$ becomes $G^\circ D^\circ$, independently of $t$.
The image of $\B_{n,d}$ is therefore precisely the embedded algebra
$\B_{n,d'}\subseteq T'\subseteq\widetilde T$.
Moreover $\rho_d$ maps to zero by Lemma~\ref{lg:lem:conductor}.
Restriction and passage to the quotient give
\eqref{lg:eq:cancellation-surjection}.
Writing $\sigma_{n,d}:T_{n,d}\to\widetilde T$ for this specialization,
its restriction gives the commuting square
\begin{equation}\label{lg:eq:finite-cancellation-square}
\begin{array}{ccc}
 \B_{n,d}&\lhook\joinrel\longrightarrow&T_{n,d}\\
 {\scriptstyle\delta_{n,d}}\downarrow&&\downarrow{\scriptstyle\sigma_{n,d}}\\
 \B_{n,d-2}&\lhook\joinrel\longrightarrow&\widetilde T.
\end{array}
\end{equation}
The lower inclusion passes through the constant copy of $T'$.
The map $\delta_{n,d}$ is surjective and kills $\rho_dT_{n,d}$.
Its kernel is identified after the Hilbert-series comparison in
Theorem~\ref{lg:thm:parameter-image}; see Theorem~\ref{lg:thm:layers}.
\end{proof}

\begin{theorem}[Exact parameter image]\label{lg:thm:parameter-image}
For every $n\ge1$ and $0\le d\le n$, one has
\begin{equation}\label{lg:eq:parameter-equality}
 \Hilb(\B_{n,d};q)=H_{n,d}(q),\qquad
 \ker\phi_{n,d}=\J_{n,d}.
\end{equation}
In particular $R_n/\J_{n,d}\cong\B_{n,d}$ embeds in $T_{n,d}$.
\end{theorem}
\begin{proof}
For $d=0$, $\B_{n,0}=\Q$ has Hilbert series $H_{n,0}=1$.
For $d=1$, one has $M=n$ and
$D(z)=(1-g_1^2z^2)^{-1/2}$. If $M$ is even, the single relation
$g_1D_M$ is a nonzero rational multiple of $g_1^{M+1}$; if $M$ is
odd, the single relation $D_{M+1}$ is such a multiple instead. Hence
\[
 T_{n,1}=\Q[g_1]/(g_1^{n+1}),\qquad \B_{n,1}=T_{n,1},
\]
because the coefficient of $z$ in $GD$ is $g_1$.
This has Hilbert series $1+q+\cdots+q^n=H_{n,1}$.

For $d\ge2$, Theorem~\ref{lg:thm:colon} and
Lemma~\ref{lg:lem:cancellation} give the coefficientwise lower bound
\begin{equation}\label{lg:eq:parameter-lower}
 \Hilb(\B_{n,d};q)\ge
 q^{\binom d2}\gb{n+1}{d}+\Hilb(\B_{n,d-2};q).
\end{equation}
The appropriate Gaussian Pascal identity gives
\begin{equation}\label{lg:eq:H-recursion}
 \begin{aligned}
 H_{n,d}-H_{n,d-2}
 &=q^{\binom d2}\left(\gb{n}{d-1}+q^d\gb{n}{d}\right)\\
 &=q^{\binom d2}\gb{n+1}{d}.
 \end{aligned}
\end{equation}
Induction in steps of two from the two base cases, with $n$ fixed,
proves $\Hilb(\B_{n,d};q)\ge H_{n,d}$.

On the other hand, Lemmas~\ref{lg:lem:B-description} and~\ref{lg:lem:positive}
give a surjection $R_n/\J_{n,d}\to\B_{n,d}$ and the independent bounds
\[
 \Hilb(\B_{n,d};q)\le\Hilb(R_n/\J_{n,d};q)\le H_{n,d}.
\]
All inequalities must be equalities degree by degree. The surjection
is therefore an isomorphism and its kernel is zero, proving both
assertions.
\end{proof}

\begin{example}\label{lg:ex:parameter}
For $n=3,m=1$, one has $d=2,M=2$ and
\[
 D_2=\frac{g_1^2-2g_2}{2},\qquad
 D_4=\frac38(g_1^2-2g_2)^2-\frac12g_2^2,
 \qquad T_{3,2}=\Q[g_1,g_2]/(D_4,g_2D_2).
\]
Here $\rho_2=g_1$ and the transition matrix in the proof of
Theorem~\ref{lg:thm:colon} is
\[
 A=\begin{pmatrix}0&-1/4\\g_1/4&1/4\end{pmatrix},
 \qquad\det A=g_1/16.
\]
The parameter ring has dimension $8$, whereas its coefficient image has
Hilbert series
$H_{3,2}=1+q+q^2+2q^3+q^4+q^5$ and dimension $7$.
Hence $\B_{3,2}$ is a proper subalgebra of $T_{3,2}$.
\end{example}

\subsection{The full ideal flag and complete-intersection layers}
The two coefficient windows determine the full ideal flag.
Cancellation identifies the layers between indices $d$ and $d-2$.

\begin{theorem}[The Lagrangian ideal flag]\label{lg:thm:layers}
The elements $Z_\lambda$, with $\lambda$ strict and $\lambda_1\le n$,
form a graded basis of $R_n$. For every $0\le d\le n$,
\begin{equation}\label{lg:eq:adapted-flag}
 \J_{n,d}=\Span_\Q\{Z_\lambda:
        \lambda_1\le n,\ \length(\lambda)>d\}.
\end{equation}
In particular $\J_{n,0}=(R_n)_{>0}$, $\J_{n,n}=0$, and
$\J_{n,n-1}=(R_n)_{\binom{n+1}{2}}$.
For $2\le d\le n$, put $M=n-d+1$ and use the ideals
$\aideal_{M,d}$ and $\bideal_{M,d}$ of
Theorem~\ref{lg:thm:colon}. There is an exact sequence of graded
$R_n$-modules
\begin{equation}\label{lg:eq:CI-layer}
 0\longrightarrow(S_d/\bideal_{M,d})(-\tbinom d2)
 \xrightarrow{\ \iota_d\ } R_n/\J_{n,d}
 \longrightarrow R_n/\J_{n,d-2}\longrightarrow0.
\end{equation}
Under the embedding $R_n/\J_{n,d}\hookrightarrow T_{n,d}$, the first
map sends $\overline f$ to $\rho_df$; the last map is the natural
quotient map. The $R_n$-action on its left term is induced by
$R_n\to T_{n,d}\to S_d/\bideal_{M,d}$. Consequently
\begin{equation}\label{lg:eq:CI-layer-Hilb}
 \Hilb(\J_{n,d-2}/\J_{n,d};q)
  =q^{\binom d2}\gb{n+1}{d}.
\end{equation}
The images of the $Z_\lambda$ with
$\lambda_1\le n$ and $\length(\lambda)\in\{d-1,d\}$ form a
basis of this layer.
\end{theorem}
\begin{proof}
Apply Lemma~\ref{common:positive} in each homogeneous component of
$\Gamma$, using its orthogonal Schur $Q$ basis, the positive operator
$\Theta$, and the projection retaining $\lambda_1\le n$ from
Lemma~\ref{lg:lem:ambient-basis}. It gives the full transformed basis.
For every $0\le d\le n$, the displayed vectors of length greater than $d$
lie in $\J_{n,d}$ by definition. Their degree-by-degree number is
$\Hilb(R_n)-H_{n,d}$, which is exactly the Hilbert series of this
ideal by Theorem~\ref{lg:thm:parameter-image}. This proves
\eqref{lg:eq:adapted-flag}. The empty partition gives the assertion at $d=0$, and there
is no retained partition of length greater than $n$, giving $d=n$.
The only retained strict partition of length $n$ is
$(n,n-1,\ldots,1)$, giving the top-degree assertion.

The conductor inclusion puts $\rho_dT_{n,d}$ inside
$\B_{n,d}=\im(\phi_{n,d})$. The colon identity identifies this
principal ideal with
$(S_d/\bideal_{M,d})(-\binom d2)$, also with the stated $R_n$-action.
Lemma~\ref{lg:lem:cancellation} gives a surjection
\[
 \B_{n,d}/\rho_dT_{n,d}\longrightarrow\B_{n,d-2}.
\]
Theorem~\ref{lg:thm:parameter-image} and the colon Hilbert series give
\[
 \Hilb(\B_{n,d})-\Hilb(\rho_dT_{n,d})
 =H_{n,d}-q^{\binom d2}\gb{n+1}{d}
 =H_{n,d-2}=\Hilb(\B_{n,d-2}).
\]
The middle equality follows by adding the length
$d-1$ and length $d$ terms of~\eqref{lg:eq:H} and applying Gaussian
Pascal. The surjection is therefore an isomorphism degree by degree.
Under the cancellation substitution,
\[
 GD=(1-t^2z^2)G^\circ\frac{D^\circ}{1-t^2z^2}=G^\circ D^\circ.
\]
Thus the composite map from $R_n$ to the smaller image is precisely
$\phi_{n,d-2}$, whose kernel is $\J_{n,d-2}$.
Transporting the resulting short exact sequence through the two image
isomorphisms gives~\eqref{lg:eq:CI-layer}. The difference of the nested
subsets in~\eqref{lg:eq:adapted-flag} gives the asserted layer basis.
\end{proof}

\begin{proposition}[Finite conductors and the terminal exception]
\label{lg:prop:finite-conductor}
With $\rho_0=\rho_1=1$, the full conductor of the finite parameter
image is
\begin{equation}\label{lg:eq:finite-conductor}
 \{u\in T_{n,d}:uT_{n,d}\subseteq\B_{n,d}\}
 =\begin{cases}
   \rho_dT_{n,d},&0\le d<n,\\
   T_{n,n},&d=n.
  \end{cases}
\end{equation}
For $n=d\ge2$, the ideal $\rho_dT_{n,d}$ is strictly smaller than
the conductor.
\end{proposition}
\begin{proof}
The cases $d=0,1$ follow from $\B_{n,d}=T_{n,d}$.
Suppose $2\le d<n$. Put $M=n-d+1\ge2$ and
$r=\lfloor M/2\rfloor\ge1$.
The cancellation map of Lemma~\ref{lg:lem:cancellation} takes
$T_{n,d}$ to $\widetilde T=T'[t]/(f_{r+1}(t))$, which is free
over $T'=T_{n,d-2}$ with basis $1,t,\ldots,t^{2r+1}$.
Its restriction to $\B_{n,d}$ has image
$\B_{n,d-2}\subseteq T'$ and kernel $\rho_dT_{n,d}$ by
Theorem~\ref{lg:thm:layers}.

If $u$ is in the conductor, both $u$ and $ug_2$ are in
$\B_{n,d}$. Writing the cancellation image of $u$ as $u'\in T'$,
the image of $ug_2$ is $u'(g_2^\circ-t^2)$, where
$g_2^\circ=0$ when $d-2<2$. It must be constant in $t$.
The independence of $1,t^2$ over $T'$ forces $u'=0$.
This proves containment in $\rho_dT_{n,d}$; the reverse inclusion
holds because this $T_{n,d}$-ideal lies in $\B_{n,d}$.

Finally, if $d=n$, then $M=1$ and the defining window kills
$D_2,D_4,\ldots,D_{2d}$. Lemma~\ref{lg:lem:finite} gives $D=1$
in $T_{n,n}[[z]]$. The coefficients of $GD=G$ therefore generate
all of $T_{n,n}$, so $\B_{n,n}=T_{n,n}$ and its conductor is the
whole ring. For $n\ge2$, $\rho_n$ has positive weight, making
$\rho_nT_{n,n}$ a proper ideal. Correspondingly, the monic
extension in this case has degree two and does not make $t^2$
independent of the constant summand.
\end{proof}

\begin{remark}[The scope of the full flag]\label{lg:rem:parity-boundary}
The flag formula holds for both parities of $n-d$ and includes
both endpoints. The complete-intersection sequence connects $d$
to $d-2$; it gives no such description of the one-step quotient
$\J_{n,d-1}/\J_{n,d}$. Its left term is the colon quotient
specified in~\eqref{lg:eq:CI-layer}. Equality of Hilbert series
alone would not identify it with a Grassmannian cohomology ring.

The restriction to odd $m$ in the subsequent tangency and containment
statements is different and cannot be removed. For example, in $R_3$,
\[
 Z_{(2,1)}=2e_1^3-4e_3\in\J_{3,1},\qquad
 Z_{(2,1)}\notin\Lf_2=\Lf_1.
\]
Indeed $e_1^2=2e_2$, and the squarefree basis shows that
$e_1^3=2e_1e_2$ and $e_3$ are independent in degree three.
The degree-three component of $\Lf_1$ is the line spanned by $e_1^3$.
Thus a full auxiliary ideal flag does not imply
$\J_{n,n-m}\subseteq\Lf_m$ for even $m$.
\end{remark}
\begin{example}[Equal Hilbert series and different coefficient algebras]
\label{lg:ex:non-Grassmannian-coefficients}
The coefficient ring of a layer need not be the ordinary Grassmannian
ring suggested by its Gaussian Hilbert series. For $M=d=2$, put
$x=g_1$ and $y=g_2$, of weights $1$ and $2$. The coefficient ring
and the Grassmannian ring have presentations
\[
 \begin{aligned}
 C=S_2/\bideal_{2,2}
   &=\Q[x,y]/(3x^3-8xy,\ x^2y-2y^2),\\
 H&=\Q[x,y]/(x^3-2xy,\ x^2y-y^2).
 \end{aligned}
\]
Here $H\simeq H^{2*}(\Gr(2,4);\Q)$ and
$\Hilb(C;q)=\Hilb(H;q)=\gb42$.
The rings $C$ and $H$ are not isomorphic even as ungraded
$\Q$-algebras, but
$C\otimes_\Q\Q(\sqrt3)\simeq H\otimes_\Q\Q(\sqrt3)$
as graded algebras. Appendix~\ref{lg:appendix:coefficients}
contains the calculation and an explicit isomorphism after this
field extension.
\end{example}

\section{Lagrangian tangency and the filtration formula}
\label{lg:sec:rigidity}

\subsection{The derivative inclusion}
For odd $m$, let
\[
 P_m=\Q[c_1,c_3,\ldots,c_m],\quad \deg c_i=i,
 \qquad K_m=\ker(P_m\longrightarrow R_n),\quad c_i\longmapsto e_i.
\]
Its image is $\Lf_m$: the coefficient of $z^{2r}$ in
$E(z)E(-z)=1$ expresses $2e_{2r}$ as a polynomial in the
preceding generators.
When $m=1$, interpret $P_{m-2}=\Q$.
Define homogeneous vector spaces
\begin{equation}\label{lg:eq:W-Q}
 W_m=\partial_{c_m}K_m,\qquad
 L_m=\{F\in P_m:F(e_1,e_3,\ldots,e_m)\in\J_{n,n-m}\}.
\end{equation}
Here $L_m$ is an ideal; no ideal structure is assumed for $W_m$.
The full kernel $K_m$ is used, with no truncation by source weight.

\begin{lemma}[Tangency]\label{lg:lem:tangency}
For every odd $m$ with $1\le m\le n$, one has $W_m\subseteq L_m$.
\end{lemma}
\begin{proof}
Put $d=n-m$ and give $\varepsilon$ weight $m$. In the commuting
extension $T_{n,d}[\varepsilon]/(\varepsilon^2)$, set
\begin{equation}\label{lg:eq:tangent-series}
 E_\varepsilon(z)=G(z)\bigl(D(z)+\varepsilon z^m\bigr).
\end{equation}
Since $D$ is even and $m$ is odd,
\[
 E_\varepsilon(z)E_\varepsilon(-z)
 =G(z)G(-z)(D(z)+\varepsilon z^m)(D(z)-\varepsilon z^m)=1.
\]
Its degree is at most $n$, because $\deg_z(GD)\le n$ and $d+m=n$.
Its coefficients therefore define a homomorphism from $R_n$, reducing
to $\phi_{n,d}$. Coefficients below $m$ are unchanged, and the coefficient
at $m$ increases by $\varepsilon$. Apply Lemma~\ref{common:witness}
using the injection from Theorem~\ref{lg:thm:parameter-image}.
\end{proof}

All sums over odd indices in this section are restricted to $1\le i\le n$.
For odd $m$ with $1\le m\le n$, write
$\Delta_m=\Hilb(\Lf_m/\Lf_{m-1};q)$ and set
\begin{equation}\label{lg:eq:epsilon-def}
 \epsilon_m=\Delta_m-q^m\left(H_{n,n-m}
                 -\sum_{\substack{i>m\\i\text{ odd},\ i\le n}}\Delta_i\right).
\end{equation}
Apply~\eqref{common:defect} to the successive odd generators
indexed by $m=2j-1$: the ideal at weight $m$ is $\J_{n,n-m}$,
with preimage $L_m$. Lemma~\ref{lg:lem:tangency} gives
\begin{equation}\label{lg:eq:defect-decomposition}
 q^{-m}\epsilon_m=\Hilb(L_m/W_m;q)
 +\Hilb\bigl(\J_{n,n-m}/(\J_{n,n-m}\cap \Lf_m);q\bigr).
\end{equation}
Both terms are Hilbert series of graded quotient spaces, so
$\epsilon_m$ has nonnegative coefficients.

\subsection{The triangular Gaussian identity}
Set $N=\binom{n+1}{2}$ and, for odd $m$, put
\begin{equation}\label{lg:eq:F-target}
 F_m(q)=\sum_{j=0}^{n-m}q^{\binom{j+1}{2}}\gb{m+j}{m},
 \qquad T_m(q)=q^m F_m(q).
\end{equation}
These are the summands of~\cite[Conjecture~1.4]{Polymath}.
The total identity~\eqref{lg:eq:target-total} is
\cite[Proposition~5.2]{Polymath}; we derive it together with
the triangular recursion.

\begin{lemma}\label{lg:lem:targets}
For every odd $m$ with $1\le m\le n$,
\begin{equation}\label{lg:eq:triangular-target}
 T_m=q^m\left(H_{n,n-m}-
                   \sum_{\substack{i>m\\i\text{ odd}}}T_i\right).
\end{equation}
Furthermore,
\begin{equation}\label{lg:eq:target-total}
 \sum_{\substack{1\le m\le n\\m\text{ odd}}}T_m
       =\prod_{i=1}^n(1+q^i)-1.
\end{equation}
\end{lemma}
\begin{proof}
Write $d=n-m$. If $d\ge2$, we first establish
\begin{equation}\label{lg:eq:F-recursion}
 F_m-(1-q^{m+2})F_{m+2}
        =q^{\binom d2}\gb{n+1}{d}=H_{n,d}-H_{n,d-2}.
\end{equation}
For fixed $m$ let $B_j=q^{\binom j2}\gb{m+j+1}{j}$.
The two elementary Gaussian identities
\[
 \gb{m+j+1}{j}=q^j\gb{m+j}{j}+\gb{m+j}{j-1},
 \qquad
 (1-q^{j-1})\gb{m+j}{j-1}
       =(1-q^{m+2})\gb{m+j}{j-2}
\]
give, for $j\ge2$,
\begin{equation}\label{lg:eq:telescope}
 \begin{aligned}
 B_j-B_{j-1}
 &=q^{\binom{j+1}{2}}\gb{m+j}{j}\\
 &\quad -(1-q^{m+2})q^{\binom{j-1}{2}}\gb{m+j}{j-2}.
 \end{aligned}
\end{equation}
Also $B_0=1$ and $B_1-B_0=q\gb{m+1}{1}$.
Summing~\eqref{lg:eq:telescope} up to $d$, with the second sum reindexed
by $j-2$, proves the first equality in~\eqref{lg:eq:F-recursion}; the
second is~\eqref{lg:eq:H-recursion}.

At the final odd value of $m$, the number $d$ is $0$ or $1$.
For $d=0$, $F_m=1=H_{n,0}$; for $d=1$,
$F_m=1+q\gb{n}{1}=H_{n,1}$.
Descending induction through the odd values of $m$, using
\eqref{lg:eq:F-recursion}, now gives
\[
 F_m=H_{n,n-m}-\sum_{i>m,\,i\text{ odd}}T_i,
\]
which is~\eqref{lg:eq:triangular-target}.

To evaluate the total, the terms of
\[
 F_1=\sum_{j=0}^{n-1}q^{\binom{j+1}{2}}[j+1]_q
\]
with $[a]_q=\gb{a}{1}$ cover disjoint consecutive intervals of degrees from $0$ to $N-1$.
Thus $T_1=q(1-q^N)/(1-q)$.
There is only one strict partition of length $n$ with largest part at
most $n$, namely $(n,n-1,\ldots,1)$, so
$H_{n,n-1}=\prod_{i=1}^n(1+q^i)-q^N$.
The case $m=1$ of~\eqref{lg:eq:triangular-target} therefore yields
\[
 \sum_{m\text{ odd}}T_m
 =H_{n,n-1}+(1-q^{-1})T_1
 =\prod_{i=1}^n(1+q^i)-1.
\]
This proves~\eqref{lg:eq:target-total}.
\end{proof}

\subsection{Completion of the Lagrangian proof}
\begin{proof}[Proof of Theorem~\ref{lg:thm:main}]
Apply Theorem~\ref{common:rigidity} to the successive odd generators
of $R_n$. At an odd index $m$, the preceding algebra is $\Lf_{m-1}$,
since the even generators are redundant.
Theorem~\ref{lg:thm:parameter-image} computes the required quotients,
Lemma~\ref{lg:lem:tangency} verifies the derivative inclusions, and
Lemma~\ref{lg:lem:targets} verifies~\eqref{common:targets}.
Thus $\Delta_m=T_m$ for every odd $m$; even increments are zero.
Summing proves~\eqref{lg:eq:main}.
\end{proof}

\section{Consequences for Lagrangian Grassmannians}\label{lg:sec:consequences}

Theorem~\ref{lg:thm:main} gives the following dimensions and
stabilization thresholds. The vanishing defects also give ideal
containments and differentiation sequences.

\begin{corollary}[Dimensions and sharp high-degree stabilization]
\label{lg:cor:dimensions}
For odd $m$ with $1\le m\le n$,
\begin{equation}\label{lg:eq:dimensions}
 \dim_\Q(\Lf_m/\Lf_{m-1})=\binom{n+1}{m+1},\qquad
 \dim_\Q \Lf_m=1+\sum_{\substack{1\le i\le m\\i\text{ odd}}}
                           \binom{n+1}{i+1}.
\end{equation}
With $N=\binom{n+1}{2}$, put $b_m=N-\binom m2$.
Then
\begin{equation}\label{lg:eq:sharp-stabilization}
 (\Lf_{m-1})_t=(R_n)_t\quad(t>b_m),\qquad
 \dim_\Q(R_n/\Lf_{m-1})_{b_m}=1.
\end{equation}
In particular, the threshold is sharp.
\end{corollary}
\begin{proof}
At $q=1$, the inner sum in~\eqref{lg:eq:increment-intro} is
$\sum_{j=0}^{n-m}\binom{m+j}{m}=\binom{n+1}{m+1}$.
This proves~\eqref{lg:eq:dimensions}.
The degree of the $j$th term in the same formula is
\[
 m+\binom{j+1}{2}+mj.
\]
It is strictly increasing with $j$. At $j=n-m$ it equals
$N-\binom m2$, and its leading coefficient is $1$.
All later odd increments, with index $i>m$, have strictly smaller top
degree $N-\binom i2$. The quotient $R_n/\Lf_{m-1}$ is the sum of those
increments and the $m$th increment at the level of Hilbert series.
This proves both claims in~\eqref{lg:eq:sharp-stabilization}.
\end{proof}

\begin{corollary}[Derivative ideals and auxiliary-ideal containment]
\label{lg:cor:structural}
For odd $m$ with $1\le m\le n$, one has
\begin{equation}\label{lg:eq:structural}
 \J_{n,n-m}\subseteq \Lf_m,\qquad
 \partial_{c_m}K_m=L_m.
\end{equation}
In particular, the derivative space is an ideal of $P_m$. The elements
\[
 Z_\lambda\quad
 (\lambda\text{ strict},\ \lambda_1\le n,\ \length(\lambda)>n-m)
\]
form a vector-space basis of $\J_{n,n-m}$, and all belong to $\Lf_m$.
\end{corollary}
\begin{proof}
The equality case of Theorem~\ref{common:rigidity} gives
\eqref{lg:eq:structural}. The asserted basis is the specialization
$d=n-m$ of Theorem~\ref{lg:thm:layers}, proved independently of the
filtration formula.
\end{proof}

\begin{corollary}[Lagrangian differentiation sequence]\label{lg:cor:module-sequence}
For odd $1\le m\le n$, differentiation in the newest odd generator
gives an exact sequence of graded $\Lf_{m-1}$-modules
\begin{equation}\label{lg:eq:module-sequence}
 0\longrightarrow \Lf_{m-1}\longrightarrow \Lf_m
 \longrightarrow(\Lf_m/\J_{n,n-m})(-m)\longrightarrow0.
\end{equation}
The last map sends the class of $F\in\Q[c_1,c_3,\ldots,c_m]$ to
the class of $\partial_{c_m}F$. It is not asserted to be an algebra
homomorphism or an $\Lf_m$-module homomorphism.
\end{corollary}
\begin{proof}
Apply~\eqref{common:module-sequence} to the ordered odd generators,
using~\eqref{lg:eq:structural}. The preceding algebra is
$\Lf_{m-1}=\Lf_{m-2}$ for $m>1$, and $\Q$ for $m=1$;
the generator weight gives the shift $(-m)$.
\end{proof}
\begin{example}\label{lg:ex:n5}
For $n=5$, $N=15$ and $\Lf_1=\Lf_2$ has Hilbert series
$1+q+\cdots+q^{15}$. The next increment is
\[
 \begin{aligned}
 \Hilb(\Lf_3/\Lf_2;q)
 &=q^3\left(1+q\gb{4}{3}+q^3\gb{5}{3}\right)\\
 &=q^3+q^4+q^5+2q^6+2q^7+2q^8+2q^9+2q^{10}+q^{11}+q^{12}.
 \end{aligned}
\]
It has dimension $15=\binom64$ and top degree $12=15-\binom32$.
Here $\Lf_4=\Lf_3$ has dimension $31$, and the final increment
$\Lf_5/\Lf_4$ has Hilbert series $q^5$, giving $\dim R_5=32$.
\end{example}

\subsection{Related basis questions}
The shifted $k$-conjugation and Schur $Q$ analogue asked for
in~\cite[Question~5.3]{Polymath} are not constructed here.
Corollary~\ref{lg:cor:structural} gives a basis of an auxiliary ideal;
it does not give a basis of the whole filtered subalgebra.
In the ordinary case, Corollary~\ref{counter:uniform-repair}
gives bases only in the specified counterexample components.

\section{Two-row prescribed bases}
\label{lg:sec:counter}

We first prove the infinite family in Theorem~\ref{lg:thm:counter-intro}
and then study the two families prescribed
in~\cite[Conjectures~1.2 and~1.3]{Polymath} in arbitrary two-row components.
The argument uses the core indexing and $k$-Schur rectangle rule,
independently of the filtration theorems. The example
in~\cite[Section~5.5]{REU} concerns an index-set correspondence;
the counterexamples below establish linear dependence of the prescribed vectors.

Recall that $k$-conjugation passes from a $k$-bounded partition to
its associated $(k+1)$-core, transposes the core, and counts in each
row the cells whose hook length is at most $k$. We use the conventions
of Lapointe--Morse~\cite{LM} and~\cite[Section~3]{Polymath}.
Throughout the uniform argument, let
\[
 r\ge2,\qquad k=3r-1,\qquad N=3r+2,\qquad t=\lfloor N/2\rfloor.
\]

\subsection{The indices and the two candidate spans}
\begin{lemma}\label{lg:lem:indices}
The weight-$N$ elements of $\mathcal P^{2,k,2}$ are exactly
\begin{equation}\label{counter:family-indices}
 (1^N),\qquad (2^a,1^{N-2a})\quad(4\le a\le t).
\end{equation}
There are $t-2$ such indices. For $r=2$, these are $(1^8)$ and $(2^4)$.
\end{lemma}
\begin{proof}
The partitions of $N=k+3$ in $(k^2)$ are
$\nu_a=(N-a,a)$, for $3\le a\le t$. If $a\ge4$, their largest
hook is $N-a+1\le k$, so the associated $(k+1)$-core is $\nu_a$
itself. Their $k$-conjugates are therefore
$\nu_a'=(2^a,1^{N-2a})$.

For $a=3$, the core corresponding to $\nu_3=(k,3)$ is
$\kappa=(k+3,3)$. The first three hooks in its top row are
$k+4,k+3,k+2$, and the remaining hooks in that row are
$k,k-1,\ldots,1$; the bottom-row hooks are $3,2,1$.
The largest hook is $k+4<2(k+1)$, since $k\ge5$, and none has
length $k+1$. Thus $\kappa$ is a $(k+1)$-core. Its small-hook row
counts are $k,3$, as required. In its transpose
$\kappa'=(2,2,2,1^k)$, each row has exactly one hook of length at most
$k$. This gives $\nu_3^{\omega(k)}=(1^N)$.
Since $k$-conjugation is an involution, this list exhausts the indices.
\end{proof}

Set $u=h_1$ and $v=h_2$, so that $e_2=u^2-v$ and
\begin{equation}\label{counter:uv-ring}
 R^{2,k}=\Q[u,v]/(h_{k+1},h_{k+2}),\qquad
 h_0=1,\quad h_1=u,\quad
 h_n=u h_{n-1}-(u^2-v)h_{n-2}\quad(n\ge2).
\end{equation}
Here $\deg u=1$ and $\deg v=2$. The complete-symmetric candidates
selected by~\eqref{counter:family-indices} span the polynomial space
\begin{equation}\label{counter:candidate-space}
 V_N=\Span_\Q\{u^N,\ u^{N-2a}v^a:4\le a\le t\}
 \subseteq\Q[u,v]_N.
\end{equation}

\begin{lemma}\label{counter:two-row-schur-span}
Before passing to the quotient, the prescribed $k$-Schur candidates
for these indices span the same space $V_N$.
\end{lemma}
\begin{proof}
For $a>0$, the level specified by the conjecture is $2$.
The $2$-rectangles $(2)$ and $(1,1)$ have Schur polynomials $h_2$
and $e_2$, respectively. Repeated use of the rectangle product rule
\cite[Theorem~40]{LM} gives
\[
 s_{(2^a,1^b)}^{(2)}
   =h_2^a e_2^{\lfloor b/2\rfloor}h_1^{b\bmod2}.
\]
For $b=N-2a$ and $\epsilon=N\bmod2$, this is
\begin{equation}\label{counter:triangular-change}
 u^\epsilon v^a(u^2-v)^{t-a}
   =\sum_{j=a}^t(-1)^{j-a}\binom{t-a}{j-a}
                         u^{N-2j}v^j.
\end{equation}
On the consecutive indices $a=4,\ldots,t$, the change of coordinates
is triangular with diagonal entries $1$. For $(1^N)$ the prescribed
level is $1$, giving $s_{(1^N)}^{(1)}=u^N$ separately.
Thus both candidate families span~\eqref{counter:candidate-space}.
\end{proof}

\subsection{A uniform ideal certificate and its exact rank defect}
\begin{proof}[Proof of Theorem~\ref{lg:thm:counter-intro}]
All coefficient calculations in this proof take place in
$\Q[u,v]$, before forming the Artinian quotient. Put
$c_{n,a}=[u^{n-2a}v^a]h_n$. The recurrence
in~\eqref{counter:uv-ring} gives
\[
 \sum_{n\ge0}h_n(1,w)z^n
   =\frac1{1-z+(1-w)z^2}.
\]
Expanding first in $w$, and using
$1-z+z^2=(1+z^3)/(1+z)$, yields
\begin{equation}\label{counter:coefficient-extraction}
 c_{n,a}=[z^{n-2a}]
              \frac{(1+z)^{a+1}}{(1+z^3)^{a+1}}.
\end{equation}
In particular, the first four coefficients, with their common sign
removed, are as follows:
\begin{equation}\label{counter:low-coefficients}
 \renewcommand{\arraystretch}{1.3}
 \begin{array}{c|rrrr}
 n&(-1)^r c_{n,0}&(-1)^r c_{n,1}&(-1)^r c_{n,2}&(-1)^r c_{n,3}\\ \hline
 3r&1&-2r&\frac32r(r-1)&\frac12r(r-1)(3-r)\\
 3r+1&1&-r&-r&\frac12r(r-1)(r+2).
 \end{array}
\end{equation}
These formulas follow directly by the binomial expansion of the
right side of~\eqref{counter:coefficient-extraction}; they also hold
at $r=2$, with coefficients of negative powers of $z$ read as zero.

Consider the homogeneous polynomial
\begin{equation}\label{counter:uniform-certificate}
 P_r=\bigl(2(2r+1)u^2+r(3r+1)v\bigr)h_{3r}
                         -(5r+3)u h_{3r+1}.
\end{equation}
Its coefficients with $v$-exponents $1,2,3$ vanish by
\eqref{counter:low-coefficients}, whereas
\[
 [u^N]P_r=(-1)^{r+1}(r+1)\ne0.
\]
Therefore $0\ne P_r\in V_N\cap(h_{3r},h_{3r+1})$.
This is a polynomial ideal-membership certificate for a dependence
among the complete-symmetric candidates, uniformly in $r$.
Lemma~\ref{counter:two-row-schur-span} transfers the failure to the
second prescribed family.

To compute the defect exactly, put $I=(h_{3r},h_{3r+1})$.
Its weight-$N$ part is spanned by
\[
 u^2h_{3r},\qquad vh_{3r},\qquad u h_{3r+1}.
\]
The usual Schur basis of $R^{2,k}$ has $t-2$ members of weight $N$,
namely the $\nu_a$ in the proof of Lemma~\ref{lg:lem:indices}.
Since $\dim\Q[u,v]_N=t+1$, it follows that $\dim I_N=3$, so the
three displayed relations are a basis.

Project $\Q[u,v]_N$ to the coefficients with $v$-exponents $1,2,3$;
its kernel is $V_N$. On the first two displayed relations, the first
two rows of this projection matrix, after removing the common sign
$(-1)^r$, form
\[
 \begin{pmatrix}
 -2r&1\\ \frac32r(r-1)&-2r
 \end{pmatrix},\qquad
 \det=\frac{r(5r+3)}2\ne0.
\]
Thus the restriction to $I_N$ has rank at least two. The nonzero
kernel vector~\eqref{counter:uniform-certificate} makes its rank
exactly two. Consequently $\dim(I_N\cap V_N)=1$, and the candidate
span in the quotient has dimension $(t-2)-1=t-3$.
Both families therefore miss precisely one dimension of the
weight-$N$ component. At $r=2$, the certificate reads
$P_2=-3u^8+14v^4$, which gives~\eqref{lg:eq:counter-intro}.
\end{proof}

\begin{corollary}[A common index replacement in the counterexample family]
\label{counter:uniform-repair}
With $r,k,N,t$ as above, replace $(1^N)$ in
\eqref{counter:family-indices} by $(2,1^{N-2})$.
The complete-symmetric candidates for this modified list form a basis
of $(R^{2,k})_N$. The $k$-Schur candidates, each taken at the level
of its own largest part, also form a basis of the same component.
\end{corollary}
\begin{proof}
For $f\in\Q[u,v]_N$, define the coefficient functional
\[
 \eta_r(f)=\frac{r(r-1)}2[u^{N-2}v]f
          +(r-1)[u^{N-4}v^2]f+[u^{N-6}v^3]f.
\]
The coefficient table~\eqref{counter:low-coefficients} shows that
$\eta_r$ vanishes on each of the three displayed generators of
$I_N$. It therefore descends to $(R^{2,k})_N$, where it vanishes
on the original candidate span. The unique relation among those
candidates has nonzero $u^N$ coefficient, by
\eqref{counter:uniform-certificate}. Consequently the candidates
with $4\le a\le t$ are linearly independent after taking the
quotient, for either family.

The new complete-symmetric vector is $u^{N-2}v$, on which
$\eta_r$ has the nonzero value $r(r-1)/2$. Its level-$2$
$k$-Schur counterpart is $u^\epsilon v(u^2-v)^{t-1}$,
with $\epsilon=N\bmod2$, and its value is
\[
 \frac{r(r-1)}2-(r-1)(t-1)+\binom{t-1}{2}
       =\frac{(r-t)(r-t+1)}2\ne0.
\]
Indeed, $t-r=\lfloor(r+2)/2\rfloor\ge2$.
In either case the new vector is outside the original candidate
span. Adding it to the $t-3$ retained independent vectors gives
a basis of the $(t-2)$-dimensional component.
\end{proof}

The first member is $(k,N)=(5,8)$. In this case the relation is
$3u^8=14v^4$, with both sides nonzero in $(R^{2,5})_8$.
Appendix~\ref{counter:appendix:local} gives its polynomial certificate,
Schur coordinates, and a second local replacement.

\subsection{A two-row basis theorem from Catalan moments}
For even weight excess, the pairing matrix of the candidates admits
a positive integral representation. We use determinant integration
and positivity of symmetric functions as in~\cite{Forrester,HM}.

\begin{proposition}\label{counter:even-offset}
Let $k\ge2$. Both families indexed by $\mathcal P^{2,k,2}$ and
prescribed by \cite[Conjectures~1.2 and~1.3]{Polymath} are bases of
$(R^{2,k})_N$ whenever $0\le N\le k$, and whenever
$N=k+d$ with $0\le d\le k$ even.
\end{proposition}
\begin{proof}
For $N\le k$, every partition of $N$ in $(k^2)$ is already a
$(k+1)$-core. The prescribed indices are therefore all
$(2^a,1^{N-2a})$, for $0\le a\le\lfloor N/2\rfloor$.
The complete-symmetric candidates are all weighted monomials in
$u=h_1,v=h_2$, and there are no relations in these weights.
The rectangle rule and the triangular expansion
in~\eqref{counter:triangular-change} give the same conclusion for
the second family. At weight zero both families consist of $1$.

Now put $N=k+d$, $t=\lfloor N/2\rfloor$, and
$p=t-d=\lfloor(k-d)/2\rfloor$.
For $d\ge1$, the partitions of $N$ in $(k^2)$ are
$(N-a,a)$ with $d\le a\le t$. When $a\ge d+1$ their largest
hook is at most $k$, so their $k$-conjugates are their transposes.
The remaining partition $(k,d)$ has associated core $(k+d,d)$:
its first $d$ top-row hooks are $k+d+1,\ldots,k+2$, the remaining
top-row hooks are $k,\ldots,1$, and the bottom-row hooks are
$d,\ldots,1$. There is no hook of length $k+1$, and the largest
hook is $k+d+1\le2k+1<2(k+1)$ because $1\le d\le k$.
Hence this partition is a $(k+1)$-core, with small-hook row counts
$k,d$. Its transpose $(2^d,1^k)$ has one hook of length at most $k$
in each row.
This proves, also for $d=0$ by the preceding case, that the
complete-symmetric candidates are
\begin{equation}\label{counter:general-two-row-candidates}
 u^N,\qquad u^{N-2a}v^a\quad(d+1\le a\le t).
\end{equation}
They number $p+1=\dim(R^{2,k})_N$. The same rectangle-rule
expansion as before is triangular on the consecutive indices
$d+1,\ldots,t$, so the second family has the same span.

Let $\tau:(R^{2,k})_{2k}\to\Q$ take the coefficient of the top
Schur class $e_2^k$, normalized by $\tau(e_2^k)=1$.
The complementary weight $k-d$ has the monomial basis
$u^{k-d-2j}v^j$, $0\le j\le p$. Poincar\'e duality reduces
independence in~\eqref{counter:general-two-row-candidates} to the
nonsingularity of its pairing matrix with this basis.
Pieri multiplication and the two-row tableau count give
\[
 \begin{gathered}
 \tau\bigl(u^{2k-2i}e_2^i\bigr)=C_{k-i}\quad(0\le i\le k),\\
 C_b=\frac1{b+1}\binom{2b}{b}\quad(b\ge0).
 \end{gathered}
\]
A standard Young tableau is a filling by $1,\ldots,|\lambda|$, each
used once, with entries increasing along every row and column.
Multiplication by $e_2^i$ adds $i$ cells to each row,
leaving the Catalan number of standard tableaux of shape
$(k-i,k-i)$. Since $v=u^2-e_2$, the pairing matrix is
\begin{equation}\label{counter:catalan-pairing}
\begin{aligned}
 K&=\bigl(\gamma_k(a_i+j)\bigr)_{0\le i,j\le p},
       &a_0&=0,\quad a_i=d+i\quad(i\ge1),\\
 \gamma_k(b)&=\sum_{i=0}^b(-1)^i\binom bi C_{k-i}
       &&(0\le b\le k).
\end{aligned}
\end{equation}

To analyze it, use the Catalan probability measure
\[
 d\mu(x)=\frac1{2\pi}\sqrt{\frac{4-x}{x}}\,dx
 \quad(0<x<4).
\]
The substitution $x=4s$ in the beta integral gives
$\int x^j\,d\mu(x)=C_j$. Define a positive measure $\nu_k$ by
\[
 \int f(z)\,d\nu_k(z)=\int_0^4 x^k f(1-1/x)\,d\mu(x)
\]
whenever the integral is absolutely convergent. The binomial expansion gives
\[
 \gamma_k(b)=\int z^b\,d\nu_k(z)\qquad(0\le b\le k).
\]
These moments are absolutely convergent: near $x=0$ the absolute
integrand is $O(x^{k-b-1/2})\,dx$.

Expanding two determinants and applying Fubini gives Andr\'eief's
integration identity (see~\cite[Equation~(1.7) and Section~2.2]{Forrester})
in the form
\begin{align}
 \det K
 &=\frac1{(p+1)!}\int
     \det(z_j^{a_i})_{i,j=0}^p\det(z_j^i)_{i,j=0}^p
                         \prod_{j=0}^p d\nu_k(z_j)\notag\\
 &=\frac1{(p+1)!}\int
       \Delta(z)^2s_{(d^p)}(z_0,\ldots,z_p)
                         \prod_{j=0}^p d\nu_k(z_j),
 \label{counter:moment-determinant}
\end{align}
where $\Delta(z)=\prod_{i<j}(z_j-z_i)$.
Every individual variable in the determinant expansion has exponent
at most $d+2p\le k$, which verifies the absolute-integrability
hypothesis for Fubini. The second equality is the Schur alternant
identity, with exactly $p+1$ variables.

For $p=0$, the determinant is $C_k>0$. Suppose $p\ge1$ and put
$Z_i=\prod_{j\ne i}z_j$. For nonzero $z_0,\ldots,z_p$, the
complementary-partition alternant identity in $p+1$ variables gives
\[
 s_{(d^p)}(z_0,\ldots,z_p)
 =\left(\prod_{j=0}^p z_j\right)^d
             h_d(z_0^{-1},\ldots,z_p^{-1})
 =h_d(Z_0,\ldots,Z_p).
\]
The second equality follows term by term from the homogeneity of
$h_d$. Both ends are polynomials, so the identity also holds when
some $z_j=0$.
For arbitrary real $y_0,\ldots,y_p$, the simplex formula is
\[
 h_d(y_0,\ldots,y_p)
   =\frac{(d+p)!}{d!}\int_{\Delta_p}
                    \left(\sum_{i=0}^p t_i y_i\right)^d\,dt.
\]
Here $\Delta_p=\{t_0,\ldots,t_{p-1}\ge0:\sum_{i<p}t_i\le1\}$,
$t_p=1-\sum_{i<p}t_i$, and $dt=dt_0\cdots dt_{p-1}$.
It follows by expanding the power and using
$\int_{\Delta_p}t_0^{b_0}\cdots t_p^{b_p}\,dt
=\prod_i b_i!/(d+p)!$ for $\sum_i b_i=d$.
For even $d$ it proves nonnegativity on all real arguments.
Thus the integrand in~\eqref{counter:moment-determinant} is
nonnegative. It is strictly positive on the open set where the
$z_i$ are positive and distinct, which has positive product measure:
one may take distinct $x_i\in(1,4)$. Consequently $\det K>0$.
The perfect pairing and the equality of the two candidate spans
complete the proof.
\end{proof}

We retain the measure $\nu_k$ and the moments $\gamma_k$ from the preceding proof.

\begin{proposition}[A uniform bound on the two-row defect]
\label{counter:defect-bound}
Let $k\ge2$ and $0\le N\le2k$. In weight $N$, each of the two
prescribed families indexed by $\mathcal P^{2,k,2}$ has codimension
at most one in $(R^{2,k})_N$.
Moreover, if $N=k+d$ with $1\le d\le k$ odd, the candidates obtained
by deleting the index $(1^N)$ are linearly independent in either family.
\end{proposition}
\begin{proof}
The ranges $N\le k$ and $d=N-k$ even are covered by
Proposition~\ref{counter:even-offset}. Suppose $d$ is positive and odd,
and put $p=\lfloor(k-d)/2\rfloor$. If $p=0$, the sole candidate is
$u^N$, and its pairing with $u^{k-d}$ is $C_k>0$.
For $p\ge1$, the complete-function tail in
\eqref{counter:general-two-row-candidates} has pairing matrix, against
the first $p$ complementary monomials,
\[
 H=\bigl(\gamma_k(d+1+i+j)\bigr)_{0\le i,j<p}.
\]
For any nonzero real polynomial $f(z)=\sum_{i=0}^{p-1}f_i z^i$, the
Catalan measure representation gives
\[
 \sum_{i,j=0}^{p-1}f_i f_j\gamma_k(d+1+i+j)
  =\int z^{d+1}f(z)^2\,d\nu_k(z)>0.
\]
Indeed, $d+1$ is even, the measure has positive mass on intervals,
and the largest moment exponent is $d+2p-1\le k-1$, so the integral
is absolutely convergent. Thus $H$ is positive definite and the $p$
tail vectors are independent. There are $p+1$ candidates in a
$(p+1)$-dimensional component, proving the codimension bound.
The triangular change in~\eqref{counter:triangular-change} involves
only the tail indices and has diagonal entries one. Hence the second
tail is independent as well.
\end{proof}

\subsection{A coefficient criterion and small odd excesses}

Recall that $d=N-k$ is the weight excess over the two-row rectangle width.
For $k\ge2$ and $1\le d\le k$, put $N=k+d$ and $t=\lfloor N/2\rfloor$.
Write
\[
 h_n(u,v)=\sum_{a=0}^{\lfloor n/2\rfloor}c_{n,a}u^{n-2a}v^a
 \qquad(n\ge0),
\]
and set $c_{n,a}=0$ if $a<0$ or $2a>n$.
Define a $d$ by $d$ matrix $M_{k,d}$, with rows indexed by
$a=1,\ldots,d$, by concatenating the following two blocks of columns:
\begin{equation}\label{counter:projection-matrix}
 M_{k,d}=\left(
    (c_{k+1,a-j})_{\substack{1\le a\le d\\0\le j\le\lfloor(d-1)/2\rfloor}}
   \ \middle|\
    (c_{k+2,a-j})_{\substack{1\le a\le d\\0\le j\le\lfloor(d-2)/2\rfloor}}
           \right).
\end{equation}
For $d=1$ the second block is empty.

\begin{lemma}[A two-row coefficient criterion]
\label{counter:projection-criterion}
In weight $N=k+d$, the span of each prescribed family has codimension
\[
 d-\operatorname{rank}_{\Q}M_{k,d}
\]
in $R^{2,k}_N$. In particular, each is a basis exactly when
$\det M_{k,d}\ne0$.
\end{lemma}
\begin{proof}
The index calculation in the proof of
Proposition~\ref{counter:even-offset} did not use the parity of $d$.
It gives the indices in~\eqref{counter:general-two-row-candidates}:
the exceptional partition $(k,d)$ has $(k+1)$-core $(k+d,d)$
and $k$-conjugate $(1^N)$; all other two-row partitions have
largest hook at most $k$ and are conjugated by transposition.
The triangular change~\eqref{counter:triangular-change} also has
no parity restriction. Thus both polynomial families are bases of
\[
 V_{k,d}=\Span_{\Q}\{u^N,\ u^{N-2a}v^a:d+1\le a\le t\}.
\]
Put $I=(h_{k+1},h_{k+2})\subseteq\Q[u,v]$. The weight-$N$ part
of $I$ is spanned by
\[
 u^{d-1-2j}v^jh_{k+1}
   \quad(0\le j\le\lfloor(d-1)/2\rfloor),\qquad
 u^{d-2-2j}v^jh_{k+2}
   \quad(0\le j\le\lfloor(d-2)/2\rfloor).
\]
These $d$ vectors are independent. Indeed $h_{k+1},h_{k+2}$
form a regular sequence by the ordinary cohomology presentation,
so their first syzygy has weight $(k+1)+(k+2)=2k+3>N$.
Equivalently, their coprimality in $\Q[u,v]$ forces the coefficient
of $h_{k+1}$ in a syzygy to be divisible by $h_{k+2}$, which is
impossible in weight $d-1\le k-1$ unless that coefficient is zero.
The same then holds for the other coefficient.

Projection onto the coordinates of
$u^{N-2a}v^a$ for $1\le a\le d$ has kernel $V_{k,d}$.
Its restriction to the displayed basis of $I_N$ is precisely
$M_{k,d}$. Hence
\[
 \dim(I_N\cap V_{k,d})=d-\operatorname{rank}M_{k,d}.
\]
Both $V_{k,d}$ and $R^{2,k}_N$ have dimension $t-d+1$:
for the latter, count $(N-a,a)$ with $d\le a\le t$, or subtract
$\dim I_N=d$ from $\dim\Q[u,v]_N=t+1$.
Passing $V_{k,d}$ to the quotient proves the assertion.
\end{proof}

\begin{proposition}[Excesses one and three]\label{counter:odd-small}
For every $k\ge2$, both prescribed families are bases in weight
$k+1$ of $R^{2,k}$. For $k\ge3$, both are bases in weight $k+3$
if and only if $k\not\equiv2\pmod3$. If $k\equiv2\pmod3$,
each has codimension exactly one. Thus
Theorem~\ref{lg:thm:counter-intro} gives all failures in the
excess-three range, not just a subsequence of failures there.
\end{proposition}

\begin{proof}
Lemma~\ref{counter:projection-criterion} reduces the assertions to
matrices of orders one and three. The first is nonzero for every
$k\ge2$. For $k=3r+c\ge3$, $c\in\{0,1,2\}$, the order-three matrix
has full rank for $c=0,1$ and rank two for $c=2$; the coefficient
tables and determinants are given in
Appendix~\ref{counter:appendix:small-excesses}. Thus the failures are
exactly $k=3r+2$, $r\ge1$, the family of
Theorem~\ref{lg:thm:counter-intro}.
\end{proof}

\begin{proposition}[Excess five]
\label{counter:excess-five}
For every $k\ge5$, both prescribed families are bases of
$(R^{2,k})_{k+5}$. For $2\le k<5$, this component and its
weight-$k+5$ index set are empty.
\end{proposition}

\begin{proof}
For $k\ge5$, write $k=3r+c$, $c\in\{0,1,2\}$. The factorized
determinants in Appendix~\ref{counter:appendix:excess-five} are nonzero
in the legal ranges $r\ge2$ for $c=0,1$ and $r\ge1$ for $c=2$.
Lemma~\ref{counter:projection-criterion} applies. For $k<5$, the weight
$k+5$ exceeds $2k$, and weight preservation of $k$-conjugation makes
the index set empty as well.
\end{proof}

\begin{proposition}[The six highest weights in two rows]
\label{counter:near-top}
For $k\ge2$ and $0\le N\le2k$ with $N\ge2k-5$, the weight-$N$ candidates in
each family indexed by $\mathcal P^{2,k,2}$ form a basis of
$(R^{2,k})_N$, except when
$(k,N)=(5,8)$ or $(8,11)$. In either component each span has codimension one.
\end{proposition}

\begin{proof}
Only positive odd $d=N-k$ remains after
Proposition~\ref{counter:even-offset}. The complementary weight
$2k-N\le5$ gives a pairing matrix of order at most three in
\eqref{counter:catalan-pairing}. Appendix~\ref{counter:appendix:near-top}
uses a Pearson recurrence for shifted Catalan moments to show that
the order-two determinant vanishes only at $(k,N)=(5,8)$, and the
order-three determinant only at $(k,N)=(8,11)$. The order-one matrix
is $(C_k)$. In either exceptional case, singularity together with
Proposition~\ref{counter:defect-bound} gives codimension one.
\end{proof}

\begin{remark}[Scope of the two-row classification]
Lemma~\ref{counter:projection-criterion} applies to every
$1\le d\le k$. Proposition~\ref{counter:even-offset} determines the even excesses;
Proposition~\ref{counter:odd-small} treats excesses one and three,
and Proposition~\ref{counter:excess-five} treats excess five. Proposition~\ref{counter:defect-bound}
shows that every remaining failure has codimension one.
Proposition~\ref{counter:near-top} settles the six highest weights.
At top weight $N=2k$, both families consist of $u^{2k}$ and are
bases for every $k\ge2$: its top coefficient is $C_k>0$. Above top weight
the components and index sets are empty. A general classification
of the remaining odd excesses $d\ge7$ outside the near-top range
is not obtained here.
\end{remark}

\appendix
\section{Finite complete intersections}
\label{ca:section}
We record the standard graded facts used in Lemma~\ref{common:ci-comparison}.
Let $S=\Q[x_1,\ldots,x_h]$, with positive weights $w_i$, and put
$W=\sum_iw_i$. The graded dual of a finite-dimensional module $M$ is
$(M^\vee)_t=\operatorname{Hom}_\Q(M_{-t},\Q)$, with
$(x_i\phi)(v)=\phi(x_iv)$.
The Koszul complex on homogeneous $f_i$ of degrees $d_i$ has term
$\bigoplus_{|J|=j}S(-\sum_{i\in J}d_i)$ in homological degree $j$;
its differential is the alternating sum of multiplication by the
$f_i$, deleting the corresponding exterior generator.
Internally graded Hom is the direct sum of spaces of homogeneous
module maps of each degree,
and $\Ext_S^i$ is computed by applying Hom to a graded free resolution.

\begin{lemma}[A finite quotient from a single common zero]\label{ca:finite-zero}
If polynomials in $S$ have only the origin as their common complex zero,
their quotient is finite dimensional over $\Q$.
\end{lemma}
\begin{proof}
After extending to $\C$, the Nullstellensatz puts a power of each variable
in the ideal~\cite[Tag~00FS]{Stacks}. Bounded-exponent monomials span the
quotient, and extension of scalars preserves linear independence.
\end{proof}

Finite-length graded duality gives the natural identifications
\begin{equation}\label{ca:finite-duality}
 \Ext_S^i(M,S(-W))=0\quad(i\ne h),\qquad
 \Ext_S^h(M,S(-W))\simeq M^\vee
\end{equation}
for finite-dimensional graded $M$.
The positive-weight form of local duality in~\cite[pp.~501, 503]{Symonds}
uses the canonical shift $S(-W)$.
Here local cohomology means the right derived functors of the submodule
consisting of elements annihilated by some power of $(x_1,\ldots,x_h)$.
For finite-length $M$, the zeroth local cohomology is $M$ and the
higher local cohomology vanishes, giving~\eqref{ca:finite-duality}.
In particular, top Ext is exact and faithful as a contravariant
functor on these modules.

\begin{lemma}[Finite complete intersections]\label{ca:ci}
Let $f_1,\ldots,f_h\in S$ be homogeneous of positive degrees $d_i$.
If their only common complex zero is the origin, they form a regular
sequence in the stated order, their Koszul complex resolves
$T=S/(f_1,\ldots,f_h)$, and
\begin{equation}\label{ca:hilbert}
 \Hilb(T;q)=\frac{\prod_{i=1}^h(1-q^{d_i})}
                  {\prod_{i=1}^h(1-q^{w_i})}.
\end{equation}
\end{lemma}
\begin{proof}
By Lemma~\ref{ca:finite-zero}, the $f_i$ form a homogeneous system of
parameters, that is, an $h$-tuple generating an ideal of finite colength.
The parameter theorem gives regularity after localization at the
homogeneous maximal ideal~\cite[Tags~00ND and~02JN]{Stacks}.
A homogeneous kernel element cannot vanish after this localization:
a polynomial with nonzero constant term could not annihilate its
lowest homogeneous component. Thus regularity holds already in $S$. The Koszul resolution of a regular
sequence gives~\eqref{ca:hilbert} by its alternating Hilbert series.
\end{proof}
Empty variable and generator lists give $S=T=\Q$, with determinant
and product equal to $1$; the formulas also apply when $h=0$.

\section{Two-row determinant calculations}
\label{counter:appendix}
The proofs here complete the coefficient and top-pairing calculations
used in Section~\ref{lg:sec:counter}. The coefficient and determinant
identities are over $\Q$; they use the recurrence~\eqref{counter:uv-ring}
and the Catalan moments defined in~\eqref{counter:catalan-pairing}.

\subsection{The first member and a local replacement}
\label{counter:appendix:local}
For the smallest member $r=2$ of the family in
Theorem~\ref{lg:thm:counter-intro}, put $x=e_1$ and $y=e_2$.
Then $u=x$, $v=x^2-y$, and
\[
 R^{2,5}=\Q[x,y]/(h_6,h_7),\qquad h_2=x^2-y,
\]
where
\begin{align*}
 h_6&=x^6-5x^4y+6x^2y^2-y^3,\\
 h_7&=x^7-6x^5y+10x^3y^2-4xy^3.
\end{align*}
The following equivalent form of $-P_2$ is an integral polynomial
identity:
\begin{equation}\label{lg:eq:certificate}
 3x^8-14(x^2-y)^4=(14y-24x^2)h_6+13x h_7.
\end{equation}
It verifies the dependence without a rank computation.
Independently, the Pieri rule or Jacobi--Trudi reduction gives
\begin{equation}\label{lg:eq:schur-counter}
 \begin{aligned}
 h_1^8&=28s_{(5,3)}+14s_{(4,4)},\\
 h_2^4&=6s_{(5,3)}+3s_{(4,4)}.
 \end{aligned}
\end{equation}
The two Schur classes on the right form a basis of $(R^{2,5})_8$,
so the candidates are nonzero and span a single line.
Also $y=h_1^2-h_2$, showing directly that $R^{2,5,2}=R^{2,5}$.
The prescribed levels give $s_{(1^8)}^{(1)}=h_1^8$ and
$s_{(2^4)}^{(2)}=h_2^4$.

\begin{remark}\label{lg:rem:local-repair}
In this one degree, replacing $h_2^4$ by $h_2h_1^6$ repairs the basis:
\[
 h_2h_1^6=19s_{(5,3)}+9s_{(4,4)},\qquad
 \det\begin{pmatrix}28&19\\14&9\end{pmatrix}=-14.
\]
Thus $h_1^8,h_2h_1^6$ form a basis of $(R^{2,5})_8$.
Corollary~\ref{counter:uniform-repair} gives a different replacement
for the entire counterexample family. These replacements apply only
to the stated components; they do not give bases for arbitrary
parameters and weights.
\end{remark}

\subsection{Excesses one and three}
\label{counter:appendix:small-excesses}
\begin{proof}[Proof of Proposition~\ref{counter:odd-small}]

The coefficient identity~\eqref{counter:coefficient-extraction}
gives the following table, in which every entry is divided by
$(-1)^r$:
\begin{equation}\label{counter:three-residue-coefficients}
\begin{array}{c|rrrr}
 n & c_{n,0} & c_{n,1} & c_{n,2} & c_{n,3}\\ \hline
 3r   &1&-2r&\frac32r(r-1)&-\frac12r(r-3)(r-1)\\
 3r+1 &1&-r&-r&\frac12r(r-1)(r+2)\\
 3r+2 &0&r+1&-\frac32r(r+1)&r(r-1)(r+1)
\end{array}
\end{equation}
For $n=3r+c$ and $a\le3$, expand the numerator in
\eqref{counter:coefficient-extraction}; a numerator term $z^j$
contributes only when $c-2a-j$ is divisible by three. Put
$b=(c-2a-j)/3$. Its contribution, after division by $(-1)^r$,
is
\[
 \binom{a+1}{j}(-1)^b\binom{r+b+a}{a}.
\]
Summing over $0\le j\le a+1$ gives the table. For the small boundary
values, coefficients with negative series exponent are zero; the
same values follow either from the displayed products defining the
binomial coefficients or from $h_0=1$, $h_1=u$, and the recurrence
in~\eqref{counter:uv-ring}.

For $d=1$, the matrix is $(c_{k+1,1})$. Writing $k=3r+c$,
$c=0,1,2$, its entry divided by $(-1)^r$ is respectively
$-r$, $r+1$, and $2(r+1)$. These do not vanish for $k\ge2$,
including $k=2$ separately with $r=0$.

For $d=3$ write $k=3r+c$, where $r\ge1$ and $c\in\{0,1,2\}$,
and use the column order
$u^2h_{k+1},vh_{k+1},uh_{k+2}$.
Put $A_c=(-1)^rM_{k,3}$. The coefficient table
\eqref{counter:three-residue-coefficients} gives
\[
 A_0=\begin{pmatrix}
 -r&1&r+1\\
 -r&-r&-\frac32r(r+1)\\
 \frac12r(r-1)(r+2)&-r&r(r-1)(r+1)
 \end{pmatrix},
\]
\[
 A_1=\begin{pmatrix}
 r+1&0&2(r+1)\\
 -\frac32r(r+1)&r+1&-\frac32r(r+1)\\
 r(r-1)(r+1)&-\frac32r(r+1)&\frac12r(r-2)(r+1)
 \end{pmatrix},
\]
\[
 A_2=\begin{pmatrix}
 2(r+1)&-1&r+1\\
 -\frac32r(r+1)&2(r+1)&r+1\\
 \frac12r(r-2)(r+1)&-\frac32r(r+1)&-\frac12r(r+1)(r+3)
 \end{pmatrix}.
\]
Expanding the determinants gives
\begin{equation}\label{counter:three-determinants}
 \det A_0=\frac{r^2(r+1)^2(3r+2)}4,\qquad
 \det A_1=\frac{r(r+1)^3(3r+4)}4,\qquad
 \det A_2=0.
\end{equation}
The first two are nonzero for $r\ge1$. The upper left $2$ by $2$
minor of $A_2$ is $(r+1)(5r+8)/2$, so $A_2$ has rank exactly
two. Lemma~\ref{counter:projection-criterion} now proves all the
statements, including the codimension in the exceptional case.
\end{proof}

\subsection{Excess five}
\label{counter:appendix:excess-five}
\begin{proof}[Proof of Proposition~\ref{counter:excess-five}]

Use Lemma~\ref{counter:projection-criterion}, with the column order
\[
 u^4h_{k+1},\quad u^2vh_{k+1},\quad v^2h_{k+1},\quad
 u^3h_{k+2},\quad uvh_{k+2}.
\]
Write $k=3r+c$, $c\in\{0,1,2\}$, and let
$A_c(r)=(-1)^rM_{k,5}$. For $s\in\{1,2,3,4\}$ and $b\ge0$, put
\[
 \eta_{s,b}(r)=
 \sum_{\substack{0\le j\le b+1\\s-2b-j\equiv0\pmod3}}
  \binom{b+1}{j}(-1)^{(s-2b-j)/3}
  \binom{r+b+(s-2b-j)/3}{b},
 \qquad \eta_{s,b}=0\quad(b<0).
\]
For an indeterminate $x$, the polynomial convention is
$\binom{x}{b}=x(x-1)\cdots(x-b+1)/b!$, with $\binom{x}{0}=1$.
The coefficient generating function gives
\[
 (-1)^r c_{3r+s,b}=\eta_{s,b}(r),\qquad
 c_{n,b}=[z^{n-2b}]\frac{(1+z)^{b+1}}{(1+z^3)^{b+1}}.
\]
For $r\ge0$, terms with negative series exponent also vanish in
$\eta_{s,b}$, since their last binomial has nonnegative top entry smaller
than $b$. Consequently, with $f_b=\eta_{c+1,b}(r)$ and
$g_b=\eta_{c+2,b}(r)$,
\[
 A_c(r)=
 \begin{pmatrix}
 f_1&f_0&0&g_1&g_0\\
 f_2&f_1&f_0&g_2&g_1\\
 f_3&f_2&f_1&g_3&g_2\\
 f_4&f_3&f_2&g_4&g_3\\
 f_5&f_4&f_3&g_5&g_4
 \end{pmatrix}.
\]
Its determinants are the following polynomial identities:
\begin{align*}
 \det A_0(r)
 &=-\frac{r^3(r-1)(r+1)^3(r+2)(3r+2)(3r+4)(3r+5)}{5760},\\
 \det A_1(r)
 &=\frac{r^2(r-1)(r+1)^4(r+2)^2(3r+4)(3r+5)}{1920},\\
 \det A_2(r)
 &=\frac{r^2(r+1)^3(r+2)^3(3r+4)(3r+5)(3r+7)}{2880}.
\end{align*}
These identities follow by substitution in the displayed matrix and
determinant expansion. Appendix~\ref{comp:excess-five-certificate}
gives an exact interpolation certificate.
The legal ranges for $k\ge5$ are $r\ge2$ for $c=0,1$, and $r\ge1$
for $c=2$, so none of the factors vanish. In particular, the smallest
widths $k=5,6,7$ have normalized determinants $42,-132,297$,
respectively. Since $\det M_{k,5}=(-1)^r\det A_c(r)$, the coefficient
criterion proves the assertion. For $2\le k<5$, the weight $k+5$
is above $2k$; weight preservation of $k$-conjugation also makes
the index set empty.
\end{proof}

\subsection{An interpolation certificate for excess five}
\label{comp:excess-five-certificate}
In the notation of Proposition~\ref{counter:excess-five}, let $P_c(r)$
denote the displayed product for $\det A_c(r)$, and set
$E_c(r)=\det A_c(r)-P_c(r)$. All degrees below are taken in $r$.
The coefficient formula gives $\deg\eta_{s,b}\le b$.
Each nonzero determinant term has degree at most
\[
 (1+2+3+4+5)-(0+1+2+0+1)=11,
\]
since the five column shifts are $0,1,2,0,1$.
Thus $\deg E_c\le11$. Substitution in the explicit coefficient
formula gives the residual certificate
\[
 \bigl(E_c(2),E_c(3),\ldots,E_c(13)\bigr)
       =(0,0,0,0,0,0,0,0,0,0,0,0)\qquad(c=0,1,2).
\]
A polynomial of degree at most eleven with these twelve distinct
zeros is identically zero. This proves the determinant identities
for every $r$. The supplement records the integer values of both
sides and reconstructs each difference polynomial from them, using
exact recurrence coefficients and fraction-free determinants. It also checks
$\det A_2(1)=42$, the smallest legal width $k=5$.

\subsection{The six highest weights}
\label{counter:appendix:near-top}
\begin{proof}[Proof of Proposition~\ref{counter:near-top}]

Put $d=N-k$. Proposition~\ref{counter:even-offset} covers $N\le k$
and even $d$, so suppose $d\ge1$ is odd. Set
\[
 p=\lfloor(k-d)/2\rfloor,\qquad \sigma=k-d-2p\in\{0,1\}.
\]
Since $k-d=2k-N\le5$, we have $p\le2$. For $p=0$, the sole
candidate $u^N$ pairs with $u^{2k-N}$ to $C_k>0$.

For $p\ge1$, put $m=d+1$, an even integer at least two.
Use the Catalan measure $\mu$ from the proof of
Proposition~\ref{counter:even-offset}. Put $y=x-1$ and, for $j\ge0$, define
\[
 \mathcal R_j^{(\sigma)}
 =\int_0^4x^\sigma(x-1)^j\,d\mu(x)
 =\int_{-1}^3y^jw_\sigma(y)\,dy,\qquad
 w_\sigma(y)=\frac{(1+y)^\sigma}{2\pi}
                         \sqrt{\frac{3-y}{1+y}}.
\]
For $A(y)=(1+y)(3-y)$, differentiation with respect to $y$ gives
\[
 (Aw_\sigma)'=(3\sigma-(2+\sigma)y)w_\sigma.
\]
The function $Aw_\sigma$ vanishes at $-1,3$, and its derivative
is integrable there. Integrating $(y^jAw_\sigma)'$ therefore gives
\begin{equation}\label{counter:near-top-pearson}
 (j+2+\sigma)\mathcal R_{j+1}^{(\sigma)}
 =(2j+3\sigma)\mathcal R_j^{(\sigma)}
                       +3j\mathcal R_{j-1}^{(\sigma)}\qquad(j\ge1).
\end{equation}
The moments $\mathcal R_j^{(\sigma)}$ are integers, by binomial
expansion and the Catalan moment formula.

For $p=1$, $k=m+1+\sigma$ and $N=2m+\sigma$.
By~\eqref{counter:catalan-pairing}, the candidates
$u^N,u^\sigma v^m$ pair with $u^{2+\sigma},u^\sigma v$ through
\[
 K=\begin{pmatrix}C_k&C_k-C_{k-1}\\
                   \gamma_k(m)&\gamma_k(m+1)\end{pmatrix}.
\]
Thus the lower row of $K$ is
$(\mathcal R_m^{(\sigma)}+\mathcal R_{m+1}^{(\sigma)},
  \mathcal R_{m+1}^{(\sigma)})$.
Equation~\eqref{counter:near-top-pearson} and
$C_{j+1}/C_j=4-6/(j+2)$ give
\[
 \det K=\frac{mC_{m+\sigma}}{m+2+\sigma}S_\sigma(m),\qquad
 S_\sigma(m)=3\mathcal R_{m-1}^{(\sigma)}-\mathcal R_m^{(\sigma)}.
\]
The prefactor is positive. Put
\[
 B_\sigma(y)=(3-y)w_\sigma(y)-(3+y)w_\sigma(-y)
                    \qquad(0<y<1).
\]
Because $m-1$ is odd, folding the negative interval gives
\begin{equation}\label{counter:near-top-folding}
 S_\sigma(m)=\int_0^1y^{m-1}B_\sigma(y)\,dy
       +\int_1^3y^{m-1}(3-y)w_\sigma(y)\,dy.
\end{equation}
For $\sigma=1$, the positive terms in $2\pi B_1(y)$ have squared difference
\[
 (1+y)(3-y)^3-(1-y)(3+y)^3=16y^3>0.
\]
Hence $B_1>0$ and $S_1(m)>0$ for every even $m\ge2$.
For $\sigma=0$,
\[
 2\pi B_0(y)=\frac{(3-y)^{3/2}}{\sqrt{1+y}}
             -\frac{(3+y)^{3/2}}{\sqrt{1-y}}<0.
\]
The Catalan moments give $S_0(2)=-1$ and $S_0(4)=0$.
For even $m\ge6$, subtract the $m=4$ instance of
\eqref{counter:near-top-folding}. The two integrands are
\[
 y^3(y^{m-4}-1)B_0(y)\quad(0<y<1),\qquad
 y^3(y^{m-4}-1)(3-y)w_0(y)\quad(1<y<3),
\]
and both are positive. Thus $S_0(m)>0$.

The matrix is singular only for $\sigma=0,m=4$, namely $(k,N)=(5,8)$.
Its upper left entry is nonzero, so its rank there is one.

For $p=2$, we have $k=m+3+\sigma$ and $N=2m+2+\sigma$. The candidates correspond to
row indices $a=0,m,m+1$, and the complementary monomial basis to
column indices $j=0,1,2$. Thus their pairing matrix is
\[
 K=\bigl(\gamma_k(a+j)\bigr)_{a\in\{0,m,m+1\},\ 0\le j\le2}.
\]
Write $T_j=\mathcal R_j^{(\sigma)}$ and $s=m+1+\sigma$. Since
$z=1-1/x$, we have the pointwise identity
\[
 (1,z,z^2)=x^{-2}\bigl(x^2,x(x-1),(x-1)^2\bigr).
\]
Absorbing $x^{-2}$ into each row functional therefore replaces the
column polynomials $1,z,z^2$ by
\[
 x^2=(1+y)^2,\qquad x(x-1)=y(1+y),\qquad (x-1)^2=y^2.
\]
Their coefficient matrix in $1,y,y^2$ is triangular with diagonal entries
one, so passing to that column basis preserves the determinant.
Subtracting the last row from the middle row also has
determinant one and gives
\begin{equation}\label{counter:near-top-p2-transform}
 \det K=\det\begin{pmatrix}
 C_s&C_{s+1}-C_s&C_{s+2}-2C_{s+1}+C_s\\
 T_m&T_{m+1}&T_{m+2}\\
 T_{m+1}&T_{m+2}&T_{m+3}
 \end{pmatrix}.
\end{equation}
Indeed the two original tail functionals have respectively the weights
$x(x-1)^m x^\sigma d\mu$ and $(x-1)^{m+1}x^\sigma d\mu$;
their difference has weight $(x-1)^m x^\sigma d\mu$.
The top functional has weight $x^{m+1}x^\sigma d\mu$.

Since $m$ is even, $T_m>0$. Put $t=T_{m-1}/T_m\in\Q$ and
$r_j=T_{m+j}/T_m$ for $j\ge-1$, so $r_0=1$ and $r_{-1}=t$.
The Pearson recurrence~\eqref{counter:near-top-pearson} gives
\[
 r_{j+1}=\frac{(2m+2j+3\sigma)r_j+3(m+j)r_{j-1}}
                  {m+j+2+\sigma}\qquad(j=0,1,2).
\]
Also put
\[
 \kappa_1=\frac{C_{s+1}}{C_s}=4-\frac6{s+2},\qquad
 \kappa_2=\frac{C_{s+2}}{C_{s+1}}=4-\frac6{s+3}.
\]
Dividing~\eqref{counter:near-top-p2-transform} by $C_sT_m^2$ therefore gives the explicit determinant
\[
 \frac{\det K}{C_sT_m^2}
 =\det\begin{pmatrix}
 1&\kappa_1-1&\kappa_1\kappa_2-2\kappa_1+1\\
 1&r_1&r_2\\r_1&r_2&r_3
 \end{pmatrix}
 =\begin{cases}
 \displaystyle\frac{2m(m+1)Q_0(t)}{(m+2)^2(m+3)^2(m+4)},&\sigma=0,\\[5pt]
 \displaystyle\frac{3m(m+1)Q_1(t)}{(m+3)^2(m+4)^2(m+5)},&\sigma=1,
 \end{cases}
\]
where
\[
 \begin{aligned}
 Q_0(t)&=9m(3m+7)t^2+3(6m^2+13m-6)t-(9m^2+38m+48),\\
 Q_1(t)&=9m(2m+7)t^2+6(2m^2+9m+6)t-(6m^2+37m+60).
 \end{aligned}
\]
These identities follow by substituting the three displayed recurrences
for $r_1,r_2,r_3$ in the $3$ by $3$ determinant; all denominators and
all factors outside $Q_\sigma$ are nonzero. Both $Q_\sigma$ are genuine
quadratics in $t$. The discriminant of $at^2+bt+c$ is $b^2-4ac$; here it gives
\[
 \operatorname{disc}_t Q_0=81(m+2)^2(16m^2+32m+1),\qquad
 \operatorname{disc}_t Q_1=144(m+3)^2(4m^2+14m+1).
\]
A quadratic over $\Q$ with nonzero leading coefficient can have a
rational root only if its discriminant is a rational square. An integer
that is a rational square is an integer square. For every $m\ge2$,
\[
 (4m+3)^2<16m^2+32m+1<(4m+4)^2,
\]
so $Q_0$ cannot vanish at the rational number $t$. For $Q_1$, the
remaining discriminant factor is $45$ when $m=2$, $121$ when $m=4$,
and for $m\ge6$ satisfies
\[
 (2m+3)^2<4m^2+14m+1<(2m+4)^2.
\]
Thus the only possible zero for even $m$ is $m=4$, $\sigma=1$.
Here $t=\mathcal R_3^{(1)}/\mathcal R_4^{(1)}=4/9$ and
\[
 Q_1(t)=4(9t-4)(15t+19)=0.
\]
This gives $k=8$, $d=3$, and $N=11$.
The lower-left $2$ by $2$ block in~\eqref{counter:near-top-p2-transform} is a positive Gram matrix
for $1,y$ with respect to the positive measure $y^m w_\sigma(y)\,dy$,
so it is nonsingular. The full matrix therefore has rank exactly two
in the exceptional case, giving codimension one.

Poincar\'e duality now gives the assertion in all three cases.
\end{proof}

\section{Exact computations}\label{comp:section}
The computational supplement is supplied as three UTF-8 text files:
\begin{enumerate}[label=(\arabic*),leftmargin=*]
\item \path{ESM_1_README_RESTORE.txt}: documentation and a Python
restoration utility;
\item \path{ESM_2_SOURCE_CODE.txt}: all source programs;
\item \path{ESM_3_DATA_RECORDS.txt}: verification data, execution logs,
and the original documentation and manifests.
\end{enumerate}
These files preserve all 174 files of the original computational package,
including their relative paths and byte contents. The restoration utility
checks the recorded SHA-256 hashes and reconstructs the original directory
without running the computations. The recovered \texttt{README.md} lists
the programs, parameter ranges, dependencies, and reproduction commands.
The execution records include source hashes. All ranks
and polynomial identities are computed over $\Q$ or by exact integer
arithmetic. The programs construct the actual subalgebras and ideals
from their generators before comparing dimensions with the formulas.
They provide finite checks of the preceding proofs.

\subsection{Quotients, images, and derivatives}
Ordinary multiplication is implemented in the rectangular Schur basis
by the Pieri rule. For $(k,\ell)=(4,4),(4,5),(5,5),(5,6)$, the programs
compute every filtered subalgebra and the ideals generated by the actual
Jacobian-tail minors, including their multiplication closure. A separate
polynomial calculation checks the finite parameter maps and
resultant--colon identities. The two ordinary Pieri programs share
one multiplication implementation.

For $1\le n\le10$, quadratic squarefree reduction computes the
Lagrangian subalgebras and the full transformed-$Q$ ideal flag.
Every potentially nonzero strict partition through the top weight is
included among the ideal generators. Parts greater than $2n$ give a
zero Pfaffian row after the transform, since $\pi_n(\Theta\qb_r)=[z^r]E(z)^2=0$
for $r>2n$; only this zero bound is used in the larger calculation.
The retained-basis bound $\lambda_1\le n$ is not imposed on the list. An independent implementation for $n\le5$
checks $195$ flag components, $155$ subalgebra components and $90$
derivative components, using all $201$ strict-partition candidates.
The source weight for a derivative of weight $t$ is $t+m$, also when
it exceeds the top weight of the ambient ring.

Gr\"obner quotient calculations check the Lagrangian coefficient
images, their defining windows and colon quotients for
$1\le d\le4$, $1\le M\le5$. Further symbolic calculations construct
the two transition matrices for $d\le5$, checking their determinants
with $M$ indeterminate; the generator identities are checked before
quotienting for $d\le4$, $M\le6$. To compute a finite conductor,
the programs impose $uv\in\mathcal B$ for every ambient standard
monomial $v$, rather than only for the coefficient generators.
These checks include the failure of the ordinary conductor formula at $m=0$ and the
Lagrangian terminal exception. They also verify the two coefficient
algebra examples, including the isomorphism over $\Q(\sqrt3)$.

\subsection{Two-row families}
Exact Pieri coordinates compare both candidate spans in every weight
for $2\le k\le30$. Recurrence-generated relation coefficients give a
separate check of Lemma~\ref{counter:projection-criterion} for all
$2\le k\le32$, $1\le d\le k$, comprising $527$ parameter pairs.
The three residue-class determinants and the rank-two minor in
Proposition~\ref{counter:odd-small} are also checked symbolically.
The Catalan program checks $117$ top-pairing entries and $62$
positive determinants, including $13$ one-dimensional cases.
It verifies the rectangular Schur identity on $60$ samples with
zero, negative or repeated coordinates. A separate integer-arithmetic
program certifies the excess-five determinants by exact interpolation,
as detailed in Appendix~\ref{comp:excess-five-certificate}. The supplement computes the
defining-ideal ranks in all $113$ nonnegative near-top cases for
$2\le k\le20$ and verifies the two $3$ by $3$ determinant
identities in Proposition~\ref{counter:near-top} symbolically.

\subsection{Intersection pairings}\label{comp:poincare}
Here $S$ denotes the tautological subbundle: its fibre at a plane is
that plane. We use $e_i=c_i(S^*)$. The integral of a top cohomology
class means its evaluation on the fundamental class. For a finite list
$\alpha$ of positive integers, put $|\alpha|=\sum_i\alpha_i$ and
$h_\alpha=\prod_i h_{\alpha_i}$. The Kostka number is
$K_{\lambda,\alpha}=\langle s_\lambda,h_\alpha\rangle$.

\begin{lemma}[Classical intersection numbers]\label{comp:intersection}
For $|\alpha|=k\ell$ and $|\beta|=n(n+1)/2$, respectively,
\begin{equation}\label{comp:kostka-integrals}
 \int_{\Gr(k,k+\ell)}\prod_i e_{\alpha_i}=K_{(k^\ell),\alpha},
 \qquad
 \int_{\LG(n,\C^{2n})}\prod_i e_{\beta_i}=K_{\delta_n,\beta},
 \quad\delta_n=(n,n-1,\ldots,1).
\end{equation}
\end{lemma}
\begin{proof}
The Hall involution $\omega$ interchanges $e_i$ and $h_i$ and sends
$s_\lambda$ to $s_{\lambda'}$. Integration on the ordinary Grassmannian
extracts the coefficient of $s_{(\ell^k)}$, giving the first identity
\cite{Fulton,Macdonald}.

The symplectic form defines a section of $\bigwedge^2S^*$ on
$\Gr(n,2n)$ whose zero locus is $\LG(n,\C^{2n})$. This locus is smooth
of codimension $\binom n2$, the rank of the bundle, so the section is
regular. Its fundamental class is
\[
 c_{\binom n2}(\bigwedge^2S^*)
 =\prod_{i<j}(x_i+x_j)=s_{(n-1,n-2,\ldots,1)}(x),
\]
where the $x_i$ are formal Chern roots of $S^*$: their elementary
symmetric polynomials are the classes $e_i$. This staircase identity
and the relation between the two integrals are classical
\cite{Macdonald,Pragacz88,HT}.
The rectangular complement of $(n-1,n-2,\ldots,1,0)$ in $(n^n)$,
rotated by $180$ degrees, is the self-transpose partition $\delta_n$.
Schubert duality thus reduces the second integral to the
$s_{\delta_n}$ coefficient. Applying $\omega$ proves the formula.
\end{proof}

This gives a check using a different construction of the multiplication
pairing. In weight $t$, pair monomials in the allowed generators with
all monomials of complementary weight. Poincar\'e duality identifies
the matrix rank with the dimension of the generated subspace, because
the complementary monomials span the full component.
The supplement verifies every top-weight elementary monomial for
$1\le n\le6$ and for the ordinary pairs
\[
 (1,4),(2,2),(2,5),(3,3),(3,4),(4,4),(4,5).
\]
Direct reduction by the defining ring relations agrees with an
independently implemented integer Pieri recursion on all $661$ integrals;
a third hook-length calculation checks the $13$ values of $\int e_1^D$,
where $D$ is the top weight. These are intersection-value checks,
rather than a claim that all filtration ranks were recomputed by this method.

\section{A coefficient-algebra comparison}
\label{lg:appendix:coefficients}
We verify the assertions in
Example~\ref{lg:ex:non-Grassmannian-coefficients}.
For $M=d=2$, equation~\eqref{lg:eq:D-U} gives
\[
 U_3=\tfrac32x^3-4xy,\qquad
 U_4=-\tfrac32x^4+6x^2y-4y^2.
\]
The ideal $(U_3,U_4)$ is
$(3x^3-8xy,\ x^2y-2y^2)$, giving the stated presentation of $C$.

In either ring $A\in\{C,H\}$ of the example, the relations give the basis
$1,x,y,x^2,xy,y^2$, with $xy^2=y^3=0$. Its positive-degree ideal
$\mathfrak m=A_{>0}$ is the unique maximal ideal. For an ideal $J$,
write $\Ann_A(J)=\{a\in A:aJ=0\}$; the \emph{socle} is
$\operatorname{Soc}(A)=\Ann_A(\mathfrak m)$. One computes
\[
 \mathfrak m^3=\Span_\Q\{xy,y^2\},\qquad
 \operatorname{Soc}(A)=\Q y^2,\qquad
 I:=\Ann_A(\mathfrak m^3)=\Span_\Q\{y,x^2,xy,y^2\}.
\]
Since $I^2\subseteq\operatorname{Soc}(A)$ and $I\mathfrak m^3=0$,
multiplication induces an intrinsic symmetric pairing
$(I/\mathfrak m^3)^{\otimes2}\to\operatorname{Soc}(A)$.
In the basis $(\overline y,\overline{x^2})$, using $y^2$ for the
socle generator, its matrices for $C$ and $H$ are, respectively,
\[
 \begin{pmatrix}1&2\\2&16/3\end{pmatrix},\qquad
 \begin{pmatrix}1&1\\1&2\end{pmatrix},
 \qquad\text{with determinants }4/3\text{ and }1.
\]
An algebra isomorphism preserves $\mathfrak m$, its powers, their
annihilators, and the socle. A basis change in the two-dimensional
source and a rescaling of the one-dimensional target alter the
pairing determinant only by a rational square. Since $4/3$ is not
a square in $\Q$, such an isomorphism cannot exist. This obstruction
concerns the coefficient algebras themselves, independently of their
grading and specified ambient-ring actions.
The restriction to $\Q$ is essential. Over $\Q(\sqrt3)$, the assignments
\[
 x\longmapsto x,\qquad
 y\longmapsto\frac{3+\sqrt3}{8}x^2-\frac{\sqrt3}{4}y
\]
define a graded isomorphism
$C\otimes_\Q\Q(\sqrt3)\to H\otimes_\Q\Q(\sqrt3)$:
substitution verifies the two
relations of $C$, and the nonzero coefficient of $y$ makes the map
surjective; both rings have dimension six.

\section*{Declaration of generative AI and AI-assisted technologies
in the manuscript preparation process}

During the preparation of this work, the authors used AI-assisted tool Codex to assist with language editing, manuscript organization, and exploration of some of the proof strategies and calculations. The authors reviewed and edited the resulting material and take full responsibility for the contents of the manuscript.

\end{document}